\documentclass[sn-mathphys-num]{sn-jnl}

\usepackage{graphicx}%
\usepackage{multirow}%
\usepackage{amsmath,amssymb,amsfonts}%
\usepackage{amsthm}%
\usepackage{mathrsfs}%
\usepackage[title]{appendix}%
\usepackage{xcolor}%
\usepackage{textcomp}%
\usepackage{manyfoot}%
\usepackage{booktabs}%
\usepackage{algorithm}%
\usepackage{algorithmicx}%
\usepackage{algpseudocode}%
\usepackage{listings}%

\theoremstyle{thmstyleone}%
\newtheorem{theorem}{Theorem}
\theoremstyle{thmstyletwo}%
\newtheorem{remark}{Remark}%
\newtheorem{lemma}{Lemma}%

\theoremstyle{thmstylethree}%
\newtheorem{definition}{Definition}%

\begin{document}

\title[Large Deviation Principle for Slow-Fast Systems with Infinite-Dimensional Mixed Fractional Brownian Motion]{Large Deviation Principle for Slow-Fast Systems with Infinite-Dimensional Mixed Fractional Brownian Motion}

\author[1,2]{ \sur{Wenting Xu}}\email{qfnuxwt@163.com}

\author*[1,2]{ \sur{Yong Xu}}\email{hsux3@nwpu.edu.cn}

\author[3]{ \sur{Xiaoyu Yang}}\email{yangxiaoyu@yahoo.com}

\author[1,2,4]{ \sur{Bin Pei}}\email{binpei@nwpu.edu.cn}

\affil*[1]{\orgdiv{School of Mathematics and Statistics}, \orgname{Northwestern Polytechnical University}, \orgaddress{ \city{Xi'an}, \postcode{710072}, \country{China}}}

\affil*[2]{\orgname{MOE Key Laboratory of Complexity Science in Aerospace}, \orgaddress{\city{Xi'an}, \postcode{710072}, \country{China}}}

\affil[3]{\orgdiv{Graduate School of Information Science and Technology}, \orgname{Osaka University}, \orgaddress{  \city{Osaka}, \postcode{5650871}, \country{Japan}}}

\affil[4]{  \orgname{Research \& Development institute of Northwestern Polytechnical University in Shenzhen}, \orgaddress{ \city{Shenzhen}, \postcode{518057},  \country{China}}}


\abstract{This work is concerned with the large deviation principle for a family of slow-fast systems perturbed by infinite-dimensional mixed  fractional Brownian motion  with Hurst
	parameter $H\in(\frac12,1)$. We adopt the weak convergence method which is
	based on the variational representation formula for infinite-dimensional mixed fractional Brownian motion.
	To obtain the weak convergence of the controlled systems, we apply the Khasminskii's averaging principle and the time discretization technique.
	In addition, we drop the boundedness assumption of the drift coefficients of the slow components and the diffusion coefficients of the fast components.
	Based on the proof of the large deviation principle, we also establish the moderate deviation principle for the slow-fast systems.}




\maketitle

\section{Introduction}\label{1}

Many complex phenomena in fields such as materials science, chemistry, fluid dynamics, biology, ecology and climate dynamics can be modeled by slow-fast systems constituted by various types of differential equations \cite{bertram2017multi, harvey2011multiple,mastny2007two,kuehn2015multiple}. One typical example is climate dynamics \cite{berner2017stochastic}. The complexity of climate systems variability arises from several factors: the nature of external forcings, the inhomogeneities in the physical and chemical properties of its components, and the wide range of dynamic processes within each component, as well as the coupling mechanisms between them. This variability spans over ten orders of magnitude in space, from Kolmogorov's dissipation scale to the Earth's radius, and exhibits even greater variations in time, ranging from microseconds to hundreds of millions of years \cite{ghil2002natural}.
Another example concerns a class of Hamiltonian systems, where the slow manifold loses its normal hyperbolicity due to a transcritical or pitchfork bifurcation as the slow variable evolves \cite{schecter2010heteroclinic}. In these systems, the heteroclinic orbits connecting specific equilibria represent a dynamic interface between ordered and disordered states.


In real-world physical systems, stochastic perturbations, multiscale phenomena, and evolutionary processes often coexist, among which, the evolution partial differential equations were used, for example to describe a free (boson) field in relativistic quantum mechanics \cite{grandy2012relativistic}, a hydromagnetic dynamo-process in cosmology \cite{nandy2010dynamo}, diffraction in random-heterogeneous media in statistical physics \cite{torquato1991random}, etc. Consequently, many problems in the natural sciences lead to slow-fast stochastic evolution partial differential equations.
This paper 
investigates the following slow-fast systems driven by Brownian motion (BM) and fractional Brownian motion (FBM) in $V\times V$
\begin{equation}\label{eqn-1.2}
	\begin{cases}
		\mathrm{d}X^{\varepsilon,\delta}_t=(AX^{\varepsilon,\delta}_t+b(X^{\varepsilon,\delta}_t,Y^{\varepsilon,\delta}_t))\mathrm{d}t+\sqrt\varepsilon g(X^{\varepsilon,\delta}_t)\mathrm{d}B^H_t,\\  
		\mathrm{d}Y^{\varepsilon,\delta}_t=\frac1\delta(AY^{\varepsilon,\delta}_t+F(X^{\varepsilon,\delta}_t,Y^{\varepsilon,\delta}_t))\mathrm{d}t+\frac{1}{\sqrt\delta} G(X^{\varepsilon,\delta}_t,Y^{\varepsilon,\delta}_t)\mathrm{d}W_t,\\
		X^{\varepsilon,\delta}_0=X_0,~ Y^{\varepsilon,\delta}_0=Y_0,~t\in[0,T],
	\end{cases}
\end{equation}	
where $V$ is a separable Hilbert space, $A$ is the infinitesimal generator of an analytic semigroup on $V$. Precise conditions on these coefficient functions will be specified later. The parameter $0<\delta\ll1$ characterizes the ratio of time-scales between processes $X^{\varepsilon,\delta}$ and $Y^{\varepsilon,\delta}$. The driving process $(B^H_t)_{t\in[0,T]}$ is a cylindrical FBM ($V$-valued FBM) with Hurst parameter $H\in(\frac12,1)$ and $(W_t)_{t\in[0,T]}$ is a $V$-valued BM.


Direct analysis of slow-fast systems is challenging due to the disparate time scales and the intricate coupling between variables.
The averaging principle (AP) provides a powerful tool for simplifying slow-fast systems with capturing the effective dynamics.
A key idea is that the rapidly varying process can be approximated as noise with an invariant measure. Utilizing this measure allows for an asymptotic analysis, which shows that the slow process converges to a limit corresponding to the average with respect to the stationary measure of the fast process.
Applying the AP yields the characteristics of the limiting slow dynamics. 
Xu and Duan et al. \cite{xu2011averaging} established the AP for SDEs with
non-Gaussian L\'{e}vy noise.
The AP for fast-slow systems driven by mixed fractional Brownian rough path is investigated \cite{pei2021averaging}.
Subsequently, Pei and Schmalfuss et al. \cite{pei2024almost} studied the slow-fast systems consisting of evolution equations, which differ from other works in that the fast component is also driven by FBM.
Recently, the AP for a class of semilinear slow-fast PDEs driven by finite-dimensional rough multiplicative noise is investigated  \cite{liM2025averaging}.

Based on the analytical methods for infinite-dimensional FBM \cite{Anh1999,Tindel2003,Maslowski2003,Pei2020} and the framework of the AP, under appropriate conditions, it follows that slow-fast system \eqref{eqn-1.2} can be described by an averaged equation
\begin{align}\label{eq-2}
	\mathrm{d}\bar{X}_t=\big(A\bar{X}_t+\bar{b}(\bar{X}_t)\big)\mathrm{d}t,\quad  
	\bar{X}_0=X_0,\quad t\in[0,T],
\end{align}
where $\bar{b}(x)=\int_Vb(x,z)\mu^x(dz),~x\in V,$
and $\mu^x$ is the unique invariant measure of the transition semigroups for the following frozen equation
\begin{equation}
	\mathrm{d}Y^x_t=(AY^x_t+F(x,Y^x_t))\mathrm{d}t+G(x,Y^x_t)\mathrm{d}W_t,\quad Y^x_0=Y_0,\quad t\in[0,T].\label{eqn-1.3}\notag
\end{equation}
The averaged process $\bar{X}$ above is valid only in the limiting sense, but the slow process $X^{\varepsilon,\delta}$ will experience fluctuations around the corresponding averaged process
$\bar{X}$ with small parameter $\varepsilon$. In order to capture the fluctuations, it is important to explore the deviation between $X^{\varepsilon,\delta}$ and $\bar{X}$.

Compared with the AP, the large deviation principle (LDP) focuses on the exact asymptotic behavior of the probability decay rate of rare events. It is one of the most active research fields in probability, and has a wide range of applications in the fields of statistical inference, partial differential equations and statistical mechanics et al. \cite{Dembo1993,Deuschel1989,Dupuis1997,Varadhan2010,Touchette2009}. 
Similarly, the moderate deviation principle (MDP) investigates the rate of convergence of probabilities, with a particular focus on deviations of smaller orders \cite{morse2017moderate,yang2024moderate,gasteratos2023moderate,bourguin2024moderate}.
To study the fluctuations of the slow variable $X^{\varepsilon,\delta}$ of the slow-fast system \eqref{eqn-1.2} around the solution $\bar{X}$ to the deterministic averaged equation \eqref{eq-2}, we define the deviation component between $X^{\varepsilon,\delta}$ and $\bar{X}$ as follows
\begin{align}
	Z^{\varepsilon,\delta}_t=\frac{X^{\varepsilon,\delta}_t-\bar{X}_t}{\sqrt{\varepsilon}h(\varepsilon)},\quad Z^{\varepsilon,\delta}_0=0,\quad t\in[0,T].\label{eqn-6.1}
\end{align}
Here, $h : (0, 1] \to  (0, \infty)$ is continuous, and satisfies that $\lim_{\varepsilon \to 0}h(\varepsilon)=\infty$ and $\lim_{\varepsilon \to 0}\sqrt{\varepsilon}h(\varepsilon)=0$ for all $\varepsilon\in(0,1]$.
When $h(\varepsilon)=\frac{1}{\sqrt\varepsilon}$, the deviation component corresponds to the LDP, which characterizes the exponential decay rate of the probability of events far from the averaged system. When $\lim_{\varepsilon \to 0}\sqrt\varepsilon h(\varepsilon)=0$, the deviation component corresponds to the MDP. Specifically, when $h(\varepsilon)\equiv1$, the deviation component corresponds to the central limit theorem (CLT), which characterizes the concentration of probability near the averaged system.


The LDP for systems was initially formulated by Freidlin and Wentzell \cite{Wentzell1984}, addressing the asymptotic probabilities of stochastic dynamical systems influenced by small noise over long times. 
Bou\'{e} and Dupuis \cite{Dupuis1998} applied weak convergence theory to establish the LDP under disturbance conditions. 
Initially introduced for stochastic differential equations (SDEs) driven by finite-dimensional BM, this method was later extended to infinite-dimensional BM in stochastic dynamical systems by Budhiraja and Dupuis \cite{Budhiraja2000, Budhiraja2008}. 
The weak convergence method has since been applied to various problems \cite{Budhiraja2019,Hu2019,Spiliopoulos2013,Wang2012}. 
Notably, Sun \cite{Sun2021} demonstrated the LDP for the slow-fast stochastic Burgers equation driven by infinite-dimensional BM.



Unlike standard BM, the self-similar and long-range dependence properties of FBM make it a suitable candidate for modeling randomness in certain complex systems \cite{Biagini2008}.
The Hurst parameter $H$ of FBM characterizes the roughness of the motion \cite{Mishura2008}.
To date, only a few studies have investigated the LDP for finite-dimensional FBM.
Based on the variational representation for	random functionals on abstract Wiener spaces \cite{Zhang2009}, the weak convergence method for finite-dimensional FBM was proposed by
Budhiraja and Song \cite{Budhiraja2020}. Inahama and Xu et al. \cite{Inahama2023} established the LDP for slow-fast systems with mixed finite-dimensional FBM. Shen et al. \cite{Shen2024} developed the LDP for the multi-scale distribution dependent SDEs driven by additive fractional white noise. 
Recently, a LDP for the slow-fast systems under the controlled rough path framework is constructed  \cite{yang2025large}. Additionally,
a MDP for rough differential equations driven by scaled
fractional Brownian rough path
is established \cite{inahama2024moderate}.


A natural extension of these problems is the study of stochastic partial differential equations (SPDEs)	driven by infinite-dimensional FBM.
For a detailed study on the well-posedness of systems driven by infinite-dimensional FBM, the reader is referred to \cite{Anh1999,Tindel2003,Maslowski2003,Pei2020}.
This paper aims to establish the LDP and MDP for slow-fast systems driven by infinite-dimensional mixed FBM. The LDP analysis in this case is more challenging than in the finite-dimensional setting, as certain key properties, such as the Arzela-Ascoli's theorem, no longer apply.
First, we present the variational representation formula for infinite-dimensional mixed FBM, based on that for infinite-dimensional BM \cite{Budhiraja2000,Budhiraja2008}. Subsequently, the focus shifts to the weak convergence of the controlled system.
Due to the absence of martingale properties in slow equations driven by FBM, classical martingale inequalities are no longer applicable. 
Furthermore, for the slow-fast systems, it is natural to investigate the influence of multi-scale effects on the system \eqref{eqn-1.2} as both $\varepsilon$ and $\delta$ tend to zero.  	
We combine the estimates of pathwise integrals using generalized Riemann-Stieltjes integral with appropriate semigroup estimates. Specifically, we apply the classical Khasminskii's time discretization approach to study the convergence of the controlled slow processes. Moreover, in the proof, due to the weaker assumptions of coefficients, we employ properties of generalized Riemann-Stieltjes integral to establish the convergence of the controlled systems. This approach bypasses the difficulties that arise from being unable to use the Arzela-Ascoli's theorem.

A brief outline of the paper is as follows. In Section \ref{2}, we introduce the notations used throughout the paper and present the variational representation for infinite-dimensional mixed FBM. Section \ref{3} outlines the necessary assumptions and provides a precise statement of the LDP for slow-fast systems driven by infinite-dimensional mixed FBM \textbf{(Theorem \ref{th3.1})}. In Section \ref{4}, we prove the preliminary lemmas. Section \ref{5} is dedicated to proving the main results. 
In Section \ref{6}, based on the LDP, We also extend to the MDP for the above slow-fast system \textbf{(Theorem \ref{th6.1})}.
In Section \ref{7}, we review the literature to discuss the feasibility and specific challenges involved in proving the LDP for fully coupled systems.
Appendix \ref{A} introduces the criterion for the LDP, while Appendix \ref{B} provides technical proofs.

\section{Notation and Preliminaries}\label{2}

Let $V=(V,\|\cdot\|,\langle~\cdot~\rangle)$ be a separable Hilbert space.
We denote the spaces of linear bounded and trace-class operators on $V$ by $L(V)$ and $L_1(V)$, respectively.
Let $L_2(V)$ denote the space of Hilbert-Schmidt operators from $V$ to itself, endowed
with the Hilbert-Schmidt norm $\|G\|_{HS}=\sqrt{\text{tr}(GG^*)}=\sqrt{\sum_{i}\|Ge_i\|^2}$, where  $\{e_i\}_{i=1}^\infty$ is an orthonormal basis of $V$.
In the following, we work on $t\in[0,T]$, $H\in(\frac12,1)$ and $\alpha\in(1-H,\frac12)$ unless otherwise stated.

\subsection{Fractional integrals and derivatives}

Now we introduce generalized the  Riemann-Stieltjes integral, some spaces and norms that will be used in this work.

Let $W^{\alpha,1}([0,T],V)$ be the space of measurable functions $f:[0,T]\to V$
such that 
$$\|f\|_{\alpha,1}:=\int_0^T\left(\frac{\|f(s)\|}{s^\alpha}+\int_{0}^{s}\frac{\|f(s)-f(\lambda)\|}{(s-\lambda)^{\alpha+1}}d\lambda\right)\mathrm{d}s<\infty.$$

For measurable functions $f:[0,T]\to V$, denote that $W^{\alpha,\infty}([0,T],V)$ with norm
$$\|f\|_{\alpha,\infty}:=\sup_{t\in[0,T]}\|f\|_{\alpha,[0,t]}<\infty,$$
where
$\|f\|_{\alpha,[0,t]}:=\|f(t)\|+\int_0^t\frac{\|f(t)-f(s)\|}{(t-s)^{\alpha+1}}\mathrm{d}s.$

Denote by $\mathcal{B}^{\alpha,2}([0,T],V)$ the space of measurable functions $f:[0,T]\to V$ endowed with the norm $\|\cdot\|^2_{\alpha,T}$ denoted by 
$$\|f\|^2_{\alpha,T}:=\sup_{t\in[0,T]}\|f(t)\|^2+\int_0^T\left(\int_0^t\frac{\|f(t)-f(s)\|}{(t-s)^{\alpha+1}}\mathrm{d}s\right)^2\mathrm{d}t<\infty.$$
It is easy to verify that $W^{\alpha,\infty}([0,T],V)\subset \mathcal{B}^{\alpha,2}([0,T],V)$.

Then, we introduce H\"{o}lder continuous path space. For $\eta\in(0,1]$, let $C^\eta([0,T],V)$ be the space of $\eta$-H\"{o}lder continuous functions $f:[0,T]\to V$ with the norm
$$\|f\|_{\eta-hld}:=\|f\|_\infty+\sup_{0\leq s <t \leq T}\frac{\|f(t)-f(s)\|}{(t-s)^\eta}<\infty,$$
where $\|f\|_\infty=\sup_{t\in[0,T]}\|f(t)\|$. 

Denote by $W^{1-\alpha,\infty}_T([0,T],V)$, the space of measurable functions $f:[0,T]\to V$ such that 
$$\|f\|_{1-\alpha,0,T}:=\sup_{0\leq s<t\leq T}\Big(\frac{\|f(t)-f(s)\|}{(t-s)^{1-\alpha}}+\int_s^t\frac{\|f(\zeta)-f(s)\|}{(\zeta-s)^{2-\alpha}}\mathrm{d}\zeta\Big)<\infty.$$ 
It is easy to verify that for any $\kappa \in(0,\alpha)$, $C^{1-\alpha+\kappa}\subset W^{\alpha,\infty}_T\subset C^{1-\alpha-\kappa}.$ 

Then, according to Z\"{a}hle \cite{Zahle1998}, for $0\leq s<t\leq T$, $f\in W^{\alpha,1}([0,T],V)$ and $g \in W^{1-\alpha,\infty}([0,T],V)$, the generalized Riemann-Stieltjes integrals
\begin{equation}\label{eqn-2.2} 
	\int_0^T f(r)\mathrm{d}g(r)=(-1)^\alpha\int_0^TD_{0+}^\alpha f(r)D_{T-}^{1-\alpha}g_{T-}(r)\mathrm{d}r,
\end{equation}
$$
\int_s^t f(r)\mathrm{d}g(r)=\int_0^T f(r)\mathbf{1}_{(s,t)}\mathrm{d}g(r),
$$
are defined, where $g_{T-}(t):=g(t)-g(T)$.
Recall that the Weyl derivatives
of $f$ and $g$ are defined respectively as follows,
\begin{align*}
	D_{a+}^{\alpha}f(t) :=\frac{1}{\Gamma(1-\alpha)}\left(\frac{f(t)}{(t-a)^\alpha}+\alpha\int_a^t\frac{f(t)-f(s)}{(t-s)^{\alpha+1}}\mathrm{d}s\right),
\end{align*}
$$
D_{b-}^{1-\alpha}g_{b-}(t) :=\frac{(-1)^{1-\alpha}}{\Gamma(\alpha)}\bigg(\frac{g(t)-g(b)}{(b-t)^{1-\alpha}}+{(1-\alpha)}\int_t^b\frac{g(t)-g(s)}{(s-t)^{2-\alpha}}\mathrm{d}s\bigg),
$$ 
for almost all $t\in(a,b)$, where $a,b\in\mathbb{R}$, $a<b$ and $\Gamma$ stands for the Euler Gamma function. Then, the integral $\eqref{eqn-2.2}$ has the following estimate
\begin{equation}\label{eqn-4}
	\Big\|\int_0^T f(t)\mathrm{d}g(t)\Big\|\le\frac{\left\|g\right\|_{1-\alpha,0,T}}{\Gamma(1-\alpha)\Gamma(\alpha)}\|f\|_{\alpha,1}.
\end{equation}
For the sake of shortness, we denote $\Lambda_{\alpha,g}^{0,T}:=\frac{\|g\|_{1-\alpha,0,T}}{\Gamma(1-\alpha)\Gamma(\alpha)}$.
For a detailed account on fractional integrals and derivatives we refer to \cite{Samko1993}.

We recall the following two auxiliary technical lemmas from \cite{Garrido2010}.
\begin{lemma}\label{lem2.1}
	For any positive constants $a$ and $d$, if $a+d-1>0$ and $a < 1$, one has
	$$
	\begin{aligned}\int_0^r(r-s)^{-a}(t-s)^{-d}\mathrm{d}s&\le(t-r)^{1-a-d}B(1-a,d+a-1),\\\int_r^t(s-r)^{-a}(t-s)^{-d}\mathrm{d}s&\le(t-r)^{1-a-d}B(1-a,d+a-1),\end{aligned}
	$$
	where $r\in(0,t)$ and $B$ is the Beta Function.
\end{lemma}

\begin{lemma}\label{lem2.2}
	For any non-negative $a$ and $d$ such that $a + d < 1$, and for any $\rho \geq 1$, there exists a positive constant $C$ such that
	$$\int_0^te^{-\rho(t-r)}(t-r)^{-a}r^{-d}\mathrm{d}r\leq C\rho^{a+d-1}.$$
\end{lemma}

\subsection{Mixed fractional Brownian motion}
\setlength{\parindent}{2em} In this subsection, we introduce a $V\times V$-valued mixed FBM of the Hurst parameter $H$ and recall some basic facts on it for later use. Let $(\Omega,\mathcal{F},\mathbb{P})$ be a suitable probability space with a filtration satisfying the usual condition.


Let $\beta^H=(\beta_t^H)_{t\in[0,T]}$ denote a one-dimensional FBM with the Hurst parameter $H$. That is, $\beta^H$ is a centered Gaussian process with the covariance function (cf. \cite{Mandelbrot1968})
$$\mathbb{E}\left[\beta^H_t\beta^H_s\right]=\frac12\left(t^{2H}+s^{2H}-|t-s|^{2H}\right).$$

For $b\in W^{\alpha,1}([0,T],V)$, the integral
$\int_0^T b(s)\mathrm{d}\beta^H_s$
is understood in the sense of definition  \eqref{eqn-2.2} pathwise, which makes sense because $\mathbb{E}[\Lambda_{\alpha,\beta^{H,i}}^{0,T}]$ is a constant which is finite (cf. \cite{Nualart2002}). Giving a orthonormal basis $\{e_i\}^{\infty}_{i=1}$ of $V$, a bounded sequence of non-negative numbers $\{\lambda_i\}^{\infty}_{i=1}$ and a sequence of independent one-dimensional FBM $\{\beta^{H,i}\}^{\infty}_{i=1}$. The cylindrical FBM $B^H$ is defined by the formal sum
\begin{equation}\label{eqn-2.3} 
	B^H:=\sum_{i=1}^\infty\sqrt{\lambda_i}e_i\beta^{H,i}
\end{equation}	
and its incremental covariance operator $Q_1$ is a bounded, non-negative and self-adjoint linear operator on $V$ which is given by $Q_1e_i=\lambda_ie_i$, $i\in\mathbb{N}$. Moreover, $Q_1\in L_1(V)$, i.e. $\text{tr}(Q_1):=\sum_{i=1}^\infty\lambda_i<\infty$. Notice that the above series is convergent in $L^2(\Omega,\mathcal{F},\mathbb{P})$ 
from the fact that
$\sum_{i=1}^\infty\lambda_i<\infty$ and $\mathbb{E}[(\beta^{H,i}_t)^2]=t^{2H}$, $t\in\mathbb{R}$.

For a standard scalar BM $\beta=(\beta_t)_{t\in[0,T]}$, define $\beta^H$ by
$\beta^H_t=\int_{0}^{T}K_H(t,s)\mathrm{d}\beta_s,$
where we set for all $0\leq s\leq t\leq T$,
$K_H(t,s):= k_H(t,s)\textbf{1}_{[0,t]}(s),$
with
$$k_H(t,s):=\frac{c_H}{\Gamma\left(H+\frac12\right)}(t-s)^{H-\frac12}F\Big(H-\frac12,\frac12-H,H+\frac12;1-\frac ts\Big),$$
where $c_H=\Big[\frac{2H\Gamma\left(\frac32-H\right)\Gamma\left(H+\frac12\right)}{\Gamma(2-2H)}\Big]^{\frac12}$ and $F$ is the Gauss hypergeometric function.
Then, we have
\begin{align}\label{eqn-2}
	B_t^H=\sum_{i=1}^{\infty}\sqrt{\lambda_i}e_i
	\int_{0}^{T}K_H(t,s)\mathrm{d}\beta^i_s=\int_{0}^{T}K_H(t,s)\mathrm{d}B_s,\quad t\in[0,T],
\end{align}
where $\{\beta^i\}_{i=1}^\infty$ is a sequence of independent standard BM and $B=\sum_{i=1}^{\infty}\sqrt{\lambda_i}e_i\beta^i$ is a $V$-valued BM ($Q_1$-Wiener process) (cf. \cite{Budhiraja2008}).

Let $G:\Omega\times[0,T]\to L(V)$ be an operator-valued map such that $G(\omega,\cdot)e_i\in W^{\alpha,1}([0,T],V)$ for each $i\in\mathbb{N}$ and almost $\omega\in\Omega$. We define
\begin{equation}\label{eqn-2.5} 
	\int_0^TG(\omega,s)\mathrm{d}B_s^H:=\sum_{i=1}^\infty\int_0^TG(\omega,s)Q_1^\frac12e_i\mathrm{d}\beta^{H,i}_s=\sum_{i=1}^\infty\sqrt{\lambda_i}\int_0^TG(\omega,s)e_i\mathrm{d}\beta^{H,i}_s,
\end{equation}
where the convergence of the sums in  $\eqref{eqn-2.5}$ is understood as $\mathbb{P}$-a.s. convergence in $V$ (cf. \cite{Duncan2002}).

Furthermore, we assume that $Q_1^{\frac12}\in L_1(V)$ to make the pathwise integral $\eqref{eqn-2.5}$ well-defined. The following result can be found in \cite[Proposition 2.1]{Maslowski2003}.

\begin{remark}\label{re2.1}
	Assume that $Q_1^{\frac12}\in L_1(V)$. Then, there exists $\Omega_1\subset\Omega$, $\mathbb{P}(\Omega_1)=1$, such that for any $G:\Omega\times[0,T]\to L(V)$ satisfying $G(\omega,\cdot)e_i\in W^{\alpha,1}([0,T],V)$ with $\omega\in\Omega_1$ and $\sup_{i\in\mathbb{N}}\|G(\omega,\cdot)e_i\|_{\alpha,1}<\infty$, the pathwise integral $\eqref{eqn-2.5}$ is well-defined on $\Omega_1$. In addition,
	$$\Big\|\int_0^TG(\omega,s)\mathrm{d}B_s^H\Big\|\leq\Lambda_{\alpha,B^H}^{0,T}\sup_{i\in\mathbb{N}}\|G(\omega,\cdot)e_i\|_{\alpha,1},\quad\omega\in\Omega_1,$$
	where $\Lambda_{\alpha,B^H}^{0,T}:=\sum_{i=1}^\infty\sqrt{\lambda_i}\Lambda_{\alpha,\beta^{H,i}}^{0,T}$. Note that $\Lambda_{\alpha,B^H}^{0,T}$ is finite a.s.\\
\end{remark}

Define the integral operator $\mathbb{K}_H$ induced from the kernel $K_H$ by
$\mathbb{K}_Hf(t)=\int_{0}^{T}K_H(t,s)f(s)\mathrm{d}s,$
for $f\in L^2([0,T],V)$. 


Let $V_{1}=Q_1^{\frac{1}{2}}V$, then $V_1$ is a Hilbert space with the inner product
$\langle h,k\rangle_{1}:=\langle Q_1^{-\frac12}h,Q_1^{-\frac12}k\rangle$, for $h,k\in V_1$.
Let $\|\cdot\|_{1}$ denote the norm in the Hilbert space $V_1$.

Since $Q_1^{\frac12}$ is a trace-class operator, the identity mapping from $V_1$ to $V$ is Hilbert-Schmidt. The Cameron-Martin Hilbert space $\mathcal{H}^H=\mathcal{H}^H([0,T],V_1)$ for $(B^H_t)_{t\in [0,T]}$ is defined by
$
\mathcal{H}^H=\{\mathbb{K}_H\dot{h}:\dot{h}\in L^2([0,T],V_1)\}.
$ 
Note that $\dot{h}$ is the weak derivative of $h$. It should be recalled that $\mathcal{H}^H\subset C^{H'}([0,T],V_1)$ for all
$H'\in(0,H)$. The inner product on $\mathcal{H}^H$ is defined by
$\langle h,g\rangle_{\mathcal{H}^H}=\langle\mathbb{K}_H\dot{h},\mathbb{K}_H\dot{g}\rangle_{\mathcal{H}^H}:=\langle\dot{h},\dot{g}\rangle_{L^2([0,T],V_1)}.$

For any $h=\mathbb{K}_H\dot{h}\in\mathcal{H}^H$, $h$ is differentiable and 
\begin{align}
	h'(t)=\frac{c_Ht^{H-\frac12}}{\Gamma\left(H-\frac12\right)}\int_0^t(t-s)^{H-\frac32}s^{\frac12-H}\dot{h}(s)\mathrm{d}s.\label{eqn-2.6}
\end{align}

\begin{lemma}\label{lem4.1}
	Let $h\in\mathcal{H}^H$. Then, $ Q_1^{-\frac12}h\in C^H([0,T],V)$ and $\|Q_1^{-\frac12}h\|_{H-hld}\leq C\|h\|_{\mathcal{H}^H}$.
	Here, $C$ is a positive constant which depends only on $T$ and $H$.
\end{lemma}

\begin{lemma}\label{lem4.2}
	Assume that $Q_1^{\frac12}\in L_1(V)$, or equivalently,
	$\sum_{i=1}^\infty\sqrt{\lambda_i}<\infty.$
	for any $G:[0,T]\to L(V)$ satisfying for $i\in\mathbb{N}$, $G e_i\in W^{\alpha,1}([0,T],V)$ and 
	$\sup_{i\in\mathbb{N}}\|Ge_i\|_{\alpha,1}<\infty$.
	Then, for any $u\in\mathcal{H}^H$, the pathwise integral
	\begin{align}
		\int_0^T G( s) \mathrm{d}u_s:=\sum_{i=1}^{\infty}\sqrt{\lambda_i}\int_0^T G( s)e_i \mathrm{d}\big(Q_1^{-\frac12}u_se_i \big),\label{eqn-9}
	\end{align}
	is well defined. In addition, we have
	\begin{align}
		\Big\|\int_0^TG( s)\mathrm{d}u_s\Big\|\leq\Lambda_{\alpha, u}^{0, T}\sup_{i\in\mathbb{N}}\|Ge_i\|_{\alpha, 1}.\label{eqn-2.11}
	\end{align}
\end{lemma}

We present the complete proofs of Lemmas \ref{lem4.1} and \ref{lem4.2} in Appendix \ref{secA.2} and \ref{secA.3}.


We also consider a $V$-valued BM $(W_t)_{t\in[0,T]}$ generated by $Q_2$, which is also a bounded, non-negative, self-adjoint and trace-class operator on $V$. Throughout, $(W_t)_{t\in[0,T]}$ and $(B^H_t)_{t\in[0,T]}$ are assumed to be independent.

Let $V_{2}=Q_2^{\frac{1}{2}}V$, then $V_2$ is a Hilbert space with the inner product
$\langle h,k\rangle_{2}:=\langle Q_2^{-\frac12}h,Q_2^{-\frac12}k\rangle$, for $h,k\in V_2$.
Let $\|\cdot\|_2$ denote the norm in the Hilbert space $V_2$. The Cameron-Martin Hilbert space $\mathcal{H}^\frac12=\mathcal{H}^\frac12([0,T],V_2)$ for $(W_t)_{t\in[0,T]}$  is defined by 
$\mathcal{H}^\frac12:=\{v=\int_0^\cdot v_s'\mathrm{d}s\in C([0,T],V_2):v'\in L^2([0,T],V_2)\}.$
The inner product of $\mathcal{H}^\frac12$ is defined by $\langle v,w\rangle_{\mathcal{H}^\frac12}:=\int_0^T\langle v_t',w_t'\rangle_2\mathrm{d}t.$

The infinite-dimensional process $(B^H_t,W_t)_{t\in[0,T]}$ is called mixed FBM of Hurst parameter $H$. 
Then, $\mathcal{H}:=\mathcal{H}^H\oplus\mathcal{H}^\frac12$ is the Cameron-Martin subspace for mixed FBM $(B^H_t,W_t)_{t\in[0,T]}$.

\subsection{Variational representation for mixed FBM}

For $N\in\mathbb{N}$, we set
$$\mathcal{S}_N=\left\{(u,v)\in\mathcal{H}:\frac12\|(u,v)\|^2_\mathcal{H}:=\frac12(\|u\|^2_{\mathcal{H}^H}+\|v\|^2_{\mathcal{H}^\frac12})\leq N\right\}.$$  
Equipped with the weak topology, the ball $\mathcal{S}_N$ can be metrized as a compact Polish space (a complete separable metric space).

Let $(\Omega,\mathcal{F},\mathbb{P})$ be the infinite-dimensional classical Wiener space, where $\Omega=C_0([0,T],V\times V)$ is a space of $V\times V$-valued continuous paths starting at 0, equipped with the uniform topology. The  Wiener measure $\mathbb{P}$ is defined on $\Omega$, and $\mathcal{F}$ is the $\mathbb{P}$-completion of the Borel $\sigma$-field on $\Omega$. The coordinate process  $(B_t,W_t)_{t\in[0,T]}$ represents a standard $V\times V$-valued BM under $\mathbb{P}$.  Let $\{\mathcal{F}_t\}_{t\geq0}$ denoted the natural augmented filtration, where $\mathcal{F}_t:=\sigma\{(B_s,W_s):0\leq s\leq t\}\vee\mathcal{N}$,
and $\mathcal{N}$ is the collection of all $\mathbb{P}$-negligible sets.
By constructing $B^H$ from $B$ using  $\eqref{eqn-2}$, we obtain a FBM $(B^H_t)_{t\in[0,T]}$ with the Hurst parameter $H$ defined on $(\Omega,\mathcal{F},\mathbb{P})$.  It is known that for all $t$,
$\mathcal{F}_t:=\sigma\{(B^H_s,W_s):0\leq s\leq t\}~\vee\mathcal{N}$.
The subsequent work in this paper will be carried out within the probability space $(\Omega,\mathcal{F},\mathbb{P})$ as defined above.

We denote $\mathcal{B}_b^N$, $N\in\mathbb{N}$, the set of all $V_1\times V_2$-valued $\{\mathcal{F}_t\}$-progressively measurable processes
$(\phi_t,\psi_t)_{t\in[0,T]}$ satisfying that
$\frac{1}{2}\int_{0}^{T}(\|\phi_s\|^2_1+\|\psi_s\|^2_2)\mathrm{d}s\leq N,~\mathbb{P}\text{-a.s.}$
We set $\mathcal{B}_b=\bigcup_{0<N<\infty}\mathcal{B}_b^N$.

Let $\mathcal{A}_b^N$ for $N\in\mathbb{N}$ denote the set of all $V_1\times V_2$-valued $\{\mathcal{F}_t\}$-progressively measurable processes $(u_t,v_t)_{t\in[0,T]}$ satisfying that
$(u,v)=(\mathbb{K}_H\dot{u},\int_0^\cdot v_s'\mathrm{d}s),~\mathbb{P}$-a.s., for some  $(\dot{u},v')\in\mathcal{B}_b^N$.
We set $\mathcal{A}_b=\bigcup_{0<N<\infty}\mathcal{A}_b^N$. Every $(u,v)\in\mathcal{A}_b^N$ can be viewed as an $\mathcal{S}_N$-valued random variable. Since $\mathcal{S}_N$ is compact, $\{\mathbb{P}\circ(u,v)^{-1}:(u,v)\in\mathcal{A}_b^N\}$ is automatically tight. By the Girsanov's formula \cite[Theorem 10.18]{Prato2014}, the law of the process $(B^H_t+u_t,W_t+v_t)_{t\in[0,T]}$ is mutually absolutely continuous to that of $(B^H_t,W_t)_{t\in[0,T]}$ for every $(u,v)\in\mathcal{A}_b^N$.

A variational representation formula for $V\times V$-valued mixed FBM is now given.
\begin{lemma}\label{lem2.3}
	Let $F$ be a bounded, Borel measurable function mapping $C([0, T] ,
	V\times V)$ into $\mathbb{R}$. Then
	$$-\log \mathbb{E}\left[\exp\{-F(B^H,W)\}\right]\\
	=\inf_{(u,v)\in\mathcal{A}_b}\mathbb{E}\left[F\left(B^H+u,W+v\right)+\frac12\|(u,v)\|_\mathcal{H}^2\right].$$
\end{lemma}
\begin{proof}
	From Budhiraja \cite[Theorem 3.6]{Budhiraja2000}, for any bounded Borel measurable function $F: C([0, T] ,V\times V)\to \mathbb{R}$, we have
	\begin{align*}
		&-\log \mathbb{E}\left[\exp\{-F(B,W)\}\right]\\
		=&\inf_{(\phi,\psi)\in\mathcal{B}_b}\mathbb{E}\bigg[F\Big(B+\int_0^\cdot \phi_s\mathrm{d}s,W+\int_0^\cdot \psi_s\mathrm{d}s\Big)+\frac12\|(\phi,\psi)\|_{L^2([0,T],V_1\times V_2)}^2\bigg].
	\end{align*}
	Note that $F$ can be taken as a $\mathbb{P}$-equivalence class of bounded measurable functions, rather than being defined everywhere. Specifically, if $F=\hat{F}$, $\mathbb{P}$-a.s. (i.e. $F(B,W)=\hat{F}(B,W),~\mathbb{P}$-a.s.), then we have $F(B+\int_{0}^{\cdot}\phi_s\mathrm{d}s,W+\int_{0}^{\cdot}\psi_s\mathrm{d}s)=\hat{F}(B+\int_{0}^{\cdot}\phi_s\mathrm{d}s,W+\int_{0}^{\cdot}\psi_s\mathrm{d}s)$, $\mathbb{P}$-a.s.,  for any $(\phi,\psi)\in\mathcal{B}_b$. This follows from the mutual absolute continuity mentioned above.

	Finally, by letting $F$ be the measurable map $(B,W)\mapsto\Phi(B^H,W)$ and replacing the symbol $(\phi,\psi)$ by $(\dot{u},v')$, we obtain the desired formula. This completes the proof.
\end{proof}

\section{Assumptions and Statements of Main Result}\label{3}

Let $A$ be an infinitesimal generator of an analytic semigroup $S$ on $V$. Assume that $-A$ has discrete spectra $0<\bar{\lambda}_1<\bar{\lambda}_2<\cdots<\bar{\lambda}_k<\cdots,$ and $\lim_{k\to\infty}\bar{\lambda}_k=\infty$.
Let $V_\beta$, $\beta\geq0$, denote the domain of the fractional power $(-A)^\beta$ equipped with the graph norm $\|x\|_{V_\beta}:=\|(-A)^\beta x\|,x\in V_\beta$.
The operator $S_rg(\cdot)$ is defined on a dense subspace of $V$ and for each $r\in[0,T]$ the operator $S_rg(\cdot)$ is  extendable to $V$ and $V\to L(V)$ is Lipschitz continuous. Recall that $\{e_i\}_{i=1}^\infty$ is an orthonormal basis of $V$.

We recall here some properties of the analytic semigroup. There exists a constant $C>0$, such that for any $0\leq\gamma\leq\zeta\leq1$, $\nu\in[0,1),\mu\in(0,1-\nu)$, $0\leq s\leq t\leq T$ and $\bar{\lambda}_1$, 
\begin{align}
	\|S_t\|_{L(V_\gamma,V_\varsigma)}\le& Ct^{-\varsigma+\gamma}e^{-\bar{\lambda}_1t},\label{eqn-3.1}\\
	\|S_{t-s}-\mathrm{id}\|_{L(V_{\upsilon+\mu},V_\upsilon)}\le& C(t-s)^\mu.\label{eqn-3.2}
\end{align}
We also note that, there exists a constant $C>0$, such that for any $\varrho,\nu\in(0,1]$, $0\leq\nu<\gamma+\varrho$ and $0\leq q\leq r\leq s\leq t$, 
\begin{align}
	\label{eqn-3.3}
	\|S_{t-r}-S_{t-q}\|_{L(V_{\nu},V_{\gamma})}\le& C(r-q)^{\varrho}(t-r)^{-\varrho-\gamma+\nu},\\
	\label{eqn-3.4}
	\|S_{t-r}-S_{s-r}-S_{t-q}+S_{s-q}\|_{L(V,V)}\le& C(t-s)^{\varrho}(r-q)^{\nu}(s-r)^{-(\varrho+\nu)}.
\end{align}

To ensure the existence and uniqueness of solution to the system \eqref{eqn-1.2}, we assume:

\begin{itemize}
	\item[\textbf{(A1)}] The coefficients $b(x,y):V\times V\to V$, $F(x,y):V\times V\to V$, $G(x,y):V\times V\to L_2(V)$ of  \eqref{eqn-1.2} are globally Lipschitz continuous in $x,y$, i.e., there exist two positive constants $C_1,C_2$, such that for any $x_1,x_2,y_1,y_2\in V$,  
	$$\begin{array}{rcl}\|b(x_1,y_1)-b(x_2,y_2)\|&\le&C_1(\|x_1-x_2\|+\|y_1-y_2\|),\\
		\|F(x_1,y_1)-F(x_2,y_2)\|&+&\|G(x_1,y_1)-G(x_2,y_2)\|_{HS}\\
		&\le& C_2(\|x_1-x_2\|+\|y_1-y_2\|).\end{array}$$
	
	\item[\textbf{(A2)}] The coefficients $b(x,y), F(x,y), G(x,y)$ of  \eqref{eqn-1.2} satisfy linear growth conditions, i.e., there exist two positive constants $C_3,C_4$, such that for any $x,y\in V$,
	$$\begin{array}{rcl}\|F(x,y)\|+\|G(x,y)\|_{HS}&\le& C_{3}(1+\|x\|+\|y\|),\\
		\|b(x,y)\|&\le& C_{4}(1+\|x\|+\|y\|).\end{array}$$
	
	\item[\textbf{(A3)}] The coefficients $g:V\to L(V)$ and $g':V\to L(V,L(V))$ are Lipschitz continuous, i.e., there exist two positive constants $L_g,M_g$, such that for any $x_1,x_2\in V$, 
	$$\begin{array}{rcl}\sup_{i\in\mathbb{N}}\|g(x_1)e_i-g(x_2)e_i\|&\le& L_g\|x_1-x_2\|,\\
		\sup_{i\in\mathbb{N}}\|g'(x_1)e_i-g'(x_2)e_i\|_{L(V)}&\le& M_{g}\|x_1-x_2\|.\end{array}$$
	
\end{itemize}

Under Assumptions \textbf{(A1)-(A3)} above, one can deduce from \cite[Lemma 4.2]{Pei2020}  that the slow-fast system \eqref{eqn-1.2} admits a unique mild pathwise solution $(X^{\varepsilon, \delta},Y^{\varepsilon, \delta})$. 

\begin{lemma}\label{lem3.1}
	Suppose that Assumptions \textnormal{\textbf{(A1)-(A3)}} hold. Then, for any initial values $X_0,Y_0\in V_\beta, \beta>\alpha$,  slow-fast system \eqref{eqn-1.2} has a unique mild pathwise solution $(X^{\varepsilon, \delta},Y^{\varepsilon, \delta})$, i.e.,
	$$\begin{cases}
		X_{t}^{\varepsilon,\delta}=S_{t}X_{0}^{\varepsilon,\delta}+\int_{0}^{t}S_{t-s}b(X_{s}^{\varepsilon,\delta},Y_{s}^{\varepsilon,\delta})\mathrm{d}s+\sqrt\varepsilon\int_{0}^{t}S_{t-s}g(X_{s}^{\varepsilon,\delta})\mathrm{d}B_{s}^{H},\\
		Y_{t}^{\varepsilon,\delta}=S_{\frac{t}{\delta}}Y_{0}^{\varepsilon,\delta}+\frac{1}{\delta}\int_{0}^{t}S_{\frac{t-s}{\delta}}F(X_{s}^{\varepsilon,\delta},Y_{s}^{\varepsilon,\delta})\mathrm{d}s+\frac{1}{\sqrt{\delta}}\int_{0}^{t}S_{\frac{t-s}{\delta}}G(X_{s}^{\varepsilon,\delta},Y_{s}^{\varepsilon,\delta})\mathrm{d}W_{s},\\ X_{0}^{\varepsilon,\delta}=X_{0}, ~Y_{0}^{\varepsilon,\delta}=Y_{0},~t\in[0,T].
	\end{cases}$$
	
\end{lemma}

Consequently, there is a measurable map 
$\mathcal{G}^{\varepsilon, \delta}:C_0([0,T],V\times V)\to C([0,T],V)$
such that $X^{\varepsilon, \delta}=\mathcal{G}^{\varepsilon, \delta}(\sqrt\varepsilon B^H, \sqrt\varepsilon W)$.

In order to study an AP and the LDP for the slow-fast system \eqref{eqn-1.2}, we further impose:
\begin{itemize}
	\item[\textbf{(A4)}] There exist constants $\beta_1,C_5>0$ and $\beta_2, \beta_3\in\mathbb{R}$, such that for any $x_1,x_2,y_1,y_2\in V$, 
	$$\begin{array}{rcl}\left\langle y_1,F(x_1,y_1)\right\rangle&\le&-\beta_1\|y_1\|^2+\beta_2,\\
		\left\langle y_1-y_2,F(x_1,y_1)-F(x_2,y_2)\right\rangle&\le&\beta_3\|y_1-y_2\|^2+C_5\|x_1-x_2\|^2.\end{array}$$ 
	\item[\textbf{(A5)}] $\eta:=2\bar\lambda_1-2\beta_3-C_2>1$, $\kappa:=2\bar\lambda_1+2\beta_1-C_3>0$, where $\bar\lambda_1$ is the first eigenvalue of $-A$, $C_i,~i=2,3$ and $\beta_j,~j=1,2,3$ were given in Assumptions \textbf{(A1)}, \textbf{(A2)} and \textbf{(A4)}.	
	\item[\textbf{(A6)}] There exists a constant $C_6>0$, such that for any $x\in V$,
	$$\sup_{y\in V}(\|b(x,y)\|+\|G(x,y)\|_{HS})\leq C_6(1+\|x\|).$$
	
	\item[\textbf{(A7)}]  The scale parameter $\varepsilon$  satisfy $\lim_{\varepsilon\to0}\frac\delta\varepsilon=0.$
\end{itemize}

\begin{remark}	\label{re3.1}
	There exists a constant $C>0$, such that for any $x_1,x_2,y_1,y_2\in V$, 
	$$\begin{aligned}\sup_{i\in\mathbb{N}}\|g(x_1)e_i-g(x_2)e_i-g(y_1)e_i+g(y_2)e_i\|\le&C\|x_1-x_2-y_1+y_2\|\\&+C\|x_1-x_2\|(\|x_1-y_1\|+\|x_2-y_2\|)\end{aligned}$$
	holds (cf. \cite[Lemma 7.1]{Nualart2002}).
\end{remark}

\begin{remark}	\label{re3.2}
	The relationship between the rates at which the time-scale parameter 
	$\delta$ and the noise intensity $\varepsilon$ converge to zero can be characterized by the following three regimes,
	$$\lim_{\varepsilon\to0}\frac\delta\varepsilon=\left\{
	\begin{array}
		{cl}0,\quad\quad\quad~~ & \mathrm{Regime}~1; \\
		\gamma\in(0,\infty), & \mathrm{Regime}~2; \\
		\infty, \quad\quad\quad~& \mathrm{Regime}~3.
	\end{array}\right.
	$$
	In this work, we assume that $\delta$ decays faster than $\varepsilon$, corresponding to Regime 1.

	In the proof of the LDP, we introduce the controlled system to employ the variational representation. It is necessary to establish the weak convergence for the controlled slow component. 
	However, no result shows that there exists a unique invariant measure for the controlled equation. Therefore, we need to find a ``replaced" equation which owns a unique invariant measure in the weak convergence.
	Assumption \textbf{(A7)} ($\lim_{\varepsilon\to0}\frac\delta\varepsilon=0$) ensures that we can apply the invariant measure of the frozen equation corresponding to the fast equation to average the controlled system.

	In the other two regimes, due to the coupling with the control terms, the long-term behavior of the controlled fast equation depends non-trivially on the control term $v$, so the above method becomes invalid. As a result, the 
	$y$-marginal of the limiting occupation measure is no longer decoupled from the control term $v^\varepsilon$. Thus, proving the convergence of the controlled system becomes challenging, making it difficult to obtain the optimal control for the Laplace principle's lower bound.

	When $H=\frac12$, \cite{dupuis2012large, Spiliopoulos2013} prove the LDP for the three regimes with the help of viable pairs, relying on a local form of the rate function and tools from ergodic control of diffusion processes, which are accessible due to the Markovian nature of the dynamics.	However, in infinite-dimensional spaces and  for $H\neq\frac12$, the viable pair method is ineffective because it is difficult to obtain the tightness of the controlled systems by directly using the Arzela-Ascoli's theorem.

	Therefore, establishing the LDP in the other two regimes in infinite-dimensional spaces presents significant challenges and requires further investigation.
\end{remark}

The ergodicity for fast component have been investigated by Fu and Liu \cite[p. 74]{Fu2011}.
Then, the skeleton equation is defined as follows
\begin{equation}\label{eqn-4.1}
	\mathrm{d}\bar{X}^u_t=\left(A\bar{X}^u_t+\bar{b}(\bar{X}^u_t)\right)\mathrm{d}t+g(\bar{X}^u_t)\mathrm{d}u_t,\quad \bar{X}^u_0=X_0,\quad t\in[0,T].
\end{equation}

According to the definition of $\bar{b}$ and Assumptions \textbf{(A1)}-\textbf{(A5)}, it is easy to prove $\bar{b}$ also
satisfies the Lipschitz and linear growth conditions. Then, we have the following lemma by making slight extension (cf. \cite[Lemma 4.3]{Pei2020}). 

\begin{lemma}\label{lem3.2}
	Suppose that Assumptions \textnormal{\textbf{(A1)}-\textbf{(A5)}} hold. Then, for any initial value $X_0\in V_\beta$, $\beta>\alpha$ and $(u,v)\in\mathcal{S}_N$,  the skeleton equation \eqref{eqn-4.1} has a unique mild pathwise solution. Moreover, $\bar{b}$
	satisfies the Lipschitz and linear growth conditions.
\end{lemma}

We define a map
$\mathcal{G}^{0}:\mathcal{H}\to C([0,T],V)$
by $\bar{X}^u=\mathcal{G}^{0}(u,v)$. Note the $\mathcal{G}^{0}(u,v)$ is independent of $v$.

The main result of the LDP is formulated as follows.

\begin{theorem}\label{th3.1}
	Suppose that Assumptions \textnormal{\textbf{(A1)}-\textbf{(A7)}} hold. Let $\varepsilon\to 0$. Then the slow component $X^{\varepsilon, \delta}$ of the slow-fast system \eqref{eqn-1.2} satisfies the LDP in $C([0,T],V)$ with the good rate function $I$ given by
	$$ I(g):=\inf_{\{(u,v)\in\mathcal{H}:~g=\mathcal{G}^0(u,v)\}}{\frac12\|(u,v)\|^2_{\mathcal{H}}}
	=\inf_{\{u\in\mathcal{H}^H:~\phi=\mathcal{G}^0(u,0)\}}{\frac12\|u\|^2_{\mathcal{H}^H}},$$
	where infimum over an empty set is taken as $+\infty$.
\end{theorem}

\begin{remark}\label{round2-1}
Under appropriate conditions, such as the diffusion matrix $g$ being symmetric with a bounded inverse (and assuming $gg^T$ is uniformly non-degenerate when $H=\frac12$), the explicit representation of the rate function can be derived.
From the skeleton equation \eqref{eqn-4.1}, we have
\begin{align*}
	I(\phi)
	=\inf_{\{u\in\mathcal{H}^H:~\phi=\mathcal{G}^0(u,0)\}}{\frac12\|u\|^2_{\mathcal{H}^H}}
	=\inf_{\{u\in \mathcal{H}^H:~\phi_0=X_0,~\phi'=A\phi+\bar{b}(\phi)+g(\phi)u'
		\}} \frac12\|u\|^2_{\mathcal{H}^H}
	,
\end{align*} 
where, 
$u'$
denote the time derivative of the control 
$u$ corresponding to the noise 
$B^H$, with the explicit expression given by 
\begin{align*}
	u'_t
	=\frac{\mathrm{d}}{\mathrm{d}t}\mathbb{K}_H(\mathbb{K}_H^{-1}u)=\mathbb{K}'_H(\mathbb{K}_H^{-1}u)
	=\frac{c_Ht^{H-\frac12}}{\Gamma\left(H-\frac12\right)}\int_0^t(t-s)^{H-\frac32}s^{\frac12-H}\dot{u}_s\mathrm{d}s,
\end{align*}
where 
$$
\dot{u}_t=\mathbb{K}_H^{-1}u_t
=
\big(c_H\Gamma(\frac32-H)\big)^{-1}\Big(t^{\frac12-H}u'_t+(H-\frac12)t^{H-\frac12}\int_0^t\frac{t^{\frac12-H}u'_t-s^{\frac12-H}u'_s}{(t-s)^{H+\frac12}}\mathrm{d}s\Big),
$$   
and for any $\psi\in L^2([0,T],V_1)$
$$
\mathbb{K}'_H\psi_t:=\frac{c_H t^{H-\frac12}}{\Gamma(H-\frac12)}\int_0^t (t-s)^{H-\frac32}s^{\frac12-H}\psi_s\mathrm{d}s.
$$
Further details can be found in Chapter 5.1.3 of \cite{nualart2006malliavin} and Sections 8.1.2 and 8.2.3 of \cite{stroock2010probability}.
Given the equation for	$u$, $\phi'=A\phi+\bar{b}(\phi)+g(\phi)u'$, the minimal-norm solution for $u$ is given by  
$u^*=\int_0^\cdot g(\phi_s)^{-1}\big(\phi'_s-A\phi_s-\bar{b}(\phi_s)\big)\mathrm{d}s$. Therefore, it follows that
\begin{align*}
	I(\phi)=&\frac12\|u^*\|^2_{\mathcal{H}^H}=\frac12\int_0^T \|\mathbb{K}_H^{-1}u^*_t\|^2_1\mathrm{d}t\\
	=&\frac{1}{2c_H^2\Gamma(\frac32-H)^2}\int_0^T
	\Big\|t^{\frac12-H}g(\phi_t)^{-1}\big(\phi'_t-A\phi_t-\bar{b}(\phi_t)\big)\\
	&+(H-\frac12)t^{H-\frac12}\int_0^t\frac{t^{\frac12-H}g(\phi_t)^{-1}\big(\phi'_t-A\phi_t-\bar{b}(\phi_t)\big)}{(t-s)^{H+\frac12}}\mathrm{d}s
	\\
	&-(H-\frac12)t^{H-\frac12}\int_0^t\frac{s^{\frac12-H}g(\phi_s)^{-1}\big(\phi'_s-A\phi_s-\bar{b}(\phi_s)\big)}{(t-s)^{H+\frac12}}\mathrm{d}s\Big\|_1^2\mathrm{d}t	
	\end{align*}
for all $\phi\in\mathcal{C}([0,T],V)$ such that $\phi_0=X_0$ and $g(\phi)^{-1}\big(\phi'-A\phi-\bar{b}(\phi)\big)\in\mathbb{K}'_H(L^2([0,T],V_1))$, and $I(\phi)=\infty$ otherwise.
The detailed can be found in \cite{gailus2025large}.
\end{remark}

\begin{remark}\label{re4.1}
	In this work, we weaken the global boundedness condition $\sup_{x,y\in V}$ $\|b(x,y)+G(x,y)\|<\infty$ to $\sup_{y\in V}\|b(x,y)\|\leq C(1+\|x\|)$. This modification differs from the conditions drawn in \cite{Pei2020}.
	
\end{remark}

\section{Some Key Lemmas}\label{4}

In this section,  We assume $0<\delta\ll\varepsilon<1$ are basically fixed. To prove Theorem \ref{th3.1}, we derive some
crucial estimates. Firstly, for any fixed $N\in\mathbb{N}$, let $(u^\varepsilon,v^\varepsilon)\in\mathcal{A}_b^N$. We consider the following controlled system associated to  $\eqref{eqn-1.2}$,
\begin{equation}\label{eqn-4.9}
	\begin{cases}
		\mathrm{d}\tilde{X}^{\varepsilon,\delta}_t=(A\tilde{X}^{\varepsilon,\delta}_t+b(\tilde{X}^{\varepsilon,\delta}_t,\tilde{Y}^{\varepsilon,\delta}_t))\mathrm{d}t+g(\tilde{X}^{\varepsilon,\delta}_t)\mathrm{d}u^\varepsilon_t+\sqrt\varepsilon g(\tilde{X}^{\varepsilon,\delta}_t)\mathrm{d}B^H_t,\\
		\begin{aligned}
			\mathrm{d}\tilde{Y}^{\varepsilon,\delta}_t=&\frac{1}{\delta}(A\tilde{Y}^{\varepsilon,\delta}_t+F(\tilde{X}^{\varepsilon,\delta}_t,\tilde{Y}^{\varepsilon,\delta}_t))\mathrm{d}t+\frac{1}{\sqrt{\delta\varepsilon}} G(\tilde{X}^{\varepsilon,\delta}_t,\tilde{Y}^{\varepsilon,\delta}_t)\mathrm{d}v^\varepsilon_t
			+\frac{1}{\sqrt\delta} G(\tilde{X}^{\varepsilon,\delta}_t,\tilde{Y}^{\varepsilon,\delta}_t)\mathrm{d}W_t,
		\end{aligned}  
		\\
		\tilde{X}^{\varepsilon,\delta}_0=X_0,~\tilde{Y}^{\varepsilon,\delta}_0=Y_0,~t\in[0,T].
	\end{cases}
\end{equation}	
By Girsanov's theorem, it deduces that there exists a unique mild solution $(\tilde{X}^{\varepsilon,\delta},\tilde{Y}^{\varepsilon,\delta})$ of the controlled system \eqref{eqn-4.9} and 
$\tilde{X}^{\varepsilon,\delta}=\mathcal{G}^{\varepsilon,\delta}(\sqrt\varepsilon B^H +u^\varepsilon, \sqrt\varepsilon W+v^\varepsilon).$

\begin{lemma}\label{lem4.3}
	Suppose that Assumptions \textnormal{\textbf{(A1)}-\textbf{(A3)}} and \textnormal{\textbf{(A6)}} hold. Let $N\in\mathbb{N}$. Then,  there exists a constant $C>0$, such that for any $p\geq1$ and $(u^\varepsilon,v^\varepsilon)\in\mathcal{A}_b^N$, we have
	$$\mathbb{E}\big[\|\tilde{X}^{\varepsilon,\delta}\|_{\alpha,\infty}^p\big]\leq C.$$ 
	Here, $C$ is a constant which only depends on $p$, $T$ and $N$.
\end{lemma}
\begin{proof}
	In this proof, $C$ is a positive constant depending on only $p$, $T$ and $N$, which may
	change from line to line. Firstly, we denote $\Lambda:=\Lambda_{\alpha,B^H}^{0,T}\vee1$. For any $\rho\geq1$, let
	$$\|f\|_{\rho,T}:=\sup_{t\in[0,T]}e^{-\rho t}\|f(t)\|,$$
	$$\|f\|_{1,\rho,T}:=\sup_{t\in[0,T]}e^{-\rho t}\int_0^t\frac{\|f(t)-f(r)\|}{(t-r)^{\alpha+1}}\mathrm{d}r.$$

	By  $\eqref{eqn-4.9}$, for any $\beta>\alpha$, we have
	\begin{align}	
		\|\tilde{X}^{\varepsilon,\delta}\|_{\rho,T}
		\leq&C\|X_0\|_{V_\beta}+\sup_{t\in[0,T]}e^{-\rho t}\Big\|\int_0^tS_{t-s}b(\tilde{X}^{\varepsilon,\delta}_s,\tilde{Y}^{\varepsilon,\delta}_s)\mathrm{d}s\Big\|\notag \\ &+\sup_{t\in[0,T]}e^{-\rho t}\sqrt\varepsilon \Big\|\int_0^tS_{t-s}g(\tilde{X}^{\varepsilon,\delta}_s)\mathrm{d}B^H_s\Big\|\notag \\
		&+\sup_{t\in[0,T]}e^{-\rho t}\Big\|\int_0^tS_{t-s}g(\tilde{X}^{\varepsilon,\delta}_s)\mathrm{d}u_s^\varepsilon\Big\|\notag\\
		=:&C\|X_0\|_{V_\beta}+I_1+I_2 + I_3. \label{eqn-3.11}
	\end{align}
	For the term $I_1$, by  $\eqref{eqn-3.1}$, Lemma \ref{lem2.2} and Assumptions \textbf{(A6)}, we have
	\begin{align}
		I_1
		\leq C(1+\|\tilde{X}^{\varepsilon,\delta}\|_{\rho,T})\int_0^T e^{-\rho (t-s)}s^{-\alpha}\mathrm{d}s \leq C(1+\rho^{\alpha-1}\|\tilde{X}^{\varepsilon,\delta}\|_{\rho,T}).\label{eqn-3.12}
	\end{align}
	Before computing $I_2$, for any $\gamma\in(\frac12,1-\alpha)$, we have
	\begin{align}
		\Big\|\int_0^tS_{t-s}g(\tilde{X}^{\varepsilon,\delta}_s)\mathrm{d}B^H_s\Big\|
		\leq& C \Lambda\int_0^t \big(1+\|\tilde{X}^{\varepsilon,\delta}_s\|\big)\left(s^{-\alpha}+(t-s)^{-\alpha}\right)\mathrm{d}s\notag\\
		&+
		C \Lambda\int_0^t\int_0^s\frac{\|\tilde{X}^{\varepsilon,\delta}_s-\tilde{X}^{\varepsilon,\delta}_r\|}{(s-r)^{\alpha+1}}\mathrm{d}r\mathrm{d}s.\label{eqn-3.13}
	\end{align}
	Here, the first inequality is obtained by Remark \ref{re2.1}, Assumption \textbf{(A3)}, $\eqref{eqn-3.1}$ and $\eqref{eqn-3.3}$.
	We obtain the second inequality by using Fubini's theorem and Lemma \ref{lem2.1}.
	Then, from Lemma \ref{lem2.2} and  \eqref{eqn-3.13}, for $\varepsilon\in(0,1)$, we have
	\begin{align}
		I_2
		\leq C\Lambda\left(1+\rho^{\alpha-1}\|\tilde{X}^{\varepsilon,\delta}\|_{\rho,T}+\rho^{-1}\|\tilde{X}^{\varepsilon,\delta}\|_{1,\rho,T}\right).\label{eqn-3.14}
	\end{align}
	According to Lemma \ref{lem4.2}, we have
	$$
	\Big\|\int_0^tS_{t-s} g(\tilde{X}^{\varepsilon,\delta}_s) \mathrm{d}u^\varepsilon_s\Big\|
	\leq\Lambda^{0,T}_{\alpha,u^\varepsilon}\sup_{i\in\mathbb{N}} \|S_{t-\cdot} g(\tilde{X}^{\varepsilon,\delta}_\cdot)e_i\|_{\alpha,1}.
	$$
	Note that
	$\Lambda^{0,T}_{\alpha,u^\varepsilon}\leq C\|u^\varepsilon\|_{\mathcal{H}^H}<\infty$, $\mathbb{P}$-a.s.
	Similar to the calculation of $I_2$, from Assumption \textbf{(A3)},  $\eqref{eqn-3.1}$ and $\eqref{eqn-3.3}$, we have
	\begin{align}
		I_3
		\leq C \|u^\varepsilon\|_{\mathcal{H}^H}
		\left(1+\rho^{\alpha-1}\|\tilde{X}^{\varepsilon,\delta}\|_{\rho,T}+\rho^{-1}\|\tilde{X}^{\varepsilon,\delta}\|_{1,\rho,T}\right).\label{eqn-3.15}
	\end{align}
	Substituting  $\eqref{eqn-3.12}$-$\eqref{eqn-3.15}$ into  $\eqref{eqn-3.11}$ leads to
	\begin{equation}\label{eqn-3.16}
		\|\tilde{X}^{\varepsilon,\delta}\|_{\rho,T}\leq C(\Lambda+\|u^\varepsilon\|_{\mathcal{H}^H}) \left(1+\rho^{\alpha-1}\|\tilde{X}^{\varepsilon,\delta}\|_{\rho,T}+\rho^{-1}\|\tilde{X}^{\varepsilon,\delta}\|_{1,\rho,T}\right).
	\end{equation}
	
	It proceeds to estimate $\|\tilde{X}^{\varepsilon,\delta}\|_{1,\rho,T}$. Firstly, we give some prior estimates of
	$A_1:=\|\int_s^t S_{t-r} g(\tilde{X}^{\varepsilon,\delta}_r) \mathrm{d}B^H_r\|$
	and
	$A_2:=\|\int_0^s \left(S_{t-r}-S_{s-r} \right) g(\tilde{X}^{\varepsilon,\delta}_r) \mathrm{d}B^H_r\|.$
	By similar steps as for the terms $\mathcal{K}_1(s,t)$ and $\mathcal{K}_2(s,t)$ in \cite[Lemma 3.4]{Pei2020}, from Remark \ref{re2.1}, Assumption \textbf{(A3)},  $\eqref{eqn-3.1}$ and $\eqref{eqn-3.3}$, 
	we have
	\begin{align*}
		A_1
		\leq&C\Lambda^{0,t}_{\alpha,B^H} \int_s^t \Big( \big(1+\|\tilde{X}^{\varepsilon,\delta}_r\|\big)\big((r-s)^{-\alpha}+(t-r)^{-\alpha}\big)\mathrm{d}r+  \int_s^r \frac{\|\tilde{X}^{\varepsilon,\delta}_r-\tilde{X}^{\varepsilon,\delta}_q\|}{(r-q)^{\alpha+1}}\mathrm{d}q \Big)\mathrm{d}r.
	\end{align*}
	By Remark \ref{re2.1}, Assumption \textbf{(A3)},  $\eqref{eqn-3.2}$ and $\eqref{eqn-3.4}$,	for any $\gamma\in(\frac12,1-\alpha)$,  we have
	\begin{align*}
		A_2
		\leq &C\Lambda^{0,s}_{\alpha,B^H} (t-s)^{\gamma} \int_0^s  \left(1+\|\tilde{X}^{\varepsilon,\delta}_r\|\right)\left((s-r)^{-\gamma}r^{-\alpha}+(s-r)^{-\alpha-\gamma}\right)\mathrm{d}r\\
		&+C\Lambda^{0,s}_{\alpha,B^H} (t-s)^{\gamma} \int_0^s  (s-r)^{-\gamma}\int_0^r \frac{\|\tilde{X}^{\varepsilon,\delta}_r-\tilde{X}^{\varepsilon,\delta}_q\|}{(r-q)^{\alpha+1}}\mathrm{d}q\mathrm{d}r	.
	\end{align*}
	Through direct calculation, we obtain
	\begin{align}
		\|\tilde{X}^{\varepsilon,\delta}\|_{1,\rho,T}
		\leq&\sup_{t\in[0,T]}e^{-\rho t}\int_0^t \frac{\|(S_t-S_s)X_0\|}{(t-s)^{\alpha+1}}\mathrm{d}s \notag\\
		&+ \sup_{t\in[0,T]}e^{-\rho t}\int_0^t \frac{\|\int_0^t S_{t-r}b(\tilde{X}^{\varepsilon,\delta}_r,\tilde{Y}^{\varepsilon,\delta}_r)  \mathrm{d}r-\int_0^s S_{s-r}b(\tilde{X}^{\varepsilon,\delta}_r,\tilde{Y}^{\varepsilon,\delta}_r) \mathrm{d}r\|}{(t-s)^{ \alpha+1}}\mathrm{d}s\notag\\
		&+ \sup_{t\in[0,T]}\sqrt\varepsilon  e^{-\rho t}\int_0^t \frac{\|\int_0^t S_{t-r}g(\tilde{X}^{\varepsilon,\delta}_r)  \mathrm{d}B^H_r-\int_0^s S_{s-r}g(\tilde{X}^{\varepsilon,\delta}_r)  \mathrm{d}B^H_r\|}{(t-s)^{ \alpha+1}}\mathrm{d}s\notag\\ 
		&+ \sup_{t\in[0,T]}e^{-\rho t}\int_0^t \frac{\|\int_0^t S_{t-r}g(\tilde{X}^{\varepsilon,\delta}_r)  \mathrm{d}u^\varepsilon_r-\int_0^s S_{s-r}g(\tilde{X}^{\varepsilon,\delta}_r)  \mathrm{d}u^\varepsilon_r\|}{(t-s)^{ \alpha+1}}\mathrm{d}s\notag\\  
		=:&H_1+H_2+H_3+H_4. \label{eqn-3.17}
	\end{align}
	For the term $H_1$, by  $\eqref{eqn-3.2}$, we deduce that
	\begin{align}
		H_1
		\leq \sup_{t\in[0,T]}e^{-\rho t}\int_0^t \frac{\|X_0\|_{V_\beta}}{(t-s)^{\alpha+1-\beta}}\mathrm{d}s
		\leq C\|X_0\|_{V_\beta}. \label{eqn-3.18}
	\end{align}
	From Assumption \textbf{(A6)}, Fubini's theorem,  $\eqref{eqn-3.1}$ and $\eqref{eqn-3.2}$, for any $\gamma\in(\frac12,1-\alpha)$, we have 
	\begin{align}
		H_2
		\leq& C\big(1+\|\tilde{X}^{\varepsilon,\delta}\|_{\rho,T}\big) \sup_{t\in[0,T]}\int_0^t e^{-\rho (t-r)}  \int_0^r (t-s)^{-1-\alpha} \mathrm{d}s\mathrm{d}r\notag\\ 
		&+C\big(1+\|\tilde{X}^{\varepsilon,\delta}\|_{\rho,T}\big)\sup_{t\in[0,T]}\int_0^t (t-s)^{\gamma-1-\alpha} \int_0^s (s-r)^{-\gamma} e^{-\rho (t-r)}   \mathrm{d}r\mathrm{d}s\notag\\
		\leq& C \big(1+ \rho^{2\alpha-1}\|\tilde{X}^{\varepsilon,\delta}\|_{\rho,T}\big). \label{eqn-3.19}
	\end{align}
	For the term $H_3$, from the estimate conclusion of $A_1$ and $A_2$, for any $\gamma\in(\frac12,1-\alpha)$, we have 
	\begin{align}
		H_3
		\leq&C\Lambda
		\sup_{t\in[0,T]} \int_0^t e^{-\rho t}  \frac  {\int_s^t (1+\|\tilde{X}^{\varepsilon,\delta}_r\|)\left((r-s)^{-\alpha}+(t-r)^{-\alpha}\right)\mathrm{d}r}{(t-s)^{\alpha+1}}\mathrm{d}s\notag\\
		&+C\Lambda
		\sup_{t\in[0,T]} \int_0^t e^{-\rho t}  (t-s)^{-\alpha-1} \int_s^t\int_s^r \frac{\|\tilde{X}^{\varepsilon,\delta}_r-\tilde{X}^{\varepsilon,\delta}_q\|}{(r-q)^{\alpha+1}}\mathrm{d}q\mathrm{d}r\mathrm{d}s\notag\\
		&+C\Lambda
		\sup_{t\in[0,T]} \int_0^t e^{-\rho t}  \frac {\int_0^s(1+\|\tilde{X}^{\varepsilon,\delta}_r\| )\left((s-r)^{-\gamma}r^{-\alpha}+(s-r)^{-\alpha-\gamma}\right)\mathrm{d}r}{(t-s)^{\alpha+1-\gamma}}\mathrm{d}s\notag\\
		&+C\Lambda 
		\sup_{t\in[0,T]} \int_0^t e^{-\rho t}  (t-s)^{\gamma-\alpha-1} \int_0^s (s-r)^{-\gamma}\int_0^r \frac{\|\tilde{X}^{\varepsilon,\delta}_r-\tilde{X}^{\varepsilon,\delta}_q\|}{(r-q)^{\alpha+1}}\mathrm{d}q\mathrm{d}r\mathrm{d}s\notag\\
		=:&H_{31}+H_{32}+H_{33}+H_{34}. \label{eqn-3.20}
	\end{align}
	By Lemmas \ref{lem2.1}, \ref{lem2.2} and Fubini's theorem, we deduce that
	\begin{align}
		H_{31}+H_{33}
		\leq C\Lambda   
		\left(1+ \rho^{2\alpha-1}\|\tilde{X}^{\varepsilon,\delta}\|_{\rho,T}   \right),
		\label{eqn-3.21}
	\end{align}
	and
	\begin{align}
		H_{32}+H_{34}
		\leq&C\Lambda \rho^{\alpha-1}\|\tilde{X}^{\varepsilon,\delta}\|_{1,\rho,T}. \label{eqn-3.22}
	\end{align}
	Combining  $\eqref{eqn-3.20}$-$\eqref{eqn-3.22}$, we have
	\begin{align}
		H_3\leq C\Lambda \left(1+\rho^{2\alpha-1}\|\tilde{X}^{\varepsilon,\delta}\|_{\rho,T}+\rho^{\alpha-1}\|\tilde{X}^{\varepsilon,\delta}\|_{1,\rho,T}\right). \label{eqn-3.23}
	\end{align}
	
	Similar to the calculation of $A_1$ and $A_2$, from  $\eqref{eqn-3.1}$-$\eqref{eqn-3.3}$, Assumption \textbf{(A3)} and Lemma \ref{lem4.2}, for any $\gamma\in(\frac12,1-\alpha)$, we can obtain
	\begin{align}
		\Big\|\int_s^t S_{t-r} g(\tilde{X}^{\varepsilon,\delta}_r) \mathrm{d}u^\varepsilon_r\Big\|
		\leq &C\Lambda^{0,t}_{\alpha,u^\varepsilon}  \int_s^t  (1+\|\tilde{X}^{\varepsilon,\delta}_r\|)\left((r-s)^{-\alpha}+(t-r)^{-\alpha}\right) \mathrm{d}r\notag\\
		&+C\Lambda^{0,t}_{\alpha,u^\varepsilon}  \int_s^t    \int_s^r \frac{\|\tilde{X}^{\varepsilon,\delta}_r-\tilde{X}^{\varepsilon,\delta}_q\|}{(r-q)^{1+\alpha}}\mathrm{d}q\mathrm{d}r,\label{eqn-4.50}
	\end{align}
	and
	\begin{align}
		\Big\|\int_0^s \left(S_{t-r}-S_{s-r} \right) g(\tilde{X}^{\varepsilon,\delta}_r) \mathrm{d}u^\varepsilon_r\Big\|
		\leq&C\Lambda^{0,s}_{\alpha,u^\varepsilon}  \int_0^s  \big(1+\|\tilde{X}^{\varepsilon,\delta}_r\|\big)\frac{r^{-\alpha}+(s-r)^{-\alpha}}{(t-s)^{-\gamma}(s-r)^{\gamma}}\mathrm{d}r\notag\\
		&+C\Lambda^{0,s}_{\alpha,u^\varepsilon}  \int_0^s  \frac{(t-s)^{\gamma}}{(s-r)^{\gamma}}\int_0^r \frac{\|\tilde{X}^{\varepsilon,\delta}_r-\tilde{X}^{\varepsilon,\delta}_q\|}{(r-q)^{1+\alpha}}\mathrm{d}q\mathrm{d}r.\label{eqn-4.51}
	\end{align}
	Then, according to the fact that $\Lambda^{0,T}_{\alpha,u^\varepsilon}\leq C\|u^\varepsilon\|_{\mathcal{H}^H}$, we have
	\begin{align}
		H_4
		\leq C\|u^\varepsilon\|_{\mathcal{H}^H}\big(1+\rho^{2\alpha-1}\|\tilde{X}^{\varepsilon,\delta}\|_{\rho,T}+\rho^{\alpha-1}\|\tilde{X}^{\varepsilon,\delta}\|_{1,\rho,T}\big). \label{eqn-3.24}
	\end{align}
	Combining  $\eqref{eqn-3.17}$-$\eqref{eqn-3.19}$, $\eqref{eqn-3.23}$ and $\eqref{eqn-3.24}$, we have
	\begin{equation}\label{eqn-3.25}
		\|\tilde{X}^{\varepsilon,\delta}\|_{1,\rho,T}\leq C(\Lambda+\|u^\varepsilon\|_{\mathcal{H}^H}) \big(1+\rho^{2\alpha-1}\|\tilde{X}^{\varepsilon,\delta}\|_{\rho,T}+\rho^{\alpha-1}\|\tilde{X}^{\varepsilon,\delta}\|_{1,\rho,T}\big).
	\end{equation}

	Putting $\rho=\big(4C(\Lambda+\|u^\varepsilon\|_{\mathcal{H}^H})\big)^{\frac{1}{1-\alpha}}$, we get from the inequality $\eqref{eqn-3.16}$ that
	\begin{equation}\label{eqn-3.26}
		\|\tilde{X}^{\varepsilon,\delta}\|_{\rho,T}\leq \frac43 C(\Lambda+\|u^\varepsilon\|_{\mathcal{H}^H}) \big(1+\rho^{-1}\|\tilde{X}^{\varepsilon,\delta}\|_{1,\rho,T}\big).
	\end{equation}
	Plugging \eqref{eqn-3.26} into $\eqref{eqn-3.25}$ and making simple transformations, we have
	$$\|\tilde{X}^{\varepsilon,\delta}\|_{1,\rho,T}\leq\frac32C(\Lambda+\|u^\varepsilon\|_{\mathcal{H}^H})+2\big(C(\Lambda+\|u^\varepsilon\|_{\mathcal{H}^H})\big)^{\frac1{1-\alpha}}\leq C(\Lambda+\|u^\varepsilon\|_{\mathcal{H}^H})^{\frac1{1-\alpha}}.$$
	Substituting the inequality above into  $\eqref{eqn-3.26}$, we get
	$\|\tilde{X}^{\varepsilon,\delta}\|_{\rho,T}\leq  C(\Lambda+\|u^\varepsilon\|_{\mathcal{H}^H})^{\frac1{1-\alpha}}.$
	Thus, we have
	$$
	\begin{aligned}
		\|\tilde{X}^{\varepsilon,\delta}\|_{\alpha,\infty}\leq&e^{\rho T}\big(\|\tilde{X}^{\varepsilon,\delta}\|_{\rho,T}+\|\tilde{X}^{\varepsilon,\delta}\|_{1,\rho,T}\big)\\
		\leq&C\exp\big(C(\Lambda+\|u^\varepsilon\|_{\mathcal{H}^H})^{\frac{1}{1-\alpha}}\big)(\Lambda+\|u^\varepsilon\|_{\mathcal{H}^H})^{\frac{1}{1-\alpha}}.
	\end{aligned}
	$$
	According to the fact that $0<\frac{1}{1-\alpha}<2$, assumption $(u^\varepsilon,v^\varepsilon)\in\mathcal{A}_b^N$ and Fernique's theorem (cf. \cite{Prato2014}), 
	the statement follows.
\end{proof}

\begin{lemma}\label{lem4.4}
	Suppose that Assumptions \textnormal{\textbf{(A1)}-\textbf{(A3)}} and \textnormal{\textbf{(A6)}} hold. 
	Let $N\in\mathbb{N}$. Then, there exists a constant $C>0$, such that  for any $(u^\varepsilon,v^\varepsilon)\in\mathcal{A}_b^N$, $\beta\in\left(\frac12,1-\alpha\right)$ and $0\leq s<t\leq T$, we have
	$$\mathbb{E}\left[\|\tilde{X}^{\varepsilon,\delta}_t-\tilde{X}^{\varepsilon,\delta}_s\|^2\right]\leq C|t-s|^{2\beta }.$$
	Here, $C$ is a constant which only depends on $T$ and $N$.
\end{lemma}
\begin{proof}
	In this proof, $C$ is a positive constant depending on only $T$ and $N$, which may
	change from line to line. 
	Firstly, we also denote $\Lambda:=\Lambda_{\alpha,B^H}^{0,T}\vee1$. Through direct calculation, we obtain
	\begin{align}
		\|\tilde{X}^{\varepsilon,\delta}_t-\tilde{X}^{\varepsilon,\delta}_s\|
		\leq&C(t-s)^\beta\|X_0\|_{V_\beta} \notag\\
		&+\Big\| \int_s^t S_{t-r}b(\tilde{X}^{\varepsilon,\delta}_r,\tilde{Y}^{\varepsilon,\delta}_r)\mathrm{d}r + \int_0^s(S_{t-r}-S_{s-r})b(\tilde{X}^{\varepsilon,\delta}_r,\tilde{Y}^{\varepsilon,\delta}_r)\mathrm{d}r \Big\|\notag\\
		&+\sqrt\varepsilon\Big\| \int_s^t S_{t-r}g(\tilde{X}^{\varepsilon,\delta}_r)\mathrm{d}B^H_r + \int_0^s(S_{t-r}-S_{s-r})g(\tilde{X}^{\varepsilon,\delta}_r)\mathrm{d}B^H_r \Big\|\notag\\
		&+\Big\| \int_s^t S_{t-r}g(\tilde{X}^{\varepsilon,\delta}_r)\mathrm{d}u^\varepsilon_r + \int_0^s(S_{t-r}-S_{s-r})g(\tilde{X}^{\varepsilon,\delta}_r)\mathrm{d}u^\varepsilon_r \Big\|\notag\\
		=:&C(t-s)^\beta\|X_0\|_{V_\beta}+B_1+B_2+B_3.\label{eqn-3.27}
	\end{align}
	By  $\eqref{eqn-3.1}$, $\eqref{eqn-3.2}$,  Assumption \textbf{(A6)} and  \eqref{eqn-3.12} in Lemma \ref{lem4.3}, we deduce that 
	\begin{align}
		B_1
		\leq& C  \Big(1+\sup_{t\in[0,T]}\|\tilde{X}^{\varepsilon,\delta}_t\| \Big)   \left((t-s)+ (t-s)^\beta \int_0^s (s-r)^{-\beta} \mathrm{d}r  \right)\notag \\
		\leq& C \big(1+\|\tilde{X}^{\varepsilon,\delta}\|_{\alpha,\infty}\big)\left((t-s)+(t-s)^\beta\right). \label{eqn-3.29}
	\end{align}
	From Lemma \ref{lem2.1}, the estimates of $A_1$ and $A_2$ in Lemma \ref{lem4.3} and $\varepsilon\in(0,1)$, we have
	\begin{align}
		B_2
		\leq&C\Lambda  \int_s^t  \big(1+\|\tilde{X}^{\varepsilon,\delta}_r\|\big)\left((r-s)^{-\alpha}+(t-r)^{-\beta}(r-s)^{\beta-\alpha}\right) \mathrm{d}r\notag\\
		&+C\Lambda \int_s^t \int_s^r \frac{\|\tilde{X}^{\varepsilon,\delta}_r-\tilde{X}^{\varepsilon,\delta}_q\|}{(r-q)^{\alpha+1}}\mathrm{d}q\mathrm{d}r\notag\\
		&+C\Lambda (t-s)^\beta \int_0^s  \big(1+\|\tilde{X}^{\varepsilon,\delta}_r\|\big)\left((s-r)^{-\beta}r^{-\alpha}+(s-r)^{-\alpha-\beta}\right)\mathrm{d}r\notag\\
		&+C\Lambda (t-s)^\beta \int_0^s(s-r)^{-\beta}\int_0^r \frac{\|\tilde{X}^{\varepsilon,\delta}_r-\tilde{X}^{\varepsilon,\delta}_q\|}{(r-q)^{\alpha+1}}\mathrm{d}q\mathrm{d}r\notag\\
		\leq&C\Lambda\big(1+\|\tilde{X}^{\varepsilon,\delta}\|_{\alpha,\infty}\big)\big((t-s)^{1-\alpha}+(t-s)^\beta+(t-s)\big). \label{eqn-3.30}
	\end{align}

	Similar to the calculation of  $\eqref{eqn-3.30}$, we obtain
	\begin{align}
		B_3
		\leq C\Lambda^{0,t}_{\alpha,u^\varepsilon}\left(1+\|\tilde{X}^{\varepsilon,\delta}\|_{\alpha,\infty}\right)\left((t-s)^{1-\alpha}+(t-s)^\beta+(t-s)\right).\label{eqn-3.31}
	\end{align}

	Without loss of generality, let $0<t-s<1$. Since  $\beta\in\left(\frac12,1-\alpha\right)$ and $(u^\varepsilon,v^\varepsilon)\in\mathcal{A}_b^N$, from  $\eqref{eqn-3.27}$-$\eqref{eqn-3.31}$, H\"{o}lder's inequality and Lemma \ref{lem4.3}, we have
	
	$$
	\begin{aligned}
		\mathbb{E}\big[\|\tilde{X}^{\varepsilon,\delta}_t-\tilde{X}^{\varepsilon,\delta}_s\|^2\big]
		\leq&C|t-s|^{2 \beta}.
	\end{aligned}
	$$
	The conclusion is verified.
\end{proof}


\begin{lemma}\label{lem4.5}
	Suppose that Assumptions \textnormal{\textbf{(A1)}-\textbf{(A3)}} and \textnormal{\textbf{(A6)}} hold. Let $N\in\mathbb{N}$ and $t\in[0,T]$. Then, for any $(u^\varepsilon,v^\varepsilon)\in\mathcal{A}_b^N$, $\sup_{0\leq s\leq t}\|\tilde{Y}^{\varepsilon, \delta}_s\|$ has moments of all orders.
\end{lemma}

We provide the complete proof of Lemma \ref{lem4.5} in Appendix \ref{secA.4}.

\begin{lemma}\label{lem4.6}
	Suppose that Assumptions \textnormal{\textbf{(A1)}-\textbf{(A7)}} hold. Let $N\in\mathbb{N}$. Then, for any $(u^\varepsilon,v^\varepsilon)\in\mathcal{A}_b^N$, there exists a constant $C>0$, such that 
	$$\int_0^T \mathbb{E}\big[\|\tilde{Y}^{\varepsilon,\delta}_t\|^2\big]\mathrm{d}t \leq C.$$
	Here, $C$ is a positive constant which only  depends on $T$ and $N$.
\end{lemma}

The proof is reported in Appendix \ref{secA.5}.

\section{Proof of Main Theorem (Theorem \ref{th3.1})}\label{5}

In this section, we establish the LDP for the slow-fast system \eqref{eqn-1.2}. To this end, we first introduce a general criterion for the LDP in Appendix \ref{A}. 
The main results are then derived in Theorem \ref{th3.1} by verifying \textit{Conditions (1)} and \textit{(2)} in Lemma \ref{lemA.1}.

\begin{proof}[Proof of Condition (1) in Lemma \ref{lemA.1}]
	
For any fixed $N\in\mathbb{N}$,
let $\{(u^n,v^n)\}_{n\geq1}\subset\mathcal{S}_N$ be a sequence, and denote by $\{\bar{X}^{u^n}=\mathcal{G}^{0}(u^n,v^n)\}_{n\geq1}$ a sequence of elements in $C([0,T],V)$. Since $\mathcal{S}_N$ is a compact Polish space, the sequence $\{(u^n,v^n)\}_{n\geq1}$ has a subsequence that converges in the weak topology. Furthermore, we have  $\sup_{n\geq1}\|(u^n,v^n)\|_{\mathcal{H}}\leq\sqrt{2N}<\infty$. Hence, by calculating similar to Lemmas \ref{lem4.3} and \ref{lem4.4}, for $0\leq s\leq t\leq T$,  we obtain the estimates 
$\sup_{n\geq 1}\|\bar{X}^{u^n}\|_{\alpha,\infty}\leq C$ and
$	\sup_{n\geq 1}\|\bar{X}^{u^n}_t-\bar{X}^{u^n}_s\|^2
\leq C(t-s)^{2\beta}$,
where $C$ is a positive constant  depending only on $T$ and $N$.
Therefore, there exists a subsequence of $\{(u^n,v^n)\}_{n\geq1}$ (still labeled by $\{n\}$) and an element $(u,v)\in\mathcal{S}_N$ such that
\begin{itemize} 
	\item $(u^n,v^n)\to (u,v)$ in $\mathcal{S}_N$ weakly, as $n\to\infty$;
	\item $\sup_{n\geq 1}\sup_{t\in[0,T]}\|\bar{X}^{u^n}_t\|<\infty$.
\end{itemize}

Next, we will prove that  $\bar{X}^{u^n}\to \bar{X}^{u}$ in $C([0,T],V)$ as $n\to \infty$. By Fubini's Theorem and \eqref{eqn-2.6}, it follows that
\begin{align}
	\|\bar{X}^{u^n}_t-\bar{X}^u_t\|
	\leq&
	\|\int_0^tS_{t-s}\bar{b}(\bar{X}^{u^n}_s)ds-\int_0^tS_{t-s}\bar{b}(\bar{X}^u_s)ds\|\notag\\
	&+\|\int_0^tS_{t-s}g(\bar{X}^{u^n}_s)du^n_s-\int_0^tS_{t-s}g(\bar{X}^u_s)du_s\|\notag\\
	\leq&\int_0^t\|\bar{X}^{u^n}_s-\bar{X}^u_s\|\Big(1+s^{H-\frac12}\int_0^s(s-r)^{H-\frac32}r^{\frac12-H}\|\dot{u}^n_r Q_1^{-\frac12}\|dr\Big)ds\notag\\
	&+C\Big\|\int_0^t r^{\frac12-H} Q_1^{-\frac12}(\dot{u}^n_r-\dot{u}_r) \int_r^t S_{t-s}g(\bar{X}^u_s)Q_1^{\frac12}s^{H-\frac12}(s-r)^{H-\frac32}\mathrm{d}s\mathrm{d}r  \Big\|\notag\\
	=:& \int_0^t\|\bar{X}^{u^n}_s-\bar{X}^u_s\|\Big(1+s^{H-\frac12}\int_0^s(s-r)^{H-\frac32}r^{\frac12-H}\|\dot{u}^n_r Q_1^{-\frac12}\|dr\Big)ds
	+P.\notag
\end{align}
	
It proceeds to estimate the term $P$. To this end, we first consider  $f^t(r):=\textbf{1}_{[0,t]}(r)\int_r^t S_{t-s}g(\bar{X}^u_s)Q_1^{\frac12}s^{H-\frac12}(s-r)^{H-\frac32}\mathrm{d}s$ as a function of $r$. From Assumption \textbf{(A3)}, we obtain the bound
$$\|f^t(r)\|\leq C (1+\|\bar{X}^u\|_\infty)\textbf{1}_{[0,t]}(r)\int_r^t s^{H-\frac12}(s-r)^{H-\frac32}\mathrm{d}s\leq C(1+\|\bar{X}^u\|_\infty).$$
This implies that the function $r\mapsto \underline{}r^{\frac12-H}f^t(r)$ belongs to $L^2([0,T],V)$. From the weak convergence of $(u^n,v^n)$ to $(u,v)$, we conclude that $\lim_{n\to\infty}P=0$ for any $t\in [0,T]$. Then, applying Gr\"{o}nwall's inequality, H\"{o}lder's inequality and Fubini's theorem, we have
\begin{align}
	\|\bar{X}^{u^n}_t-\bar{X}^u_t\|
	\leq& P\exp\Big(\int_0^t 1 ds+\int_0^ts^{H-\frac12}\int_0^s (s-r)^{H-\frac32}r^{\frac12-H}\|\dot{u}^n_r\|_1drds\Big) \notag\\
	\leq &C P\exp\Big( 1+\int_0^t r^{\frac12-H}\|\dot{u}^n_r\|_1dr\Big)\notag\\
	\leq &C P\exp \Big(1+\Big(\int_0^tr^{1-2H}dr\Big)^{\frac12}\Big(\int_0^t\|\dot{u}^n_r\|_1^2dr\Big)^{\frac12}\Big)\notag\\
	\leq & CP \to 0,\quad n\to\infty,\notag
\end{align}
where $C$ is a positive constant which only  depends on $T$ and $N$.

Thus, by a standard subsequential argument, we see that the full sequence $\{\bar{X}^{u^n}=\mathcal{G}^{0}(u^n,v^n)\}_{n\geq 1}$ converges to $\bar{X}^{u}=\mathcal{G}^{0}(u,v)$ in $C([0,T],V)$. This implies that $\Gamma_N:=\{\mathcal{G}^{0}(u,v):(u,v)\in\mathcal{S}_N\}$ is compact in $C([0,T],V)$. 
\end{proof}

\begin{proof}[Proof of \textit{Condition (2)} in Lemma \ref{lemA.1}]
Before proving the sequence $\{\tilde{X}^{\varepsilon, \delta}\}_{\varepsilon\in(0,1)}$ weakly converges to $\bar{X}^u$, we construct the auxiliary processes.	Taking any sequence $\{(u^\varepsilon,v^\varepsilon)\}_{\varepsilon\in(0,1)}\subset \mathcal{A}_b^N$ such that $\{(u^\varepsilon,v^\varepsilon)\}_{\varepsilon\in(0,1)}$ converges to $(u,v)\in\mathcal{S}_N$ in distribution as $\varepsilon\to 0$, the controlled system are defined by the following SPDEs,
$$
\begin{cases}
	\mathrm{d}\hat{X}^{\varepsilon,\delta}_t=(A\hat{X}^{\varepsilon,\delta}_t+b(\tilde{X}^{\varepsilon,\delta}_{t(\Delta)},\hat{Y}^{\varepsilon,\delta}_t))\mathrm{d}t+g(\hat{X}^{\varepsilon,\delta}_t)\mathrm{d}u^\varepsilon_t,\\  
	\mathrm{d}\hat{Y}^{\varepsilon,\delta}_t=\frac1\delta(A\hat{Y}^{\varepsilon,\delta}_t+F(\tilde{X}^{\varepsilon,\delta}_{t(\Delta)},\hat{Y}^{\varepsilon,\delta}_t)\mathrm{d}t
	+\frac{1}{\sqrt\delta} G(\tilde{X}^{\varepsilon,\delta}_{t(\Delta)},\hat{Y}^{\varepsilon,\delta}_t)\mathrm{d}W_t,\\
	\hat{X}^{\varepsilon,\delta}_0=X_0,~\hat{Y}^{\varepsilon,\delta}_0=Y_0,~t\in[0,T],
\end{cases}
$$
where $t(\Delta)=\lfloor\frac t \Delta\rfloor\Delta$ is the nearest breakpoint preceding $t$ with $t\in[0,T]$. Without loss of generality, we assume $\Delta<1$.

We derive the following error estimate between the process $\tilde{Y}^{\varepsilon,\delta}$ and $\hat{Y}^{\varepsilon,\delta}$.

\begin{lemma}\label{lem5.2}
	There exists a constant $C>0$ such that for any $x,y\in V$ and $\varepsilon,\delta\in(0,1)$, we have
	$$
	\int_0^T\mathbb{E}\big[\|\tilde{Y}^{\varepsilon,\delta}_t-\hat{Y}^{\varepsilon,\delta}_t\|^2 \big]\mathrm{d}t\leq C\Delta+C\frac\delta\varepsilon.
	$$
	Here, $C$ is a  constant which only depends  on $T$ and $N$.
\end{lemma}
\begin{proof}
	Let $J_t:=\tilde{Y}^{\varepsilon,\delta}_t-\hat{Y}^{\varepsilon,\delta}_t$. Applying It\^{o}'s formula to $\|J_t\|^2$ and taking Assumptions \textbf{(A1)}-\textbf{(A7)} into account, we obtain
\begin{align}
	\frac{\mathrm{d}\|J_t\|^2}{\mathrm{d}t}=&\frac2\delta \left\langle  J_t,A J_t    \right\rangle  
	+\frac2\delta \big\langle  J_t, F(\tilde{X}^{\varepsilon,\delta}_t,\tilde{Y}^{\varepsilon,\delta}_t) - F(\tilde{X}^{\varepsilon,\delta}_{t(\Delta)},\hat{Y}^{\varepsilon,\delta}_t) \big\rangle \notag\\
	&+\frac{2}{\sqrt{\delta\varepsilon}}  \Big\langle  J_t,  G(\tilde{X}^{\varepsilon,\delta}_t,\tilde{Y}^{\varepsilon,\delta}_t) \frac{\mathrm{d}v^\varepsilon_t}{\mathrm{d}t} \Big\rangle  
	+\frac1\delta \big\|   G(\tilde{X}^{\varepsilon,\delta}_t,\tilde{Y}^{\varepsilon,\delta}_t) - G(\tilde{X}^{\varepsilon,\delta}_{t(\Delta)},\hat{Y}^{\varepsilon,\delta}_t)    \big\|^2_{HS}   \notag\\
	&+ \frac{2}{\sqrt{\delta}}  \Big\langle  J_t,  \big(G(\tilde{X}^{\varepsilon,\delta}_t,\tilde{Y}^{\varepsilon,\delta}_t) - G(\tilde{X}^{\varepsilon,\delta}_{t(\Delta)},\hat{Y}^{\varepsilon,\delta}_t)\big) \frac{\mathrm{d}W_t}{\mathrm{d}t}   \Big\rangle \notag \\ 
	\leq&\frac{-(2\bar\lambda_1-2\beta_3-C_2-1)}{\delta}\|J_t\|^2
	+\frac{C}{\delta}\| \tilde{X}^{\varepsilon,\delta}_t- \tilde{X}^{\varepsilon,\delta}_{t(\Delta)} \|^2
	\notag\\
	&+\frac{C}{\varepsilon}\big(1+\|\tilde{X}^{\varepsilon,\delta}_t\|^2\big)\Big\|\frac{\mathrm{d}v^\varepsilon_t}{\mathrm{d}t}\Big\|^2_2
	+\frac{2}{\sqrt\delta} \Big\langle  J_t,  \big(G(\tilde{X}^{\varepsilon,\delta}_t,\tilde{Y}^{\varepsilon,\delta}_t) - G(\tilde{X}^{\varepsilon,\delta}_{t(\Delta)},\hat{Y}^{\varepsilon,\delta}_t)\big) \frac{\mathrm{d}W_t}{\mathrm{d}t}   \Big\rangle  ,\notag
\end{align}
where we use 
Young's inequality in the last step. According to Lemmas \ref{lem4.3} and \ref{lem4.5}, we obtain that the fourth term of the final inequality is a true martingale. In particular, we have $\mathbb{E}[\int_0^t \langle \tilde{Y}^{\varepsilon,\delta}_s,G(\tilde{X}^{\varepsilon,\delta}_s,\tilde{Y}^{\varepsilon,\delta}_t) - G(\tilde{X}^{\varepsilon,\delta}_{s(\Delta)},\hat{Y}^{\varepsilon,\delta}_s) \rangle \mathrm{d}W_s]=0$.
From Assumption \textbf{(A5)}, we have $\eta=2\bar\lambda_1-2\beta_3-C_2>1$. By the comparison theorem, we have
\begin{align*}
	\mathbb{E}\big[\|J_t\|^2\big]\leq& \frac{C}{\delta}\int_0^t\mathbb{E}\big[ \| \tilde{X}^{\varepsilon,\delta}_s- \tilde{X}^{\varepsilon,\delta}_{s(\Delta)} \|^2 \big]  e^{\frac{1-\eta}{\delta}(t-s)}\mathrm{d}s\\
	&+\frac{C}{\varepsilon}\int_0^t  \mathbb{E}\bigg[\left(1+\|\tilde{X}^{\varepsilon,\delta}_s\|^2\right)\left\|\frac{\mathrm{d}v
	^\varepsilon_s}{\mathrm{d}s}\right\|^2_2\bigg] e^{\frac{1-\eta}{\delta}(t-s)}\mathrm{d}s.
\end{align*}
Then, from Fubini's theorem, $(u^\varepsilon,v^\varepsilon)\in\mathcal{A}_b^N$, Lemmas \ref{lem4.3} and \ref{lem4.4}, we have
\begin{align*}
	\int_0^T\mathbb{E}\left[\|J_t\|^2\right]\mathrm{d}t 
	\leq& \frac{C}{\delta}\int_0^T\int_0^t\mathbb{E}\big[ \| \tilde{X}^{\varepsilon,\delta}_s- \tilde{X}^{\varepsilon,\delta}_{s(\Delta)} \|^2 \big]  e^{\frac{1-\eta}{\delta}(t-s)}\mathrm{d}s\mathrm{d}t\notag\\
	&+\frac{C}{\varepsilon}\int_0^T\int_0^t  \mathbb{E}\bigg[\left(1+\|\tilde{X}^{\varepsilon,\delta}_s\|^2\right)\left\|\frac{\mathrm{d}v
	^\varepsilon_s}{\mathrm{d}s}\right\|^2_2\bigg] e^{\frac{1-\eta}{\delta}(t-s)}\mathrm{d}s\mathrm{d}t\notag\\
	\leq& \frac C{\eta-1} \int_0^T \mathbb{E}\big[ \| \tilde{X}^{\varepsilon,\delta}_s- \tilde{X}^{\varepsilon,\delta}_{s(\Delta)} \|^2 \big] \mathrm{d}s\\
	&+\frac {C\delta}{(\eta-1)\varepsilon}   \mathbb{E}\bigg[\Big(1+\sup_{t\in[0,T]}\|\tilde{X}^{\varepsilon,\delta}_t\|^2\Big)\int_0^T \left\| \frac{\mathrm{d}v^\varepsilon_s}{\mathrm{d}s}\right\|^2_2  \mathrm{d}s\bigg]  \notag\\
	\leq& C\Delta^{2\beta}+C\frac\delta\varepsilon,\notag
\end{align*}
where $\beta\in\left(\frac12, 1-\alpha\right)$. Supposing that $\Delta$ is small enough, we have
$$	\int_0^T\mathbb{E}\left[\|J_t\|^2\right]\mathrm{d}t \leq C\Delta+C\frac\delta\varepsilon .$$  
The proof is completed.
\end{proof}
By taking the same manner in Lemmas \ref{lem4.3} and \ref{lem4.6}, for any $p\geq 1$, we have 
\begin{equation}\label{eqn-4.21}
	\mathbb{E}\big[\|\hat{X}^{\varepsilon,\delta}\|_{\alpha,\infty}^p\big]\leq C,\quad
	\int_0^T \mathbb{E}\big[\|\hat{Y}^{\varepsilon,\delta}_t\|^2\big]\mathrm{d}t \leq C,
\end{equation}
where the positive constant $C$ is independent of $\varepsilon$, $\delta$ and $\Delta$.

Let $R>0$ be large enough, we define the stopping time 
$\tau_R:=\inf\{t\geq0: \Lambda^{0,T}_{\alpha,B^H}\geq R\}\wedge T.$ 
Then, we set $A_{R,T}:=\{\Lambda^{0,T}_{\alpha,B^H}\leq  R\}$.


	Next, we show the proof of \textit{Condition (2)} in Lemma \ref{lemA.1}.
	
	\textbf{Step 1.} For any fixed $N\in\mathbb{N}$, let $(u^\varepsilon,v^\varepsilon)\in \mathcal{A}_b^N$. 
	In this step, we aim to estimate the term $\mathbb{E}\big[ \|\tilde{X}^{\varepsilon,\delta}- \hat{X}^{\varepsilon,\delta}\|^2_{\alpha,T}  \big]$.

	It is easy to see that
	\begin{align}
		\mathbb{E}\left[ \|\tilde{X}^{\varepsilon,\delta}- \hat{X}^{\varepsilon,\delta}\|^2_{\alpha,T}  \right]
		\leq& \mathbb{E}\left[ \|\tilde{X}^{\varepsilon,\delta}- \hat{X}^{\varepsilon,\delta}\|^2_{\alpha,T} \textbf{1}_{\{\tau_R< T\}} \right]\notag\\
		&+\mathbb{E}\left[ \|\tilde{X}^{\varepsilon,\delta}- \hat{X}^{\varepsilon,\delta}\|^2_{\alpha,T} \textbf{1}_{\{\tau_R\geq T\}} \right].  \label{eqn-4.22}
	\end{align}
	For the first term of \eqref{eqn-4.22}, Applying  H\"{o}lder's inequality and Markov's inequality, we indicate that
	\begin{align}
		\mathbb{E}\left[ \|\tilde{X}^{\varepsilon,\delta}- \hat{X}^{\varepsilon,\delta}\|^2_{\alpha,T} \textbf{1}_{\{\tau_R< T\}} \right]
		\leq&  \mathbb{E}\left[ \|\tilde{X}^{\varepsilon,\delta}- \hat{X}^{\varepsilon,\delta}\|^4_{\alpha,T}  \right]^\frac12  \mathbb{P}\left(\tau_R<T \right)^\frac12 \notag\\
		\leq& C R^{-\frac12}\sqrt{\mathbb{E}\big[\Lambda^{0,T}_{\alpha,B^H}\big]},  \label{eqn-4.23}
	\end{align}
	where the final inequality comes from  \eqref{eqn-4.21} and Lemma \ref{lem4.3}.
	
	We now compute the second term of 
	\eqref{eqn-4.22}. It is easy to see that,
	\begin{align}
		&\mathbb{E}\big[ \|\tilde{X}^{\varepsilon,\delta}- \hat{X}^{\varepsilon,\delta}\|^2_{\alpha,T} \textbf{1}_{\{\tau_R\geq T\}} \big]\notag\\
		\leq& C \mathbb{E}\bigg[\Big\| \int_0^\cdot S_{\cdot-s}\left(b(\tilde{X}^{\varepsilon,\delta}_s,\tilde{Y}^{\varepsilon,\delta}_s)-b(\tilde{X}^{\varepsilon,\delta}_{s(\Delta)},\hat{Y}^{\varepsilon,\delta}_s)\right)\mathrm{d}s    \Big\|^2_{\alpha,T} \textbf{1}_{A_{R,T}}   \bigg]\notag\\
		&+C \mathbb{E}\bigg[\Big\| \int_0^\cdot S_{\cdot-s}\left(g(\tilde{X}^{\varepsilon,\delta}_s)-g(\hat{X}^{\varepsilon,\delta}_s)\right)\mathrm{d}u^\varepsilon_s   \Big\|^2_{\alpha,T} \textbf{1}_{A_{R,T}}  \bigg]\notag\\
		&+\varepsilon C \mathbb{E}\bigg[\Big\| \int_0^\cdot S_{\cdot-s}g(\tilde{X}^{\varepsilon,\delta}_s) dB^H_s   \Big\|^2_{\alpha,T} \textbf{1}_{A_{R,T}}   \bigg]  \notag\\
		=:& M_1+M_2+M_3.\label{eqn-4.24}
	\end{align}
	Before computing $M_i$, $i=1,2,3$, we estimate $A_3:=\left\| \int_0^\cdot S_{\cdot-s} f\mathrm{d}s \right\|^2_{\alpha,T}$, where $f:[0,T]\to V$ is a measurable function. According to Fubini's theorem and H\"{o}lder's inequality, for any $\gamma\in(\frac12,1-\alpha)$, we conclude that
	\begin{align}
		A_3\leq& \sup_{t\in[0,T]}\bigg(\Big\| \int_0^t S_{t-r} f(r) \mathrm{d}r \Big\|  +\int_0^t \frac{\|\int_0^tS_{t-r}f(r)\mathrm{d}r-\int_0^sS_{s-r}f(r)\mathrm{d}r\|}{(t-s)^{\alpha+1}}\mathrm{d}s  \bigg)^2 \notag\\
		\leq& C \sup_{t\in[0,T]} \Big(\int_0^t(t-r)^{-\alpha}\|f(r)\|\mathrm{d}r\Big)^2  \notag\\
		\leq& C\int_0^T \|f(r)\|^2 \mathrm{d}s	,
		\label{eqn-43}
	\end{align}
	where we use Lemma \ref{lem2.1} in the second inequality. 
	Then, from Assumption \textbf{(A1)},  \eqref{eqn-4.25}, Lemmas \ref{lem4.4} and \ref{lem5.2}, we have
	\begin{align}
		M_1
		\leq& C \int_0^T\mathbb{E}\Big[\|\tilde{X}^{\varepsilon,\delta}_s-\tilde{X}^{\varepsilon,\delta}_{s(\Delta)}\|^2 \textbf{1}_{A_{R,T}}+\|\tilde{Y}^{\varepsilon,\delta}_s-\hat{Y}^{\varepsilon,\delta}_s\|^2 \textbf{1}_{A_{R,T}} \Big]\mathrm{d}s
		\leq C\Delta+C\frac\delta\varepsilon.
		\label{eqn-4.26}
	\end{align}
	
	For the term $M_2$, we have 
	\begin{align}
		M_2
		\leq& C \mathbb{E}\bigg[ \int_0^T \Big( \int_0^t (t-s)^{-\alpha-1}\Big\|\int_s^tS_{t-r}(g(\tilde{X}^{\varepsilon,\delta}_r)-g(\hat{X}^{\varepsilon,\delta}_r))\mathrm{d}u^\varepsilon_r\Big\|\mathrm{d}s  \Big)^2\mathrm{d}t \textbf{1}_{A_{R,T}}  \bigg]\notag\\
		&+	C \mathbb{E}\bigg[ \int_0^T \Big( \int_0^t \frac
		{\|\int_0^s (S_{t-r}-S_{s-r})(g(\tilde{X}^{\varepsilon,\delta}_r)-g(\hat{X}^{\varepsilon,\delta}_r))\mathrm{d}u^\varepsilon_r\|}{(t-s)^{\alpha+1}}\mathrm{d}s  \Big)^2\mathrm{d}t \textbf{1}_{A_{R,T}}  \bigg]\notag\\
		&+C \mathbb{E}\bigg[ \sup_{t\in[0,T]} \Big\| \int_0^t S_{t-s} (g(\tilde{X}^{\varepsilon,\delta}_s)-g(\hat{X}^{\varepsilon,\delta}_s))  \mathrm{d}u^\varepsilon_s\Big\|^2 \textbf{1}_{A_{R,T}}  \bigg]\notag\\
		=:& M_{21}+M_{22}+M_{23}.
		\label{eqn-4.27}
	\end{align}
	From Fubini's theorem and the fact that  $(u^\varepsilon,v^\varepsilon)\in \mathcal{A}_b^N$, we have
	\begin{align}
		M_{21}
		\leq& C\int_0^T\mathbb{E}\bigg[\bigg( \int_0^t 
		\int_s^t \Big(  
		\frac{(r-s)^{-\alpha}+(t-r)^{-\alpha}}{(t-s)^{\alpha+1}}\|\tilde{X}^{\varepsilon,\delta}_r-\hat{X}^{\varepsilon,\delta}_r\|\notag\\
		&+ \int_s^r\frac{\sup_{i\in \mathbb{N}}\big\|\big(g(\tilde{X}^{\varepsilon,\delta}_r)-g(\hat{X}^{\varepsilon,\delta}_r)-g(\tilde{X}^{\varepsilon,\delta}_q)+g(\hat{X}^{\varepsilon,\delta}_q)\big)e_i\big\|}{(t-s)^{\alpha+1}(r-q)^{\alpha+1}}\mathrm{d}q
		\Big)
		\mathrm{d}r\mathrm{d}s\bigg)^2 \textbf{1}_{A_{R,t}} \bigg] \mathrm{d}t \notag\\
		\leq& C\int_0^T\mathbb{E}\bigg[ \bigg(\int_0^t(t-r)^{-2\alpha}\|\tilde{X}^{\varepsilon,\delta}_r-\hat{X}^{\varepsilon,\delta}_r\|\mathrm{d}r\notag\\
		&+  \int_0^t(t-s)^{-\alpha-1}\int_s^t \int_s^r \frac{\|\tilde{X}^{\varepsilon,\delta}_r-\hat{X}^{\varepsilon,\delta}_r-\tilde{X}^{\varepsilon,\delta}_q+\hat{X}^{\varepsilon,\delta}_q\|}{(r-q)^{\alpha+1}}\mathrm{d}q\mathrm{d}r\mathrm{d}s   \notag\\
		&+ \int_0^t(t-s)^{-\alpha-1}\int_s^t\|\tilde{X}^{\varepsilon,\delta}_r-\hat{X}^{\varepsilon,\delta}_r\|\int_s^r\frac{\|\tilde{X}^{\varepsilon,\delta}_r-\tilde{X}^{\varepsilon,\delta}_q\|}{(r-q)^{\alpha+1}}\mathrm{d}q \mathrm{d}r\mathrm{d}s \notag\\
		&+ \int_0^t(t-s)^{-\alpha-1}\int_s^t\|\tilde{X}^{\varepsilon,\delta}_r-\hat{X}^{\varepsilon,\delta}_r\|\int_s^r\frac{\|\hat{X}^{\varepsilon,\delta}_r-\hat{X}^{\varepsilon,\delta}_q\|}{(r-q)^{\alpha+1}}\mathrm{d}q \mathrm{d}r\mathrm{d}s\bigg)^2  \textbf{1}_{A_{R,t}}\bigg]\mathrm{d}t\notag\\
		\leq&C\int_0^T\mathbb{E}\bigg[\bigg(   \int_0^t\bigg(\left((t-r)^{-2\alpha}+(t-r)^{-\alpha}\right)\|\tilde{X}^{\varepsilon,\delta}_r-\hat{X}^{\varepsilon,\delta}_r\|\notag\\
		&\times \left(1+\Delta(\tilde{X}^{\varepsilon,\delta}_r)+\Delta(\hat{X}^{\varepsilon,\delta}_r)\right)\notag\\
		&+(t-r)^{-\alpha}\int_0^r \frac{\|\tilde{X}^{\varepsilon,\delta}_r-\hat{X}^{\varepsilon,\delta}_r-\tilde{X}^{\varepsilon,\delta}_q+\hat{X}^{\varepsilon,\delta}_q\|}{(r-q)^{\alpha+1}}\mathrm{d}q \bigg)  \mathrm{d}r\textbf{1}_{A_{R,t}}\bigg)^2\bigg]\mathrm{d}t\notag\\
		\leq& C\int_0^T \mathbb{E}\bigg[t^{-2\alpha+1}\int_0^t\left((t-r)^{-\alpha}+1\right)^2  \|\tilde{X}^{\varepsilon,\delta}_r-\hat{X}^{\varepsilon,\delta}_r\|^2\mathrm{d}r \textbf{1}_{A_{R,t}}\notag\\
		&+ t^{-2\alpha+1}\int_0^t\Big(\int_0^r \frac{\|\tilde{X}^{\varepsilon,\delta}_r-\hat{X}^{\varepsilon,\delta}_r-\tilde{X}^{\varepsilon,\delta}_q+\hat{X}^{\varepsilon,\delta}_q\|}{(r-q)^{\alpha+1}}\mathrm{d}q\Big)^2   \mathrm{d}r\textbf{1}_{A_{R,t}}
		\bigg]\mathrm{d}t\notag\\
		\leq&C\int_0^T\mathbb{E}\big[\|\tilde{X}^{\varepsilon,\delta}-\hat{X}^{\varepsilon,\delta}\|^2_{\alpha,t}\textbf{1}_{A_{R,t}}\big]\mathrm{d}t,
		\label{eqn-4.28}
	\end{align}
	where  $\Delta(\tilde{X}^{\varepsilon,\delta}_t)=\int_0^t\frac{\|\tilde{X}^{\varepsilon,\delta}_t-\tilde{X}^{\varepsilon,\delta}_s\|}{(t-s)^{\alpha+1}}\mathrm{d}s$, $\Delta(\hat{X}^{\varepsilon,\delta}_t)=\int_0^t\frac{\|\hat{X}^{\varepsilon,\delta}_t-\hat{X}^{\varepsilon,\delta}_s\|}{(t-s)^{\alpha+1}}\mathrm{d}s$ can be dominated by $\|\tilde{X}^{\varepsilon,\delta}\|_{\alpha,\infty}$ and $\|\hat{X}^{\varepsilon,\delta}\|_{\alpha,\infty}$ which are in turn dominated by $C=C(R,N,T)>0$ independent of $\varepsilon$, $\delta$, $\Delta$.  We obtained the first inequality by using  \eqref{eqn-4.50} in Lemma \ref{lem4.3}. The second inequality comes from Remark \ref{re3.1}. In particular, we use Fubini's theorem and Lemma \ref{lem2.1} in the first term of the second inequality. We use Fubini's theorem and H\"{o}lder's inequality in the third and fourth inequalities, respectively. Applying the fact that $1-H<\alpha<\frac12$, we have the fifth inequality.

	Similarly, from the fact that  $(u^\varepsilon,v^\varepsilon)\in \mathcal{A}_b^N$,  for any $\gamma\in(\frac12, 1-\alpha)$, we have
	\begin{align}
		M_{22}
		\leq&C\int_0^T\mathbb{E}\bigg[\bigg(\int_0^t \int_0^s\Big( \frac{(s-r)^{-\gamma}r^{-\alpha}+(s-r)^{-\alpha-\gamma}}{(t-s)^{\alpha+1-\gamma}}\|\tilde{X}^{\varepsilon,\delta}_r-\hat{X}^{\varepsilon,\delta}_r\|\notag\\
		&+\int_0^r\frac{\sup_{i\in\mathbb{N}}\big\|\big(g(\tilde{X}^{\varepsilon,\delta}_r)-g(\hat{X}^{\varepsilon,\delta}_r)-g(\tilde{X}^{\varepsilon,\delta}_q)+g(\hat{X}^{\varepsilon,\delta}_q)\big)e_i\big\|}{(t-s)^{\alpha+1-\gamma}(r-q)^{\alpha+1}(s-r)^{\gamma}}\mathrm{d}q	\Big)\mathrm{d}r\mathrm{d}s	\bigg)^2 \textbf{1}_{A_{R,t}}  \bigg]\mathrm{d}t\notag\\
		\leq& C\int_0^T\mathbb{E}\bigg[\bigg(\int_0^t \|\tilde{X}^{\varepsilon,\delta}_r-\hat{X}^{\varepsilon,\delta}_r\|(t-r)^{-\alpha}\big(r^{-\alpha}+(t-r)^{-\alpha}+1\big)\mathrm{d}r  \notag\\
		&+ \int_0^t(t-r)^{-\alpha}\int_0^r \frac{\|\tilde{X}^{\varepsilon,\delta}_r-\hat{X}^{\varepsilon,\delta}_r-\tilde{X}^{\varepsilon,\delta}_q+\hat{X}^{\varepsilon,\delta}_q\|}{(r-q)^{\alpha+1}}\mathrm{d}q\mathrm{d}r\notag\\
		&+\int_0^t \|\tilde{X}^{\varepsilon,\delta}_r-\hat{X}^{\varepsilon,\delta}_r\|\big(\Delta(\tilde{X}^{\varepsilon,\delta}_r)+\Delta(\tilde{X}^{\varepsilon,\delta}_r)\big)(t-r)^{-\alpha}\mathrm{d}r	
		\bigg)^2   \textbf{1}_{A_{R,t}}
		\bigg]\mathrm{d}t\notag\\
		\leq& C\int_0^T \mathbb{E}\big[\|\tilde{X}^{\varepsilon,\delta}-\hat{X}^{\varepsilon,\delta}\|^2_{\alpha,t}\textbf{1}_{A_{R,t}} \big]\mathrm{d}t,	
		\label{eqn-4.29}
	\end{align}
	where the first inequality follows  \eqref{eqn-4.51} in Lemma \ref{lem4.3}.
	We use  Remark \ref{re3.1}, Fubini's theorem and Lemma \ref{lem2.1} in the second step. 
	In addition, we use H\"{o}lder's inequality in the third step.

	Similarly to  \eqref{eqn-3.13}, from Remark \ref{re3.1}, we have
	\begin{align}
		M_{23}
		\leq& 
		C\mathbb{E}\bigg[\sup_{t\in[0,T]}\big(\Lambda^{0,t}_{\alpha,u^\varepsilon}\big)^2\Big(\int_0^t\frac{\|\tilde{X}^{\varepsilon,\delta}_s-\hat{X}^{\varepsilon,\delta}_s\|}{s^\alpha}\mathrm{d}s
		+\int_0^t  \frac{\|\tilde{X}^{\varepsilon,\delta}_s-\hat{X}^{\varepsilon,\delta}_s\|}  {(t-s)^{\alpha}} \mathrm{d}s  \notag\\
		&+\int_0^t\int_0^s\frac {\sup_{i\in\mathbb{N}}\big\|\big(g(\tilde{X}^{\varepsilon,\delta}_s)-g(\hat{X}^{\varepsilon,\delta}_s)-g(\tilde{X}^{\varepsilon,\delta}_r)+g(\hat{X}^{\varepsilon,\delta}_r)\big)e_i
			\big\|}{(s-r)^{\alpha+1}}\mathrm{d}r\mathrm{d}s	
		\Big)^2 \textbf{1}_{A_{R,T}}	\bigg]\notag\\
		\leq& C\mathbb{E}\bigg[\sup_{t\in[0,T]} \bigg(\int_0^t \|\tilde{X}^{\varepsilon,\delta}_s-\hat{X}^{\varepsilon,\delta}_s\|
		\big(s^{-\alpha}+(t-s)^{-\alpha}+\Delta(\tilde{X}^{\varepsilon,\delta}_s)+\Delta(\hat{X}^{\varepsilon,\delta}_s)\big)\mathrm{d}s \notag\\
		&+\int_0^t\int_0^s \frac{\|\tilde{X}^{\varepsilon,\delta}_s-\hat{X}^{\varepsilon,\delta}_s-\tilde{X}^{\varepsilon,\delta}_r+\hat{X}^{\varepsilon,\delta}_r\|}{(s-r)^{\alpha+1}}\mathrm{d}r\mathrm{d}s    \bigg)^2 \textbf{1}_{A_{R,T}}\bigg]\notag\\
		\leq& C\mathbb{E}\bigg[\int_0^T\|\tilde{X}^{\varepsilon,\delta}_t-\hat{X}^{\varepsilon,\delta}_t\|^2\mathrm{d}t\textbf{1}_{A_{R,T}}\notag\\
		&+\int_0^T\Big(\int_0^t\frac{\|\tilde{X}^{\varepsilon,\delta}_t-\hat{X}^{\varepsilon,\delta}_t-\tilde{X}^{\varepsilon,\delta}_s+\hat{X}^{\varepsilon,\delta}_s\|}{(t-s)^{\alpha+1}}\mathrm{d}s\Big)^2\mathrm{d}t \textbf{1}_{A_{R,T}} \bigg]\notag\\
		\leq& C\int_0^T\mathbb{E}\big[\|\tilde{X}^{\varepsilon,\delta}-\hat{X}^{\varepsilon,\delta}\|^2_{\alpha,t}\textbf{1}_{A_{R,t}}
		\big]\mathrm{d}t\notag\\
		&+ C\mathbb{E}\bigg[\int_0^T\Big(\int_0^t\frac{\|\tilde{X}^{\varepsilon,\delta}_t-\hat{X}^{\varepsilon,\delta}_t-\tilde{X}^{\varepsilon,\delta}_s+\hat{X}^{\varepsilon,\delta}_s\|}{(t-s)^{\alpha+1}}\mathrm{d}s\Big)^2\mathrm{d}t \textbf{1}_{A_{R,T}} \bigg].
		\label{eqn-4.30}
	\end{align}
	Here, we obtain the third inequality by using H\"{o}lder's inequality.
	For simplicity, let 
	$$A_4:= C\mathbb{E}\bigg[\int_0^T\Big(\int_0^t\frac{\|\tilde{X}^{\varepsilon,\delta}_t-\hat{X}^{\varepsilon,\delta}_t-\tilde{X}^{\varepsilon,\delta}_s+\hat{X}^{\varepsilon,\delta}_s\|}{(t-s)^{\alpha+1}}\mathrm{d}s\Big)^2\mathrm{d}t \textbf{1}_{A_{R,T}} \bigg].$$
	We have
	\begin{align}
		A_4
		\leq& 
		C\mathbb{E}\bigg[\int_0^T\Big(\int_0^t(t-s)^{-\alpha-1}\Big\|\int_0^tS_{t-r}\left(b(\tilde{X}^{\varepsilon,\delta}_r,\tilde{Y}^{\varepsilon,\delta}_r)-b(\tilde{X}^{\varepsilon,\delta}_{r(\Delta)},\hat{Y}^{\varepsilon,\delta}_r)\right)\mathrm{d}r\notag\\
		&-\int_0^sS_{s-r}\left(b(\tilde{X}^{\varepsilon,\delta}_r,\tilde{Y}^{\varepsilon,\delta}_r)-b(\tilde{X}^{\varepsilon,\delta}_{r(\Delta)},\hat{Y}^{\varepsilon,\delta}_r)\right)\mathrm{d}r\Big\|\mathrm{d}s \Big)^2  \mathrm{d}t \textbf{1}_{A_{R,T}} \bigg]\notag\\
		&+C\mathbb{E}\bigg[\int_0^T\Big(\int_0^t(t-s)^{-\alpha-1}\Big\|\int_0^tS_{t-r}\left(g(\tilde{X}^{\varepsilon,\delta}_r)-g(\hat{X}^{\varepsilon,\delta}_r)\right)\mathrm{d}u^\varepsilon_r\notag\\
		&-\int_0^sS_{s-r}\left(g(\tilde{X}^{\varepsilon,\delta}_r)-g(\hat{X}^{\varepsilon,\delta}_r)\right)\mathrm{d}u^\varepsilon_r\Big\|\mathrm{d}s \Big)^2  \mathrm{d}t \textbf{1}_{A_{R,T}} \bigg]\notag\\
		&+\varepsilon  C \mathbb{E}\bigg[\int_0^T\Big(\int_0^t\frac{\|\int_0^t S_{t-r}g(\tilde{X}^{\varepsilon,\delta}_r)dB^H_r-\int_0^s S_{s-r}g(\tilde{X}^{\varepsilon,\delta}_r)dB^H_r  \|}{(t-s)^{ \alpha+1}}\mathrm{d}s\Big)^2\mathrm{d}t  \textbf{1}_{A_{R,T}} \bigg] \notag\\
		=:&A_{41}+A_{42}+A_{43}.
		\label{eqn-4.31}
	\end{align}
	Following similar arguments as for the term $M_1$ in \eqref{eqn-4.26}, we have 
	\begin{align}
		A_{41}
		\leq& 
		C\Delta+C\frac\delta\varepsilon.
		\label{eqn-4.32}
	\end{align}
	Referring to the terms $M_{21}$ in \eqref{eqn-4.28} and $M_{22}$ in \eqref{eqn-4.29}, we have 
	\begin{align}
		A_{42}
		\leq& 
		C\int_0^T \mathbb{E}\left[\|\tilde{X}^{\varepsilon,\delta}-\hat{X}^{\varepsilon,\delta}\|^2_{\alpha,t}\textbf{1}_{A_{R,t}} \right] \mathrm{d}t.	
		\label{eqn-4.33}
	\end{align}
	For the term $A_{43}$, by analogous calculations as for the term $H_3$ in Lemma \ref{lem4.3} and taking the same $\rho$ as in \eqref{eqn-3.26}, we obtain
	\begin{align}
		A_{43}
		\leq  \varepsilon C \mathbb{E}\big[e^{2\rho T}\big(1+\|\tilde{X}^{\varepsilon,\delta}\|_{\rho,T}+\|\tilde{X}^{\varepsilon,\delta}\|_{1,\rho,T}\big)^2
		\textbf{1}_{A_{R,T}} \big]\
		\leq C \varepsilon .    
		\label{eqn-4.34}
	\end{align}
	It is worth noting that $\rho$, $\|\tilde{X}^{\varepsilon,\delta}\|_{\rho,T}$ and $\|\tilde{X}^{\varepsilon,\delta}\|_{1,\rho,T}$ are dominated by a constant $C=C(R,N,T)>0$ and independent of $\varepsilon$, $\delta$, $\Delta$.
	Combining  \eqref{eqn-4.31}-\eqref{eqn-4.34}, we conclude that
	$$A_4\leq C\int_0^T \mathbb{E}\big[\|\tilde{X}^{\varepsilon,\delta}-\hat{X}^{\varepsilon,\delta}\|^2_{\alpha,t}  \textbf{1}_{A_{R,t}}\big]\mathrm{d}t +C\Big(\Delta+\frac\delta\varepsilon+\varepsilon\Big).$$
	Then, substituting the inequality above into  \eqref{eqn-4.30} and combining  \eqref{eqn-4.27}-\eqref{eqn-4.29}, we have 
	\begin{align}
		M_{2}
		\leq& 
		C\int_0^T\mathbb{E}\big[\|\tilde{X}^{\varepsilon,\delta}-\hat{X}^{\varepsilon,\delta}\|^2_{\alpha,t}  \textbf{1}_{A_{R,t}}\big]\mathrm{d}t+C\Big(\Delta+\frac\delta\varepsilon+\varepsilon\Big).
		\label{eqn-4.35}
	\end{align}
	
	For the term $M_3$, from Lemma \ref{lem4.3}, we have 
	\begin{align}
		M_{3}
		\leq C \varepsilon     .
		\label{eqn-4.36}
	\end{align}
	Substituting  \eqref{eqn-4.26}, \eqref{eqn-4.35} and \eqref{eqn-4.36} into  \eqref{eqn-4.24} yields that
	$$\mathbb{E}\big[ \|\tilde{X}^{\varepsilon,\delta}- \hat{X}^{\varepsilon,\delta}\|^2_{\alpha,T} \textbf{1}_{A_{R,T}} \big] \leq  C\int_0^T\mathbb{E}\big[\|\tilde{X}^{\varepsilon,\delta}-\hat{X}^{\varepsilon,\delta}\|^2_{\alpha,t}  \textbf{1}_{A_{R,t}}\big]\mathrm{d}t+C\Big(\Delta+\frac\delta\varepsilon+\varepsilon\Big). $$
	Applying Gronwall's inequality, we indicate that
	\begin{align}
		\mathbb{E}\big[ \|\tilde{X}^{\varepsilon,\delta}- \hat{X}^{\varepsilon,\delta}\|^2_{\alpha,T} \textbf{1}_{A_{R,T}} \big] \leq C\Big(\Delta+\frac\delta\varepsilon+\varepsilon \Big).
		\label{eqn-4.37}
	\end{align}
	Substituting  \eqref{eqn-4.23} and \eqref{eqn-4.37} into  \eqref{eqn-4.22}, we have 
	\begin{align}
		\mathbb{E}\big[ \|\tilde{X}^{\varepsilon,\delta}- \hat{X}^{\varepsilon,\delta}\|^2_{\alpha,T}  \big] \leq C\Big(\Delta+\frac\delta\varepsilon+\varepsilon\Big)
		+C'R^{-\frac12}\sqrt{\mathbb{E}\big[\Lambda^{0,T}_{\alpha,B^H}\big]},
		\label{eqn-4.38}
	\end{align}
	where the $C$ is a positive constant which is independent of $\varepsilon$, $\delta$, $\Delta$ and $C'$ is a positive constant which is independent of $R$, $\varepsilon$, $\delta$, $\Delta$.


	\textbf{Step 2.} For any fixed $N\in\mathbb{N}$, let $(u^\varepsilon,v^\varepsilon)\in \mathcal{A}_b^N$. We construct the following SPDE,
	$$
	d\bar{X}^{u^\varepsilon}_t=\big(A\bar{X}^{u^\varepsilon}_t+\bar{b}(\bar{X}^{u^\varepsilon}_t)\big)\mathrm{d}t+g(\bar{X}^{u^\varepsilon}_t)\mathrm{d}u^\varepsilon_t,\quad\bar{X}^{u^\varepsilon}_0=X_0,\quad t\in[0,T].
	$$
	In other words, $\bar{X}^{u^\varepsilon}=\mathcal{G}^{0}(u^\varepsilon,v^\varepsilon)$. In this step, We will estimate $\mathbb{E}\big[ \|\hat{X}^{\varepsilon,\delta}- \bar{X}^{u^\varepsilon}\|^2_{\alpha,T} \big]$.

	It is easy to see that
	\begin{align}
		\mathbb{E}\big[ \|\hat{X}^{\varepsilon,\delta}- \bar{X}^{u^\varepsilon}\|^2_{\alpha,T}\textbf{1}_{A_{R,T}} \big]
		\leq&
		C\mathbb{E}\Big[\Big\|\int_0^\cdot S_{\cdot-s}\big(
		b(\tilde{X}^{\varepsilon,\delta}_{s(\Delta)},\hat{Y}^{\varepsilon,\delta}_s)-\bar{b}(\tilde{X}^{\varepsilon,\delta}_{s(\Delta)})\big)\mathrm{d}s	
		\Big\|^2_{\alpha,T}  \textbf{1}_{A_{R,T}}
		\Big]\notag\\
		&+ C\mathbb{E}\Big[\Big\|\int_0^\cdot S_{\cdot-s}\big(
		\bar{b}(\tilde{X}^{\varepsilon,\delta}_{s(\Delta)})-\bar{b}(\tilde{X}^{\varepsilon,\delta}_s)\big)\mathrm{d}s	
		\Big\|^2_{\alpha,T}  \textbf{1}_{A_{R,T}}\Big]\notag\\
		&+ C\mathbb{E}\Big[\Big\|\int_0^\cdot S_{\cdot-s}\big(
		\bar{b}(\tilde{X}^{\varepsilon,\delta}_s)-\bar{b}(\hat{X}^{\varepsilon,\delta}_s)\big)\mathrm{d}s	
		\Big\|^2_{\alpha,T}  \textbf{1}_{A_{R,T}}\Big]\notag\\
		&+ C\mathbb{E}\Big[\Big\|\int_0^\cdot S_{\cdot-s}\big(
		\bar{b}(\hat{X}^{\varepsilon,\delta}_s)-\bar{b}(\bar{X}^{u^\varepsilon}_s)\big)\mathrm{d}s	
		\Big\|^2_{\alpha,T} \textbf{1}_{A_{R,T}}\Big]\notag\\
		&+C\mathbb{E}\Big[ \Big\| \int_0^\cdot S_{\cdot-s}\big(g(\hat{X}^{\varepsilon,\delta}_s)-g(\bar{X}^{u^\varepsilon}_s)\big)\mathrm{d}u^{\varepsilon}_s\Big\|^2_{\alpha,T}  \textbf{1}_{A_{R,T}}
		\Big]\notag\\
		=:&K_1+K_2+K_3+K_4+K_5.
		\label{eqn-4.41}
	\end{align}

	Following similar arguments as for the term $J_1$ in \cite[pp. 20-21]{Pei2020}, for any $\zeta\in(\frac{1+2\alpha}{4},\frac12)$, we conclude that
	\begin{align}
		K_1
		\leq&C\Delta^{2\zeta}
		+\frac{ C}{ \Delta^{2}} \max_{0\leq k \leq \lfloor\frac T\Delta\rfloor-1}\mathbb{E}\bigg[\Big\|\int_{k\Delta}^{(k+1)\Delta}S_{t-r}\left(b(\tilde{X}^{\varepsilon,\delta}_{k\Delta},\hat{Y}^{\varepsilon,\delta}_r)-\bar{b}(\tilde{X}^{\varepsilon,\delta}_{k\Delta})	\right)\mathrm{d}r\Big\|^2\textbf{1}_{A_{R,T}}\bigg]
		\notag\\
		&+\frac C\Delta\sup_{t\in[0,T]}\int_0^t\sum_{k=0}^{\lfloor\frac s\Delta\rfloor-1}\frac{\mathbb{E}\Big[\big\|\int_{k\Delta}^{(k+1)\Delta}\big(b(\tilde{X}^{\varepsilon,\delta}_{k\Delta},\hat{Y}^{\varepsilon,\delta}_r)-\bar{b}(\tilde{X}^{\varepsilon,\delta}_{k\Delta})\big)\mathrm{d}r \big\|^2\textbf{1}_{A_{R,T}}\Big]}{(s-k\Delta)^{2\zeta}(t-s)^{\alpha+\frac32-2\zeta}}\mathrm{d}s. 
		\label{eqn-4.48}
	\end{align}
	Note that for $2\zeta<1$, we have 
	\begin{align}
		\displaystyle\sum_{k=0}^{\lfloor\frac{s}{\Delta}\rfloor-1}(s-k\Delta)^{-2\zeta}
		\leq&C\Delta^{-1}.\label{eqn-4.49}
	\end{align}
	According to \cite[Lemma 4.12]{Pei2020}, we obtain 
	\begin{align}
		\mathbb{E}\bigg[\Big\|\int_{k\Delta}^{(k+1)\Delta}S_{(k+1)\Delta-r}(b(\tilde{X}^{\varepsilon,\delta}_{k\Delta},\hat{Y}_{r}^{\varepsilon,\delta})-\bar{b}(\tilde{X}^{\varepsilon,\delta}_{k\Delta}))\mathrm{d}r\Big\|^{2}\textbf{1}_{A_{R,T}}\bigg]
		\leq C\delta^{2}\Big(\frac{2}{\eta}\frac{\Delta}{\delta}-\frac{4}{\eta^{2}}+e^{\frac{-\eta}{2}\frac{\Delta}{\delta}}\Big).\notag
	\end{align}

	Therefore, substituting \eqref{eqn-4.49} into  \eqref{eqn-4.48} leads to
	\begin{align}
		K_1\leq 
		C\Delta^{2\zeta}+C \Delta^{-2}\delta^2\Big(\frac{2}{\eta}\frac{\Delta}{\delta}-\frac{4}{\eta^{2}}+e^{\frac{-\eta}{2}\frac{\Delta}{\delta}}\Big)\leq C\Big(\frac\delta\Delta+\Delta^{2\zeta}\Big).
		\label{eqn-4.52}
	\end{align}

	For the terms $K_2$ and $K_3$, by referring to the term $A_3$  in \eqref{eqn-43} from \textbf{Step 1}, and applying Lemmas \ref{lem3.2} and \ref{lem4.4}, we have 
	\begin{align}
		K_2+K_3
		\leq& C\int_0^T \mathbb{E}\big[\|\tilde{X}^{\varepsilon,\delta}_{t(\Delta)}-\tilde{X}^{\varepsilon,\delta}_t\|^2 \textbf{1}_{A_{R,T}}+\|\tilde{X}^{\varepsilon,\delta}_t-\hat{X}^{\varepsilon,\delta}_t\|^2 
		\textbf{1}_{A_{R,T}}\big]\mathrm{d}t\notag\\
		\leq& C\Big(\Delta+\frac\delta\varepsilon+\varepsilon\Big) ,\label{eqn-4.53}
	\end{align}
	where $C>0$ is a constant which is independent of $\varepsilon$, $\delta$, $\Delta$.
 We use \eqref{eqn-4.38} in \textbf{Step 1}   to obtain the final inequality.
	Similarly, we have 
	\begin{align}
		K_4
		\leq 
		C\int_0^T \mathbb{E}\big[\|\hat{X}^{\varepsilon,\delta}- \bar{X}^{u^\varepsilon}
		\|^2_{\alpha,t}  
		\textbf{1}_{A_{R,T}}
		\big] \mathrm{d}t	.
		\label{eqn-4.54}
	\end{align}
	
	The term $K_5$ can be estimated similar to the term $M_2$  
	in \textbf{Step 1}. We have
	\begin{align}
		K_5
		\leq& C\int_0^T\mathbb{E}\big[\|\hat{X}^{\varepsilon,\delta}-\bar{X}^{u^\varepsilon}\|^2_{\alpha,t}\textbf{1}_{A_{R,T}}\big]\mathrm{d}t
		+\sum_{i=1}^{4}K_i\notag\\
		\leq& C\int_0^T\mathbb{E}\big[\|\hat{X}^{\varepsilon,\delta}-\bar{X}^{u^\varepsilon}\|^2_{\alpha,t}\textbf{1}_{A_{R,T}}\big]\mathrm{d}t +C\Big(\frac\delta\Delta+\Delta^{2\zeta}+\frac\delta\varepsilon+\varepsilon\Big)
		.
		\label{eqn-4.55}
	\end{align}
	
	Substituting  \eqref{eqn-4.52}-\eqref{eqn-4.55} into  \eqref{eqn-4.41} leads to
	\begin{align*}
		\mathbb{E}\big[\|\hat{X}^{\varepsilon,\delta}-\bar{X}^{u^\varepsilon}\|^2_{\alpha,T}\textbf{1}_{A_{R,T}}\big]
		\leq & C\int_0^T\mathbb{E}\big[\|\hat{X}^{\varepsilon,\delta}-\bar{X}^{u^\varepsilon}\|^2_{\alpha,t} \textbf{1}_{A_{R,T}}\big]\mathrm{d}t\\
		&+C\Big(\frac\delta\Delta+\Delta^{2\zeta}+\frac\delta\varepsilon+\varepsilon\Big) 
.
	\end{align*}
	Making use of Gronwall's inequality, we have
	\begin{align*}
		\mathbb{E}\big[\|\hat{X}^{\varepsilon,\delta}-\bar{X}^{u^\varepsilon}\|^2_{\alpha,T}\textbf{1}_{A_{R,T}}\big]
		\leq C\Big(\frac\delta\Delta+\Delta^{2\zeta}+\frac\delta\varepsilon+\varepsilon\Big)
.
	\end{align*}
	According to Markov's inequality, we have
\begin{align}
	\mathbb{E}\big[\|\hat{X}^{\varepsilon,\delta}-\bar{X}^{u^\varepsilon}\|^2_{\alpha,T}\big]
	\leq&
	\mathbb{E}\big[\|\hat{X}^{\varepsilon,\delta}-\bar{X}^{u^\varepsilon}\|^2_{\alpha,T}\textbf{1}_{A_{R,T}}\big]
	+\mathbb{E}\left[ \|\hat{X}^{\varepsilon,\delta}-\bar{X}^{u^\varepsilon}\|^4_{\alpha,T}  \right]^\frac12  \mathbb{P}\left(\tau_R<T \right)^\frac12  \notag\\
	\leq&
	C\big(\frac\delta\Delta+\Delta^{2\zeta}+\frac\delta\varepsilon+\varepsilon\big)
	+C'R^{-\frac12}\sqrt{\mathbb{E}[\Lambda^{0,T}_{\alpha,B^H}]}
	,\label{eqn-3.84}
\end{align}
where the $C$ is a positive constant which is independent of $\varepsilon$, $\delta$, $\Delta$ and $C'$ is a positive constant which is independent of $R$, $\varepsilon$, $\delta$, $\Delta$.

\textbf{Step 3.} 
By  \eqref{eqn-4.38} and \eqref{eqn-3.84}, taking $\limsup_{\varepsilon\to 0}$ for every fixed large $R$ and then letting $R\to \infty$, we have
$$\lim_{\varepsilon\to 0} \mathbb{E}\big[ \|\tilde{X}^{\varepsilon,\delta}- \bar{X}^{u^\varepsilon}\|^2_{\alpha,T}  \big]=0. $$	Precisely, we take   $\Delta=\Delta(\delta(\varepsilon))$ such that $\frac{\Delta}{\delta(\varepsilon)}\to\infty $ and $\Delta\to 0$, as $\varepsilon\to0$. Note that, from Assumption \textbf{(A7)}, $\delta=o(\varepsilon)$. For example, we could take $\Delta=\varepsilon$.

For any fixed $N\in\mathbb{N}$, let $\{(u^\varepsilon,v^\varepsilon)\}_{\varepsilon\in(0,1)}\subset\mathcal{A}_b^N$ such that   $\{(u^\varepsilon,v^\varepsilon)\}_{\varepsilon\in(0,1)}$ converges to $(u,v)\in\mathcal{S}_N$ in the weak topology as $\varepsilon\to 0$. 
According to the proof of \textit{Condition (1)}, it follows that $\{\bar{X}^{u^\varepsilon}=\mathcal{G}^{0}(u^\varepsilon,v^\varepsilon)\}_{\varepsilon\in(0,1)}$ weakly convergences to $\bar{X}^{u}=\mathcal{G}^{0}(u,v)$ in $C([0,T],V)$ as $\varepsilon\to 0$. 
Then, by Portemanteau's theorem \cite[Theorem 13.16]{Klenke2020}, for any bounded Lipschitz function $f:C([0,T],V)\to \mathbb{R}$, it follows that as $\varepsilon\to0$
\begin{align*}
	\big\|\mathbb{E}\big[f(\tilde{X}^{\varepsilon,\delta})\big]-\mathbb{E}\big[f(\bar{X}^{u})\big]\big\|\leq&
	\big\|\mathbb{E}\big[f(\tilde{X}^{\varepsilon,\delta})\big]-\mathbb{E}\big[f(\bar{X}^{u^\varepsilon})\big]\big\|+\big\|\mathbb{E}\big[f(\tilde{X}^{u^\varepsilon})\big]-\mathbb{E}\big[f(\bar{X}^{u})\big]\big\|\\
	\leq& C \mathbb{E}\big[\|\tilde{X}^{\varepsilon,\delta}-\bar{X}^{u^\varepsilon}\|^2_{\alpha,T}\big]^{\frac12}+\big\|\mathbb{E}\big[f(\tilde{X}^{u^\varepsilon})\big]-\mathbb{E}\big[f(\bar{X}^{u})\big]\big\|\to 0.
\end{align*}
From the above, it follows that $\{\mathcal{G}^\varepsilon(\sqrt\varepsilon B^H+u^\varepsilon, \sqrt\varepsilon W+v^\varepsilon)\}_{\varepsilon\in(0,1)}$ converges to $\mathcal{G}^{0}(u,v)$ in distribution as $\varepsilon\to 0$.	
\end{proof}

\begin{proof}[Proof of Theorem \ref{th3.1}]
	According to the Lemma \ref{lemA.1}, the proof of the \textit{Condition (1)} and \textit{(2)}, we have that the slow component $X^{\varepsilon,\delta}$ of system \eqref{eqn-1.2} satisfies the LDP with the good
	rate function $I:C([0,T],V)\to [0,\infty]$. The proof was completed. 
\end{proof}

\section{Moderate Deviation Principle}\label{6}

Based on the LDP established in Section \ref{5}, this section will further investigate the MDP for the stochastic slow-fast partial differential equation \eqref{eqn-1.2} under infinite-dimensional mixed FBM.

To ensure the existence and uniqueness of mild solutions for the stochastic slow-fast differential equation \eqref{eqn-1.2}, assume that the Assumptions \textbf{(A1)}-\textbf{(A3)} in Section \ref{3} hold. To derive the AP for system \eqref{eqn-1.2}, assume that the Assumptions \textbf{(A4)}-\textbf{(A6)} in Section \ref{3} are satisfied. To further establish the MDP for system \eqref{eqn-1.2}, assume that the Assumption \textbf{(A7)} and the following condition hold:
\begin{itemize}
	\item[\textbf{(H1)}]
	The averaged drift coefficient $\bar{b}\in C^1$, and there exists a constant $L>0$, such that for any $x_1,~x_2\in V$,
	$$\|D\bar{b}(x_1)-D\bar{b}(x_2)\|\leq L\|x_1-x_2\|.$$
	\item[\textbf{(H2)}]
	Function $h : (0, 1] \to  (0, \infty)$ is continuous, and satisfies that $\lim_{\varepsilon \to 0}h(\varepsilon)=\infty$ and $\lim_{\varepsilon \to 0}\sqrt{\varepsilon}h(\varepsilon)=0$ for all $\varepsilon\in(0,1]$.
\end{itemize}

From slow-fast system \eqref{eqn-1.2}, it follows that the deviation component \eqref{eqn-6.1} satisfies the following SPDE
\begin{align}\label{eqn-4.2}
	\begin{cases}
		dZ^{\varepsilon,\delta}_t=\frac{1}{\sqrt{\varepsilon}h(\varepsilon)}\big(A(X^{\varepsilon,\delta}_t-\bar{X}_t)+b(X^{\varepsilon,\delta}_t,Y^{\varepsilon,\delta}_t)-\bar{b}(\bar{X}_t)\big)dt+\frac{1}{h(\varepsilon)}g(X^{\varepsilon,\delta}_t)dB^H_t,\\
		Z^{\varepsilon,\delta}_0=0,~ t\in[0,T].
	\end{cases}
\end{align}
Then, there is a measurable map
$\tilde{\mathcal{G}}^{\varepsilon,\delta}:C_0([0,T],V\times V)\to C([0,T],V)$
such that $Z^{\varepsilon,\delta}=\tilde{\mathcal{G}}^{\varepsilon,\delta}\Big(\frac{B^H}{h(\varepsilon)},\frac{W}{h(\varepsilon)}\Big)$.

To prove that the slow-fast system \eqref{eqn-1.2} satisfies the MDP, we need to perform some key estimates. 
For any fixed $N\in\mathbb{N}$, let $(u^\varepsilon,v^\varepsilon)\in\mathcal{A}_b^N$, and consider the controlled system corresponding to the slow-fast system  \eqref{eqn-1.2} as follows  
\begin{equation}\label{eqn-4.5}
	\begin{cases}
		\mathrm{d}\bar{X}^{\varepsilon,\delta}_t=(A\bar{X}^{\varepsilon,\delta}_t+b(\bar{X}^{\varepsilon,\delta}_t,\bar{Y}^{\varepsilon,\delta}_t))\mathrm{d}t
		+\sqrt{\varepsilon}h(\varepsilon)g(\bar{X}^{\varepsilon,\delta}_t)\mathrm{d}u^\varepsilon_t
		+\sqrt\varepsilon g(\bar{X}^{\varepsilon,\delta}_t)\mathrm{d}B^H_t,\\  
		\mathrm{d}\bar{Y}^{\varepsilon,\delta}_t=\frac1\delta(A\bar{Y}^{\varepsilon,\delta}_t+F(\bar{X}^{\varepsilon,\delta}_t,\bar{Y}^{\varepsilon,\delta}_t))\mathrm{d}t
		+\frac{h(\varepsilon)}{\sqrt{\delta}}G(\bar{X}^{\varepsilon,\delta}_t,\bar{Y}^{\varepsilon,\delta}_t)\mathrm{d}v^\varepsilon_t
		+\frac{1}{\sqrt\delta} G(\bar{X}^{\varepsilon,\delta}_t,\bar{Y}^{\varepsilon,\delta}_t)\mathrm{d}W_t,\\
		\bar{X}^{\varepsilon,\delta}_0=X_0,~\bar{Y}^{\varepsilon,\delta}_0=Y_0,~t\in[0,T].
	\end{cases}
\end{equation}	
It is easy to see that there exists a unique  mild solution $(\bar{X}^{\varepsilon,\delta},\bar{Y}^{\varepsilon,\delta})$ to the controlled system  \eqref{eqn-4.5} with the initial value $(x_0,y_0)\in V_\beta\times V_\beta$ for any $\beta>\alpha$.

Define the controlled deviation component $\tilde{Z}^{\varepsilon,\delta}$ as follows
$$\tilde{Z}^{\varepsilon,\delta}_t=\frac{\bar{X}^{\varepsilon,\delta}_t-\bar{X}_t}{\sqrt\varepsilon h(\varepsilon)},\quad \tilde{Z}^{\varepsilon,\delta}_0=0,\quad t\in[0,T],$$
where $\bar{X}$ denotes the solution to the deterministic averaged equation \eqref{eq-2}.
It follows that the controlled deviation component $\tilde{Z}^{\varepsilon,\delta}$ satisfies the following SPDE,
\begin{equation}\label{eqn-4.6}
	\begin{cases}
		\mathrm{d}\tilde{Z}^{\varepsilon,\delta}_t
		=\frac{\big(A(\bar{X}^{\varepsilon,\delta}_t-\bar{X}_t)+b(\bar{X}^{\varepsilon,\delta}_t,\bar{Y}^{\varepsilon,\delta}_t)-\bar{b}(\bar{X}_t)\big)}{\sqrt{\varepsilon}h(\varepsilon)}\mathrm{d}t
		+g(\bar{X}^{\varepsilon,\delta}_t)\mathrm{d}u^\varepsilon_t +\frac{g(\bar{X}^{\varepsilon,\delta})}{h(\varepsilon)}\mathrm{d}B^H_t,\\
		\tilde{Z}^{\varepsilon,\delta}_0=0,~t\in[0,T].
	\end{cases}
\end{equation}
Note that $\tilde{Z}^{\varepsilon,\delta}=\tilde{\mathcal{G}}^{\varepsilon,\delta}\big(\frac{B^H}{h(\varepsilon)}+u^\varepsilon,\frac{W}{h(\varepsilon)}+v^\varepsilon\big)$.

\begin{lemma}\label{lem6.1}
	Suppose that Assumptions \textnormal{\textbf{(A1)}-\textbf{(A3)}}, \textnormal{\textbf{(A6)}} and \textnormal{\textbf{(H2)}} hold. Let $N\in\mathbb{N}$. Then, there exists a constant $C>0$, such that for any $p\geq 1$ and  $(u^\varepsilon,v^\varepsilon)\in\mathcal{A}_b^N$, we have
	$$\mathbb{E}\big[\|\bar{X}^{\varepsilon,\delta}\|_{\alpha,\infty}^p\big]\leq C.$$
	Here, $C$ is a constant which only depends on $p$, $T$ and $N$.
\end{lemma}
\begin{proof}
	Under the Assumptions  \textbf{(A1)}-\textbf{(A3)} and \textbf{(A6)}, with $\varepsilon\in(0,1)$ and $\sqrt\varepsilon h(\varepsilon )\in(0,1)$ from Assumption \textbf{(H2)}, a proof similar to that of Lemma \ref{lem4.3} in Section \ref{4} can be used to prove Lemma \ref{lem6.1}.
\end{proof}

\begin{remark}\label{re6.1}
	From Lemma  \ref{lem6.1}, it follows that the solution $\bar{X}$ to the deterministic averaged  equation \eqref{eq-2} is bounded, i.e., for any $p\geq1$,  we have
	\begin{align*}
		\|\bar{X}\|^p_{\alpha,\infty}\leq C. 
	\end{align*}
	Here, $C$ is a positive constant which only depends on $p$ and $T$.
\end{remark}

\begin{lemma}\label{lem6.2}
	Suppose that Assumptions \textnormal{\textbf{(A1)}-\textbf{(A3)}}, \textnormal{\textbf{(A6)}} and \textnormal{\textbf{(H2)}} hold. Let $N\in\mathbb{N}$. Then, there exists a constant $C>0$, such that for every $(u^\varepsilon,v^\varepsilon)\in\mathcal{A}_b^N$, $\beta\in(\frac12,1-\alpha)$  and $0\leq s\leq  t\leq T$, we have 
	$$\mathbb{E}\big[\|\bar{X}^{\varepsilon,\delta}_t-\bar{X}^{\varepsilon,\delta}_s\|^2\big]\leq C(t-s)^{2\beta}.$$
	Here, $C$ is a constant which only depends on $T$ and $N$.
\end{lemma}
\begin{proof}
	Under the Assumptions  \textnormal{\textbf{(A1)}-\textbf{(A3)}} and \textnormal{\textbf{(A6)}}, with $\varepsilon\in(0,1)$ and $\sqrt\varepsilon h(\varepsilon)\in(0,1)$ from Assumption \textbf{(H2)}, a proof similar to that of Lemma \ref{lem4.4} in Section \ref{4} can be used to prove Lemma \ref{lem6.2}. 
\end{proof}

The skeleton equation is defined as follows
\begin{align}
	\mathrm{d}\bar{Z}_t=\big(A+D\bar{b}(\bar{X}_t)\big)\bar{Z}_t\mathrm{d}t+ g(\bar{X}_t)\mathrm{d}u_t,\quad\bar{Z}_0=0, \quad t\in[0,T].\label{eqn-4.3}
\end{align} 
According to the Remark \ref{re6.1}, Assumption \textbf{(H1)} implies that there exists a constant $C>0$ independent of $t$ and $z$, such that $\|D\bar{b}(\bar{X}_t)z\|\leq C\|z\|$ for any $t\in[0,T]$ and $z\in V$. Hence, according to \cite[Theorem 3.5]{Pei2020}, there exists a unique pathwise mild solution $\bar{Z}\in W^{\alpha,\infty}([0,T],V)$ to the skeleton equation \eqref{eqn-4.3} for any $(u,v)\in \mathcal{S}_N$. We define a map
$
\tilde{\mathcal{G}}^0 : \mathcal{H}\to C([0,T],V)$ 
by 
\begin{equation}\label{eqn70}
\bar{Z}=\tilde{\mathcal{G}}^0(u,v).	
\end{equation}
Here, the solution $\tilde{\mathcal{G}}^0(u,v)$ is independent of $v$.

\begin{lemma}\label{lem6.3}
	Suppose that Assumptions \textnormal{\textbf{(A1)}}-\textnormal{\textbf{(A6)}} and \textnormal{\textbf{(H1)}} hold. Let $N\in\mathbb{N}$. Then, there exists a constant $C>0$, such that  for any $p\geq 1$ and $(u,v)\in \mathcal{S}_N$, we have
	$$\sup_{t\in[0,T]}\|\bar{Z}_t\|^p\leq C.$$
	Here, $C$ is a constant which only depends on $p$, $T$ and $N$.
\end{lemma}
\begin{proof}
	
	According to Assumptions \textbf{(A3)} and \textbf{(H1)} and $(u,v)\in\mathcal{S}_N$, using Remark \ref{re6.1}  and H\"{o}lder's inequality, we obtained
	$$
	\begin{aligned}
		\|\bar{Z}_t\|^p
		\leq & C\Big\|\int_0^t S_{t-s} D\bar{b}(\bar{X}_s)\bar{Z}_s\mathrm{d}s\Big\|^p
		+C\Big\|\int_0^t S_{t-s} g(\bar{X}_s)\mathrm{d}u_s\Big\|^p\notag\\
		\leq & C \int_0^t \|D\bar{b}(\bar{X}_s)\|^p\|\bar{Z}_s\|^p\mathrm{d}s\notag\\
		&+ C\|u\|_{\mathcal{H}^H}^p \bigg(\int_0^t\Big(
		\big(1+\|\bar{X}_s\|\big)\big(s^{-\alpha}+ (t-s)^{-\alpha}\big)+\int_0^s\frac{\|\bar{X}_s-\bar{X}_r\|}{(s-r)^{\alpha+1}}\mathrm{d}r\Big)\mathrm{d}s\bigg)^p\notag\\
		\leq &C\big(1+\|\bar{X}\|_{\alpha,\infty}^p\big)\Big(\int_0^t\|\bar{Z}_s\|^p\mathrm{d}s+\Big(\int_0^t(1+s^{-\alpha}+ (t-s)^{-\alpha}\big)\mathrm{d}s\Big)^p\Big)\notag\\
		\leq &C \int_0^t \|\bar{Z}_s\|^p \mathrm{d}s +C. \label{eqn-4.8}
	\end{aligned}
	$$
	In the second step, we obtain the result from  \eqref{eqn-3.13} in Section \ref{4} and  $\Lambda^{0,T}_{\alpha,u}\leq C\|u\|_{\mathcal{H}^H}$.

	By applying Gronwall's inequality and taking the supremum on both sides over 
	$t\in[0,T]$, we have
	$$\sup_{t\in[0,T]}\|\bar{Z}_t\|^p\leq C.$$
\end{proof}

\begin{lemma}\label{lem6.4}
	Suppose that Assumptions \textnormal{\textbf{(A1)}}-\textnormal{\textbf{(A6)}} and \textnormal{\textbf{(H1)}} hold. Let $N\in\mathbb{N}$. Then, there exists a constant $C>0$, such that for any  $(u,v)\in \mathcal{S}_N$,  $\beta\in(\frac12,1-\alpha)$  and  $0\leq s\leq t\leq T$, we have
	$$\|\bar{Z}_t-\bar{Z}_s\|^2 \leq C(t-s)^{ 2\beta }.$$
	Here, $C$ is a constant which only depends on $T$ and $N$.
\end{lemma}
\begin{proof}
	From \eqref{eqn-4.50} and \eqref{eqn-4.51} in Section \ref{4}, we obtain
	\begin{align}
		\|\bar{Z}_t-\bar{Z}_s\|
		\leq & \Big\|\int_{s}^{t}S_{t-r} D\bar{b}(\bar{X}_r)\bar{Z}_r\mathrm{d}r\Big\|
		+\Big\|\int_0^s(S_{t-r}-S_{s-r})D\bar{b}(\bar{X}_r) \bar{Z}_r\mathrm{d}r\Big\|\notag\\
		&+\Big\|\int_s^tS_{t-r} g(\bar{X}_r)\mathrm{d}u_r\Big\|
		+\Big\|\int_0^s(S_{t-r}-S_{s-r}) g(\bar{X}_r)\mathrm{d}u_r\Big\|\notag\\
		\leq & C\big(1+\sup_{t\in[0,T]}\|\bar{X}_t\|\big)\Big(\int_s^t\|\bar{Z}_r\|\mathrm{d}r
		+(t-s)^\beta\int_0^s (s-r)^{-\beta}\|\bar{Z}_r\|\mathrm{d}r\Big)\notag\\
		&+C\Lambda^{0,t}_{\alpha,u} \int_s^t  \Big(\big(1+\|\bar{X}_r\|\big)\big((r-s)^{-\alpha}+(t-r)^{-\alpha}\big)+ \int_s^r \frac{\|\bar{X}_r-\bar{X}_q\|}{(r-q)^{\alpha+1}}\mathrm{d}q\Big)\mathrm{d}r\notag\\
		&+C\Lambda^{0,s}_{\alpha,u}(t-s)^{\beta}  \int_0^s  \big(1+\|\bar{X}_r\|\big)({(s-r)^{-\beta}r^{-\alpha}+(s-r)^{-\alpha-\beta}})\mathrm{d}r\notag\\
		&+ C\Lambda^{0,s}_{\alpha,u}(t-s)^{\beta} \int_0^s  (s-r)^{-\beta}\int_0^r \frac{\|\bar{X}_r-\bar{X}_q\|}{(r-q)^{\alpha+1}}\mathrm{d}q\mathrm{d}r
		\notag\\
		\leq & C\big(1+\|\bar{X}\|_{\alpha,\infty}\big)\big((t-s)^\beta+(t-s)\big) \sup_{0\leq r\leq t}\|\bar{Z}_r\|\notag\\
		&+C\|u\|_{\mathcal{H}^H}\big(1+\|\bar{X}\|_{\alpha,\infty}\big)\big((t-s)^{1-\alpha}+(t-s)^\beta+(t-s)\big)\notag\\
		\leq & C\big((t-s)^{1-\alpha}+(t-s)^\beta+(t-s)\big)\big(1+\sup_{0\leq r\leq t}\|\bar{Z}_r\|\big), \label{eqn-6.9}
	\end{align}
	where, the third inequality follows from $\Lambda^{0,T}_{\alpha,u}\leq C\|u\|_{\mathcal{H}^H}$, and the final step is derived from Remark \ref{re6.1}. 
	Without loss of generality, assume $0<t-s<1$. Substituting the conclusion of Lemma \ref{lem6.3} into equation \eqref{eqn-6.9}, from   $\beta\in\left(\frac12,1-\alpha\right)$, we have
	$$\|\bar{Z}_t-\bar{Z}_s\|^2\leq C(t-s)^{2\beta }.$$
\end{proof}

\begin{lemma}\label{lem6.5}
	Suppose that Assumptions \textnormal{\textbf{(A1)}}-\textnormal{\textbf{(A7)}} and  \textnormal{\textbf{(H2)}}  hold. Let $N\in\mathbb{N}$.
	Then, there exists a constant $C>0$, such that for any $(u^\varepsilon,v^\varepsilon)\in\mathcal{A}_b^N$, we have
	$$\int_0^T\mathbb{E}\big[\|\bar{Y}^{\varepsilon,\delta}_t\|^2\big]\mathrm{d}t\leq C.$$
	Here, $C$ is a constant which only depends on $T$ and $N$.
\end{lemma}

\begin{proof}
	Under the Assumption  \textbf{(A7)} and  $\lim_{\varepsilon\to 0}\sqrt\varepsilon h(\varepsilon)=0$ from Assumption \textbf{(H2)}, 
	choose the parameters $\varepsilon$ and $\delta$ such that $\delta h^2(\varepsilon)\leq\left(\bar\lambda_1+\beta_1\right)^2$.
	Under the assumptions  \textnormal{\textbf{(A1)}-\textbf{(A7)}} with $\varepsilon\in(0,1)$ and $\sqrt\varepsilon h(\varepsilon)\in(0,1)$, the proof of Lemma \ref{lem6.5} follows similarly to the proof of Lemma \ref{lem4.6} in Section \ref{4}. 
\end{proof}

Now, we provide a precise statement of our second main theorem.

\begin{theorem}\label{th6.1}
	Suppose that Assumptions \textnormal{\textbf{(A1)}}-\textnormal{\textbf{(A7)}} and \textnormal{\textbf{(H1)}}-\textnormal{\textbf{(H2)}} hold. Let $\varepsilon\to 0$. Then the slow component $X^{\varepsilon,\delta}$ of the slow-fast system \eqref{eqn-1.2} satisfies the MDP in $C({0,T},V)$ with speed function $b(\varepsilon)=\frac{1}{h^2(\varepsilon)}$ and the good rate function ${\tilde{I}}:C([0,T],V)\to [0,\infty]$ defined by 
	$$ \tilde{I}(\phi):=\inf_{\{(u,v)\in\mathcal{H}:~\phi=\tilde{\mathcal{G}}^0(u,v)\}}{\frac12\|(u,v)\|^2_{\mathcal{H}}}
	=\inf_{\{u\in\mathcal{H}^H:~\phi=\tilde{\mathcal{G}}^0(u,0)\}}{\frac12\|u\|^2_{\mathcal{H}^H}},$$
	where $\tilde{\mathcal{G}}^0$ is defined by \eqref{eqn70} and infimum over an empty set is taken as $+\infty$.
\end{theorem}

\begin{remark}\label{round2-2}
	 Assume the diffusion matrix $g$ being symmetric with a bounded inverse. According to the skeleton equation \eqref{eqn-4.3} corresponding to the MDP, we have
	\begin{align*}
		\tilde{I}(\phi)
		=\inf_{\{u\in\mathcal{H}^H:~\phi=\tilde{\mathcal{G}}^0(u,0)\}}{\frac12\|u\|^2_{\mathcal{H}^H}}
		=\inf_{\{u\in \mathcal{H}^H:~\phi_0=0,~\phi'=\big(A+D\bar{b}(\bar{X})\big)\phi+g(\bar{X})u'
			\}} \frac12\|u\|^2_{\mathcal{H}^H}
		,
	\end{align*} 
	Given the equation for	$u$, $\phi'=\big(A+D\bar{b}(\bar{X})\big)\phi+g(\bar{X})u'$, the minimal-norm solution for $u$ is given by  
	$\tilde{u}^*=\int_0^\cdot g(\bar{X}_s)^{-1}\big(\phi'_s-(A+D\bar{b}(\bar{X}_s))\phi_s\big)\mathrm{d}s$. Therefore, it follows that
	\begin{align*}
		\tilde{I}(\phi)=&\frac12\|\tilde{u}^*\|^2_{\mathcal{H}^H}=\frac12\int_0^T \|\mathbb{K}_H^{-1}\tilde{u}^*_t\|^2_1\mathrm{d}t\\
		=&\frac{1}{2c_H^2\Gamma(\frac32-H)^2}\int_0^T
		\Big\|t^{\frac12-H}g(\bar{X}_t)^{-1}\big(\phi'_t-(A+D\bar{b}(\bar{X}_t))\phi_t\big)\\
		&+(H-\frac12)t^{H-\frac12}\int_0^t\frac{t^{\frac12-H}g(\bar{X}_t)^{-1}\big(\phi'_t-(A+D\bar{b}(\bar{X}_t))\phi_t\big)}{(t-s)^{H+\frac12}}\mathrm{d}s
		\\
		&-(H-\frac12)t^{H-\frac12}\int_0^t\frac{s^{\frac12-H}g(\bar{X}_s)^{-1}\big(\phi'_s-(A+D\bar{b}(\bar{X}_s))\phi_s\big)}{(t-s)^{H+\frac12}}\mathrm{d}s\Big\|_1^2\mathrm{d}t
		,
	\end{align*}
	for all $\phi\in\mathcal{C}([0,T],V)$, such that $\phi_0=X_0$ and $g(\bar{X}_s)^{-1}\big(\phi'_s-(A+D\bar{b}(\bar{X}_s))\phi_s\big)\in\mathbb{K}'_H(L^2([0,T],V_1))$ and $\tilde{I}(\phi)=\infty$ otherwise.
	The detailed can be found in \cite{bourguin2024moderate}.
\end{remark}

According to the definitions of the MDP and the LDP, the slow variable $X^{\varepsilon,\delta}$ of the slow-fast system \eqref{eqn-1.2} satisfies the MDP if and only if the deviation component $Z^{\varepsilon,\delta}$ satisfies the LDP.
Therefore, we will now prove that the deviation component $Z^{\varepsilon,\delta}$
satisfies the LDP in $C([0,T],V)$ with respect to the speed function  $b(\varepsilon)=\frac{1}{h^2(\varepsilon)}$ and the rate function $\tilde{I}$.
This results are derived by verifying \textit{Conditions (1)} and \textit{(2)} in Lemma \ref{lemA.1}.

\begin{proof}[Proof of Condition (1) in Lemma \ref{lemA.1}]
	By Assumption \textbf{(H1)}, Lemmas \ref{lem6.3} and \ref{lem6.4},
	taking the same way as in the proof of Theorem \ref{th3.1}, it is easy to verify that 
	$\Gamma_N:=\{\tilde{\mathcal{G}}^0(u,v):(u,v)\in\mathcal{S}_N\}$ is a compact set in $C([0,T],V)$.
\end{proof}

\begin{proof}[Proof of Condition (2) in Lemma \ref{lemA.1}]
	Let $N\in\mathbb{N}$ and  $\varepsilon\in(0,1]$.
	Assume that $\{(u^\varepsilon,v^\varepsilon)\}_{\varepsilon\in(0,1)}\subset\mathcal{A}_b^N$ converges in distribution to $(u,v)\in\mathcal{S}_N$ as $\varepsilon\to 0$. 
	We will prove as $\varepsilon\to 0$, 
	\begin{align}
		\tilde{\mathcal{G}}^{\varepsilon,\delta}\Big(\frac{B^H}{h(\varepsilon)}+u^\varepsilon,\frac{W}{h(\varepsilon)}+v^\varepsilon\Big)
		\xrightarrow{\mathrm{weakly}} \tilde{\mathcal{G}}^0(u,v), \label{eqn-68}
	\end{align}
	i.e., the controlled deviation component \eqref{eqn-4.6} weakly converges to the solution of the skeleton equation \eqref{eqn-4.3}.
	To this end, we first define the auxiliary process
	\begin{align*}
		\mathrm{d}\bar{X}^{\varepsilon}_t=\big(A\bar{X}^{\varepsilon}_t+\bar{b}(\bar{X}^{\varepsilon}_t)\big)\mathrm{d}t
		+\sqrt{\varepsilon} g(\bar{X}^{\varepsilon}_t)\mathrm{d}B^H_t+\sqrt{\varepsilon} h(\varepsilon)g(\bar{X}^{\varepsilon}_t)\mathrm{d}u^\varepsilon_t,\quad\bar{X}^{\varepsilon}_0=x_0,\quad t\in[0,T].
	\end{align*}
	and divide $\tilde{Z}^{\varepsilon,\delta}=\hat{Z}^{\varepsilon,\delta}+\bar{Z}^{\varepsilon}$, where we set
	\begin{align}
		\hat{Z}^{\varepsilon,\delta}:=\frac{\bar{X}^{\varepsilon,\delta}-\bar{X}^{\varepsilon}}{\sqrt{\varepsilon}h(\varepsilon)},\quad \bar{Z}^{\varepsilon}:=\frac{\bar{X}^{\varepsilon}-\bar{X}}{\sqrt{\varepsilon}h(\varepsilon)},\label{eqn-4.16}
	\end{align}
	where the $\bar{X}$ is the solution to ODE \eqref{eq-2}. 
	Note that $\hat{Z}^{\varepsilon,\delta}_0=0$ and $\bar{Z}^{\varepsilon }_0=0$. In order to show  \eqref{eqn-68}, it is enough to verify the following two statements.
	\begin{enumerate}
		\item[\textbf{(a)}] For any $\sigma>0$,
		$$\lim_{\varepsilon \to 0} \mathbb{P}\big(\|\hat{Z}^{\varepsilon,\delta}\|_\infty>    \sigma      \big)=0.$$
		\item[\textbf{(b)}] As $\varepsilon\to 0$, 
		$$\bar{Z}^{\varepsilon }\xrightarrow{\mathrm{weakly}} \bar{Z}\quad \text{in $C([0,T],V)$}.$$
	\end{enumerate}

	Firstly, we will prove Statement \textbf{(a)}.  It is equivalent to show that for any $\sigma>0$,
	$$
	\begin{aligned}
		\lim_{\varepsilon\to0}\mathbb{P}\big(\|\bar{X}^{\varepsilon,\delta}-\bar{X}^{\varepsilon}\|_{\infty}>\sqrt{\varepsilon}h(\varepsilon)    \sigma   \big)=0.
		\label{eqn-71}
	\end{aligned}
	$$
	Now, we construct the auxiliary process as follows
	$$
	\begin{aligned}
		\begin{cases}
			\mathrm{d}\hat{X}^{\varepsilon,\delta}_t=(A\hat{X}^{\varepsilon,\delta}_t+b(\bar{X}^{\varepsilon,\delta}_{t(\Delta)},\hat{Y}^{\varepsilon,\delta}_t))\mathrm{d}t
			+\sqrt{\varepsilon}h(\varepsilon)g(\hat{X}^{\varepsilon,\delta}_t)\mathrm{d}u^\varepsilon_t
			+\sqrt\varepsilon g(\hat{X}^{\varepsilon,\delta}_t)\mathrm{d}B^H_t,\\  
			\mathrm{d}\hat{Y}^{\varepsilon,\delta}_t=\frac1\delta(A\hat{Y}^{\varepsilon,\delta}_t+F(\bar{X}^{\varepsilon,\delta}_{t(\Delta)},\hat{Y}^{\varepsilon,\delta}_t))\mathrm{d}t
			+\frac{1}{\sqrt\delta} G(\bar{X}^{\varepsilon,\delta}_{t(\Delta)},\hat{Y}^{\varepsilon,\delta}_t)\mathrm{d}W_t,\\
			\hat{X}^{\varepsilon,\delta}_0=X_0,~\hat{Y}^{\varepsilon,\delta}_0=Y_0,~t\in[0,T],
		\end{cases}
		\label{eqn-4.25}
	\end{aligned}	
	$$
	where $t(\Delta)=\left\lfloor\frac t\Delta\right\rfloor\Delta $ is the nearest breakpoint preceding $t\in[0,T]$. Without loss of generality, assume that $\Delta<1$.  By essentially the same argument as in Lemmas \ref{lem6.1} and \ref{lem6.5}, we have for every $p\geq 1$,
	\begin{align*}
		\mathbb{E}\big[\|\hat{X}^{\varepsilon,\delta}\|^p_{\alpha,\infty} \big]\leq C ,\quad\int_0^T \mathbb{E}\big[\|\hat{Y}^{\varepsilon,\delta}_t\|^2\big]\mathrm{d}t\leq C,
	\end{align*} 
	where the positive constant $C$ is independent of $\varepsilon$, $\delta$ and $\Delta$.  
	In addition, similar to the proof of Lemma \ref{lem5.2} in Section \ref{5}, it is easily shown that there exists a constant $C>0$, depending only on 
	$T$ and $N$, such that for any 
	$x,y\in V$   and  $\varepsilon,\delta\in(0,1)$, we have  
	\begin{align*}
		\int_0^T\mathbb{E}\big[\|\bar{Y}^{\varepsilon,\delta}_t-\hat{Y}^{\varepsilon,\delta}_t\|^2 \big]\mathrm{d}t\leq C\Delta+C\delta h^2(\varepsilon).	
	\end{align*}

	Take $R>0$ large enough and define the stopping time 
	$$\tau_R:=\inf\{t\geq 0: \Lambda^{0,T}_{\alpha,B^H}\geq R\}\wedge T.$$
	Then, set $A_{R,T}:=\{\Lambda^{0,T}_{\alpha,B^H}\leq R\}$.

	Similar to the proof of Theorem \ref{th3.1},
	it follows from $\zeta\in(\frac{1+2\alpha}{4},\frac12)$ and $\Delta<1$ that
	\begin{align}
		\mathbb{E}\Big[\frac{1}{\varepsilon h^2(\varepsilon)}\|\bar{X}^{\varepsilon,\delta}-\bar{X}^{\varepsilon}\|^2_{\alpha,T}  \textbf{1}_{A_{R,T}}  \Big]
		\leq C_{T,N,R}\Big(\frac{1}{\Delta h^2(\varepsilon)}\frac{\delta}{\varepsilon}+\frac{\Delta^{\alpha+\frac12}}{\varepsilon h^2(\varepsilon)} +\frac{\delta}{\varepsilon}\Big).\label{eqn-4.57}
	\end{align}
	According to the Assumption \textbf{(A7)}, it follows that as $\varepsilon \to 0$, we have $\frac\delta\varepsilon\to 0$. One can choose an appropriate $0<\Delta=\Delta(\varepsilon)<1$, such that as $\varepsilon\to 0$, we have
	$$\frac{1}{\Delta h^2(\varepsilon)}\frac{\delta}{\varepsilon}\to 0 \quad\text{and}\quad \frac{\Delta^{\alpha+\frac12}}{\varepsilon h^2(\varepsilon)}\to 0.$$
	For example, let  
	$\theta\in(\frac{2}{2\alpha+3},1)$, and when $h(\varepsilon)=\varepsilon^{-\frac{\theta}{2}}$, choose $\Delta=\varepsilon^{\theta}$.

	Therefore, for any fixed  $R>0$, we have
	$$\lim_{\varepsilon\to 0}\mathbb{E}\Big[\frac{1}{\varepsilon h^2(\varepsilon)}\|\bar{X}^{\varepsilon,\delta}-\bar{X}^{\varepsilon}\|^2_{\alpha,T}  \textbf{1}_{A_{R,T}}  \Big]=0. $$
	By Markov's inequality, it follows that
	\begin{align}
		&\mathbb{P}\Big(\frac{1}{\varepsilon h^2(\varepsilon)}\|\bar{X}^{\varepsilon,\delta}-\bar{X}^{\varepsilon}\|^2_{\alpha,T} \geq \sigma \Big)\notag\\
		\leq & \mathbb{P}(\tau_R< T)
		+\mathbb{P}\Big(\frac{1}{\varepsilon h^2(\varepsilon)}\|\bar{X}^{\varepsilon,\delta}-\bar{X}^{\varepsilon}\|^2_{\alpha,T}\geq \sigma,~ \tau_R\geq T\Big)\notag\\
		\leq &\mathbb{P}\big(\Lambda^{0,T}_{\alpha,B^H}\geq R\big)
		+\frac{1}{\sigma}\mathbb{E}\Big[\frac{1}{\varepsilon h^2(\varepsilon)}\|\bar{X}^{\varepsilon,\delta}-\bar{X}^{\varepsilon}\|^2_{\alpha,T} \textbf{1}_{A_{R,T}} \Big].\label{eqn-4.58}
	\end{align}
	According to \eqref{eqn-4.57}, for any fixed  $R>0$  and  $\sigma>0$, taking $\limsup_{\varepsilon\to 0}$, the second term on the right-hand side of   \eqref{eqn-4.58} converges to zero. Then, letting $R\to \infty$, the first term on the right-hand side of   \eqref{eqn-4.58} converges to zero. Therefore, we have
	$$\lim_{\varepsilon\to 0}\mathbb{P}\Big(\frac{1}{\varepsilon h^2(\varepsilon)}\|\bar{X}^{\varepsilon,\delta}-\bar{X}^{\varepsilon}\|^2_{\alpha,T} \geq \sigma \Big)=0.$$
	From the definition of $\hat{Z}^{\varepsilon,\delta}$ in  \eqref{eqn-4.16}, it follows that for any $\sigma>0$, we have
	$$\lim_{\varepsilon\to 0} \mathbb{P}\Big(\|\hat{Z}^{\varepsilon,\delta}\|^2_{\alpha,T} \geq \sigma\Big)=0.$$ 
	Thus, Statement \textbf{(a)} is proven.
	
	Similar to the proof of Theorem \ref{th3.1} in Section \ref{5}, it is easily shown that as 
	$\varepsilon\to 0$, if $\{(u^\varepsilon,v^\varepsilon)\}_{\varepsilon\in(0,1)}\subset\mathcal{A}_b^N$ weakly converges to $(u,v)\in\mathcal{S}_N$, then $\{\bar{Z}^\varepsilon=\tilde{\mathcal{G}}^0(u^\varepsilon,v^\varepsilon)\}_{\varepsilon\in(0,1)}$ weakly converges to $\bar{Z}$ in $C([0,T],V)$. Thus, Statement \textbf{(b)} is proven.
\end{proof}

\begin{proof}[Proof of Theorem \ref{th6.1}]
	According to the Lemma \ref{lemA.1}, the proof of the \textit{Condition (1)} and \textit{(2)}, we have that the slow component $Z^{\varepsilon,\delta}$ of system \eqref{eqn-4.2} satisfies the LDP with the speed function  $b(\varepsilon)=\frac{1}{h^2(\varepsilon)}$ and the good
	rate function $\tilde{I}:C([0,T],V)\to [0,\infty]$. The proof was completed. 
\end{proof}

\begin{remark}
	Compared to the study of the LDP, the MDP requires further investigation into the regularity of the averaged drift coefficient and the deviation component. Additionally, due to the involvement of the speed function, the estimate for the controlled deviation component becomes complex.  
\end{remark}

\section{Remarks and Generalizations}\label{7}

Here we discuss the choice of the coefficient $g$. When $g:=g(x,y)$,  the system is referred to as a fully coupled system.

In     \cite{hairer2020averaging},  the stochastic sewing lemma was employed to establish both the well-posedness and the AP for the finite-dimensional fully coupled slow-fast systems. Based on \cite{hairer2020averaging}, \cite{gailus2025large} studied the LDP in the finite-dimensional case, but the diffusion term of the fast equation is assumed to be independent of the slow variable, while the drift coefficient is uniformly bounded with respect to the slow variable. This setup allows for independent estimates of the fast variable, leading to tightness estimates for the slow variable.
In the infinite-dimensional case, \cite{li2022mild} introduced the mild increments and mild H\"{o}lder spaces, derived the mild stochastic sewing lemma, and established the AP. However, \cite{li2022mild} is limited to systems where the fast equation is independent of the slow variable, and thus cannot accommodate feedback from the slow variable. Based on \cite{li2022mild}, \cite{li2025averaging} applied techniques from \cite{hairer2020averaging} and utilized the solution of the fast equation with the slow variable fixed to decouple the fast and slow variables, thereby proving the AP. Notably, in the system studied in \cite{li2025averaging}, the fast equation only includes additive noise, which significantly simplifies the tightness estimates and convergence proofs.

In the study of the LDP for systems driven by FBM, how to define the controlled system is important when using the variational representation and weak convergence method. 
Although when $g:=g(x)$, the integral with respect to the control term could be defined with the H\"{o}lder topology. But in the fully coupled systems, it will be more difficult. The definition of the integral with respect to the control term with aid of the stochastic sewing lemma \cite{hairer2020averaging} is an open problem.
As in the finite-dimensional case, 
\cite{gailus2025large} defines the integral with respect to the control terms by using the generalized Riemann-Stieltjes integral, rather than employing the mixed Wiener-Young integral associated with the stochastic sewing lemma \cite{hairer2020averaging}. 
In this paper, we similarly approach the problem within the framework of generalized Riemann-Stieltjes integrals.

Based on the main system in this paper as follows
\begin{align*}
	\begin{cases}
		\mathrm{d}X^{\varepsilon,\delta}_t=(AX^{\varepsilon,\delta}_t+b(X^{\varepsilon,\delta}_t,Y^{\varepsilon,\delta}_t))\mathrm{d}t+\sqrt\varepsilon g(X^{\varepsilon,\delta}_t)\mathrm{d}B^H_t,\\  		\mathrm{d}Y^{\varepsilon,\delta}_t=\frac1\delta(AY^{\varepsilon,\delta}_t+F(X^{\varepsilon,\delta}_t,Y^{\varepsilon,\delta}_t))\mathrm{d}t+\frac{1}{\sqrt\delta} G(X^{\varepsilon,\delta}_t,Y^{\varepsilon,\delta}_t)\mathrm{d}W_t,\\		X^{\varepsilon,\delta}_0=X_0,~ Y^{\varepsilon,\delta}_0=Y_0,~t\in[0,T],
	\end{cases}
\end{align*}
we now investigate whether the result can be extended to the case $g:=g(x,y)$, so that this term incorporates feedback from the fast variable.	
The study of the LDP for such fully coupled systems presents significant challenges. Currently, results on the LDP for infinite-dimensional fully coupled slow-fast systems are lacking. The following points outline the main difficulties encountered in the analysis of the LDP for these systems.
\begin{enumerate}
	\item [1.]
	The well-posedness of the equations cannot be established by directly extending results from the literature.
	First, the well-posedness of infinite-dimensional systems \cite{hairer2020averaging} cannot be directly inferred from that in finite-dimensional spaces. Since our study is conducted on a Hilbert space with an operator $A$, proving the well-posedness requires more complex calculations.
	Second, the well-posedness of fully coupled infinite-dimensional systems cannot be derived from existing results in the infinite-dimensional setting.
	Specifically, in \cite{li2025averaging}, the fast equation is driven only by additive noise, which weakens the coupling between the fast and slow variables. 
	Consequently, the well-posedness of fully coupled infinite-dimensional systems remains an open problem, requiring further research.

	\item [2.] 
	Even though the variational representation and weak convergence method remain effective for the fully coupled systems, the tightness of the controlled slow component cannot be established in this case, making it challenging to obtain the weak convergence result.
	1) The Khasminskii's time discretization employed in this paper fails when extended to the infinite-dimensional fully coupled systems. 
	2) The viable pair method is ineffective in the infinite-dimensional fully coupled case, as establishing the tightness of the controlled slow variables is challenging in infinite-dimensional spaces by using Arzela-Ascoli's theorem. 
	3) Additionally, obtaining the boundedness for the controlled fully coupled slow variables presents significant challenges. Traditional techniques, as discussed in \cite{gailus2025large, li2022mild, li2025averaging} and our work, fail to provide boundedness for  these variables. 
	Therefore, we consider the techniques used in \cite{hairer2020averaging}.
	However, due to the dependence on the control terms, obtaining the corresponding Markov semigroup and invariant measure for the frozen equation is challenging. In fact, the method in \cite{hairer2020averaging}, which replaces the original controlled fast variables by solutions to the controlled fast equation with fixed slow variable, is not applicable.
	In conclusion, without imposing strong regularity conditions (such as uniform boundedness) on the drift and diffusion coefficients of the slow-fast systems, obtaining boundedness for the controlled slow variables and proving tightness require further investigation.
\end{enumerate}

\backmatter

\bmhead{Acknowledgements}

This work was partly supported by the Key International (Regional) Cooperative Research Projects of
the Natural Science Foundation (NSF) of China (Grant 12120101002) and the NSF of China (Grant 12072264).

\bmhead{Data Availability}

No datasets were generated or analysed during the current study.

\section*{Declarations}

\textbf{Conflict of interest.}
The authors declare that they have no known competing financial interests or personal relationships that could have appeared to influence the work reported in this paper.





\begin{appendices}
	
\renewcommand*{\thedefinition}{A.\arabic{definition}}
\renewcommand*{\thelemma}{A.\arabic{lemma}}
\renewcommand{\theequation}{A.\arabic{equation}}
\setcounter{equation}{0}

\section{Criterion for the LDP} \label{A}

We recall the definition of the LDP. Let $\mathcal{E}$ be a Polish space and let $\{X^\varepsilon\}_{\varepsilon\in(0,1)}$ be a collection of $\mathcal{E}$-valued stochastic processes. In fact, in this work, $\mathcal{E}=C([0,T],V)$.

\begin{definition}\label{deA.1}
	
	\begin{enumerate}
		\item [(i)]  A function $I:\mathcal{E}\to[0,\infty)$ is called a rate function on $\mathcal{E}$, if for any $M<\infty$, the
		level set $\{x\in\mathcal{E}:~I(x)\leq M\}$ is a compact subset of $\mathcal{E}$. For $A\in\mathcal{B}(\mathcal{E})$, we define $I(A)=\inf_{x\in A}I(x)$.
		\item [(ii)]   Let $I$ be a rate function on $\mathcal{E}$. A collection  $\{X^\varepsilon\}_{\varepsilon\in(0,1)}$ of $\mathcal{E}$-valued stochastic processes is said to
		satisfy the LDP in $\mathcal{E}$, as $\varepsilon\to0$, with rate function $I$ if the following two conditions hold:
		\begin{itemize}
			\item Large deviation upper bound. For each open set G in $\mathcal{E}$,
			$$\operatorname*{limsup}_{\varepsilon\to0}-\varepsilon\log\mathbb{P}(X^\varepsilon\in G)\leq I(G).$$
			\item Large deviation lower bound. For each closed set F in $\mathcal{E}$,
			$$\operatorname*{liminf}_{\varepsilon\to0}-\varepsilon\log\mathbb{P}(X^\varepsilon\in F)\geq I(F).$$
			
		\end{itemize}
		
	\end{enumerate}
\end{definition}
\begin{definition}\label{deA.2}
	Let $I$ be a rate function on the Polish space $\mathcal{E}$. A collection $\{X^\varepsilon\}_{\varepsilon \in (0,1)}$ of $\mathcal{E}$-valued random variables is said to satisfy the Laplace principle upper bound (lower bound, respectively) on $\mathcal{E}$ with rate function $I$ if for all bounded continuous functions $h :\mathcal{E}\to\mathbb{R}$,
	\begin{equation}\label{eqn-A.1} 
		\operatorname*{limsup}_{\varepsilon\to0}-\varepsilon\log\mathbb{E}\Big[\exp\Big\{-\frac{h(X^\varepsilon)}{\varepsilon}\Big\}\Big]\leq\inf_{x\in\mathcal{E}}\{h(x)+I(x)\},
	\end{equation}	
	respectively,   
	\begin{equation}\label{eqn-A.2} 
		\operatorname*{liminf}_{\varepsilon\to0}-\varepsilon\log\mathbb{E}\Big[\exp\Big\{-\frac{h(X^\varepsilon)}{\varepsilon}\Big\}\Big]\geq\inf_{x\in\mathcal{E}}\{h(x)+I(x)\}.
	\end{equation}	
	The Laplace principle (LP) is said to hold for $\{X^\varepsilon\}_{\varepsilon \in (0,1)}$ with rate function $I$ if both the Laplace upper bound and lower bound are satisfied for all bounded continuous functions $h :\mathcal{E}\to\mathbb{R}$.

\end{definition}

It is well known that a collection $\{X^\varepsilon\}_{\varepsilon \in (0,1)}$ of $\mathcal{E}$-valued stochastic processes satisfies the LDP with the rate function $I$ if and only if it satisfies the LP with the rate function $I$ (cf. \cite[Theorems 1.5 and 1.8]{Budhiraja2019}).

\begin{lemma}\label{lemA.1}
	For $0<\varepsilon<1$, let $\mathcal{G}^\varepsilon:\Omega=C_0([0,T],V\times V)\to\mathcal{E}$ be a  measurable map. We now provide a sufficient condition on the maps $\mathcal{G}^\varepsilon$ for the family
	$\{\mathcal{G}^\varepsilon(\sqrt\varepsilon B^H,\sqrt\varepsilon W)\}_{\varepsilon\in(0,1)}$ to satisfy the LDP.
	Assume that there exists a measurable map $\mathcal{G}^0:\mathcal{H}\to\mathcal{E}$ such that the following conditions hold.
	\begin{enumerate}
		\item  [(1)]   For any $N\in\mathbb{N}$, the set 
		$\Gamma_N=\{\mathcal{G}^0(u,v):(u,v)\in\mathcal{S}_N\}$
		is a compact subset of $\mathcal{E}$.
		\item  [(2)] For $N\in\mathbb{N}$, consider a family of $\mathcal{S}_N$-valued random elements $(u^\varepsilon,v^\varepsilon)\in\mathcal{A}_b^N$ on $(\Omega,\mathcal{F},\mathbb{P})$, such that $\{(u^\varepsilon,v^\varepsilon)\}_{\varepsilon \in(0,1)}$ converges in distribution to $(u,v)$ as $\varepsilon\to 0$.  Then $\{\mathcal{G}^\varepsilon(\sqrt\varepsilon B^H+u^\varepsilon,\sqrt\varepsilon W+v^\varepsilon)\}_{\varepsilon \in(0,1)} $ converges in distribution to $\mathcal{G}^0(u,v)$.

	\end{enumerate}
	
	Then the family $\{\mathcal{G}^\varepsilon(\sqrt\varepsilon B^H,\sqrt\varepsilon W)\}_{\varepsilon\in(0,1)}$ satisfies the LDP in $C([0,T],V)$ with the rate function $I$ given by
	\begin{align}
		I(x)
		=&\inf_{\{(u,v)\in\mathcal{H}:x=\mathcal{G}^0(u,v)\}}\{\frac12\|(u,v)\|^2_{\mathcal{H}}\}.
		\label{eqn-A.9} 
	\end{align}	
\end{lemma} 
\begin{proof}
	Due to the equivalence between the LDP and the LP in Polish spaces, it suffices to prove $\eqref{eqn-A.1}$ and $\eqref{eqn-A.2}$ with the rate function $I$ as defined in $\eqref{eqn-A.9}$ for all real-valued, bounded and continuous functions $h$ on $\mathcal{E}$, and to verify that $I$ is a good rate function.\\
	\textit{Proof of the upper bound $\eqref{eqn-A.1}$. }
	Without loss of generality, we assume that $$\inf_{x\in\mathcal{E}}\{h(x)+I(x)\}<\infty.$$
	Let $\xi>0$ be arbitrary. Then, there exists $x_0\in\mathcal{E}$ such that
	\begin{equation}\label{eqn-A.10} 
		h(x_0)+I(x_0)\leq\inf_{x\in\mathcal{E}}\{h(x)+I(x)\}+\frac{\xi}{2}<\infty.
	\end{equation}	
	From the definition of 
	$I$, there exists $(\tilde{u},0)\in\mathcal{H}$ such that
	$\frac{1}{2}\|\tilde{u}\|^2_{\mathcal{H}^H}\leq I(x_0)+\frac{\xi}{2}$, for $x_0=\mathcal{G}^0(\tilde{u},0).$
	Applying Lemma \ref{lem2.3} to the function $h\circ\mathcal{G}^{\varepsilon}$, we obtain
	\begin{align}
		-\varepsilon\log\mathbb{E}\Big[\exp\Big\{-\frac{h(X^{\varepsilon})}{\varepsilon}\Big\}\Big] 
		=\inf_{(u,v)\in\mathcal{A}_b}\mathbb{E}\Big[h\circ\mathcal{G}^{\varepsilon}\left(\sqrt\varepsilon B^H+u, \sqrt\varepsilon W+v \right)+\frac12\|(u,v)\|_{\mathcal{H}}^2\Big].\label{eqn-A.11} 
	\end{align}
	Then, we have
	\begin{align}
		&\limsup_{\varepsilon\to0}-\varepsilon\log\mathbb{E}\Big[\exp\Big\{-\frac{h(X^{\varepsilon})}{\varepsilon}\Big\}\Big]\notag\\ 
		\leq&\limsup_{\varepsilon\to0}\mathbb{E}\Big[h\circ\mathcal{G}^{\varepsilon}\left(\sqrt\varepsilon B^H+\tilde{u}, \sqrt\varepsilon W+0\right)+\frac12\|\tilde{u}\|_{\mathcal{H}^{H}}^2\Big] \notag\\
		\leq&\operatorname*{lim}_{\varepsilon\to0}\mathbb{E}\left[h\circ\mathcal{G}^{\varepsilon}\left(\sqrt\varepsilon B^H+\tilde{u}, \sqrt\varepsilon W\right)\right]+I(x_0)+\frac\xi2.\label{eqn-A.12} 
	\end{align}
	
	Given that $h$ is bounded and continuous, it follows from \textit{Condition (2)} in Lemma \ref{lemA.1} that, as $\varepsilon\to0$, the last term in the inequality above is equal to
	\begin{equation}\label{eqn-A.13}
		h\circ\mathcal{G}^0\left(\tilde{u},0\right)+I(x_0)+\frac\xi2=h(x_0)+I(x_0)+\frac\xi2.
	\end{equation}	
	By combining  $\eqref{eqn-A.10}$, $\eqref{eqn-A.12}$ and $\eqref{eqn-A.13}$, we deduce that
	$$\limsup_{\varepsilon\to0}-\varepsilon\log\mathbb{E}\Big[\exp\Big\{-\frac{h(X^\varepsilon)}{\varepsilon}\Big\}\Big]\leq\inf_{x\in\mathcal{E}}\{h(x)+I(x)\}+\xi.$$ 
	Since $\xi$ is arbitrary, the upper bound follows.\\
	\textit{Proof of the lower bound $\eqref{eqn-A.2}$. } Let $\xi>0$ be arbitrary. Then, for any $\varepsilon>0$, there exists $(u^\varepsilon,v^\varepsilon)\in\mathcal{A}_b$ such that
	\begin{align}
		&\inf_{(u,v)\in\mathcal{A}_b}\mathbb{E}\left[h\circ\mathcal{G}^{\varepsilon}\left(\sqrt\varepsilon B^H+u, \sqrt\varepsilon W+v\right)+\frac12\|(u,v)\|_{\mathcal{H}}^2\right]\notag\\
		\geq&\mathbb{E}\left[h\circ\mathcal{G}^{\varepsilon}\left(\sqrt\varepsilon B^H+u^{\varepsilon}, \sqrt\varepsilon W+v^\varepsilon\right)+\frac12\|(u^{\varepsilon},v^\varepsilon)\|_{\mathcal{H}}^2\right]-\xi.
		\label{eqn-A.14}
	\end{align}	
	Note that, from  $\eqref{eqn-A.11}$ and $\eqref{eqn-A.14}$, for any $\varepsilon>0$, $\mathbb{E}\left[\frac12\|(u^\varepsilon,v^\varepsilon)\|_{\mathcal{H}}^2\right]\leq2M+\xi$, where $M=\|h\|_\infty$. Now define the stopping time $\tau_{N}^{\varepsilon}=\inf\{t\geq 0:\frac{1}{2}\int_{0}^{t}(\|\dot{u}^{\varepsilon}_s\|_1^{2}+\|(v^{\varepsilon})_s^{\prime}\|_2^{2})\mathrm{d}s\geq N\}\wedge T$, where $u^\varepsilon=\mathbb{K}_H\dot{u}^\varepsilon$, and $(v^\varepsilon)^{\prime}$ denotes the time derivative of $v^\varepsilon$. We denote $(\dot{u}^{\varepsilon,N},(v^{\varepsilon,N})^{\prime})=(\dot{u}^{\varepsilon}\textbf{1}_{[0,\tau_{N}^{\varepsilon}]},(v^\varepsilon)^{\prime}\textbf{1}_{[0,\tau_{N}^{\varepsilon}]})$ and $(u^{\varepsilon,N},v^{\varepsilon,N})=(\mathbb{K}_H\dot{u}^{\varepsilon,N},\int_0^\cdot (v^{\varepsilon,N})^{\prime}_s \mathrm{d}s)$. Then, the process $(u^{\varepsilon,N},v^{\varepsilon,N})\in\mathcal{A}_b^N$ with $\frac12\|(u^{\varepsilon,N},v^{\varepsilon,N})\|_{\mathcal{H}}^2\leq N$, $\mathbb{P}$-a.s. Moreover,
	\begin{align*}
		\mathbb{P}\big((u^\varepsilon,v^\varepsilon)\neq (u^{\varepsilon,N},v^{\varepsilon,N})\big)
		\leq\mathbb{P}\Big(\frac12\|(u^\varepsilon,v^\varepsilon)\|_{\mathcal{H}}^2\geq N\Big)
		\leq\frac{2M+\xi}N.
	\end{align*}
	
	By choosing $N$ large enough such that $\frac{2M(2M+\xi)}N\leq\xi$, we obtain that $\eqref{eqn-A.14}$ holds with $(u^\varepsilon,v^\varepsilon)$ replaced with $(u^{\varepsilon,N},v^{\varepsilon,N})$ and $\xi$ with $2\xi$. Therefore, we will denote $(u^{\varepsilon,N},v^{\varepsilon,N})$ simply as $(u^\varepsilon,v^\varepsilon)$. Furthermore, by the definition of $\mathcal{A}_b^N$,
	\begin{align*}
		\sup_{\varepsilon\in(0,1)}\frac12\|(u^\varepsilon,v^\varepsilon)\|_{\mathcal{H}}^2\leq N,~~~~\mathbb{P}-a.s.
	\end{align*}	
	Since $\xi>0$ is arbitrary, in order to prove the lower bound  $\eqref{eqn-A.2}$, we now only need to show that 
	\begin{align*}
		\begin{aligned}\operatorname*{liminf}_{\varepsilon\to0}\mathbb{E}\Big[h\circ\mathcal{G}^{\varepsilon}\left(\sqrt\varepsilon B^H+u^{\varepsilon}, \sqrt\varepsilon W+v^\varepsilon\right)+\frac{1}{2}\|(u^{\varepsilon},v^{\varepsilon})\|_{\mathcal{H}}^{2}\Big]\geq\inf_{x\in\mathcal{E}}\{h(x)+I(x)\}.\end{aligned}
	\end{align*}	
	By selecting a subsequence (still relabelled by \{$\varepsilon$\}) along which $\{(u^\varepsilon,v^\varepsilon)\}_{\varepsilon \in (0,1)}$ converges to $(u,v)$ in distribution. Since $h$ is a bounded and continuous function, and the function on $\mathcal{H}$ defined
	by $(u,v)\to\frac12\|(u,v)\|_{\mathcal{H}}^2$ is lower semi-continuous with respect to the weak topology, from \textit{Condition (2)} in Lemma \ref{lemA.1} and Fatou's lemma, we have
	$$\begin{aligned}
		&\liminf_{\varepsilon\to0}\mathbb{E}\Big[h\circ\mathcal{G}^{\varepsilon}\left(\sqrt\varepsilon B^H+u^{\varepsilon},\sqrt\varepsilon W+v^{\varepsilon}\right)+\frac{1}{2}\|(u^{\varepsilon},v^{\varepsilon})\|_{\mathcal{H}}^{2}\Big]\\ \geq&\mathbb{E}\Big[h\circ\mathcal{G}^0(u,v)+\frac{1}{2}\|(u,v)\|_{\mathcal{H}}^2\Big]  \\
		\geq&\inf_{\{(x,u,v):x=\mathcal{G}^0(u,v)\}}\Big\{h(x)+\frac12\|(u,v)\|_{\mathcal{H}}^2\Big\} \\
		\geq&\inf_{x\in\mathcal{E}}\{h(x)+I(x)\}.
	\end{aligned}$$
	This completes the proof of the lower bound.\\
	\textit{Compactness of Level Sets.} To prove that $I$ is a good rate function, we need to show that, for
	any $0<M<\infty$, the level set $\{x:I(x)\leq M\}$ is compact. In order to prove this, we will show that
	$$\{x:I(x)\leq M\}=\bigcap_{n=1}^\infty\Gamma_{M+\frac1n}.$$
	According to \textit{Condition (1)} in Lemma \ref{lemA.1}, the
	set $\Gamma_{M+\frac1n}$ is compact for each $n$.
	Therefore, the compactness of the level set $\{x:I(x)\leq M\}$ follows. 
	
	Let $x\in\mathcal{E}$ such that $I(x)\leq M$. By the definition of $I(x)$, there exists $(u^n,v^n)\in\mathcal{H}$ such that $\frac12\|(u^n,v^n)\|_{\mathcal{H}}^2\leq M+\frac{1}{n}$ and $x=\mathcal{G}^0(u^n,v^n)$, This shows that $x\in\bigcap_{n=1}^\infty\Gamma_{M+\frac1n}$. Conversely, suppose $x\in\bigcap_{n=1}^\infty\Gamma_{M+\frac1n}$. Then, for each $n$, there exists $(u^n,0)\in S_{M+\frac1n}$ such that $x=\mathcal{G}^0(u^n,0)$. Thus, we have $I(x)\leq\frac12\|u^n\|_{\mathcal{H}^H}^2\leq M+\frac{1}{n}$ for all $n$. Taking the limit as $n\to \infty$, we conclude that $I(x)\leq M$. This implies that $I$ is a good rate function.
\end{proof}

\section{Some Technical Proofs} \label{B}

\subsection{Proof of Lemma \ref{lem4.1}}\label{secA.2}
For $h=\mathbb{K}_H\dot{h}\in\mathcal{H}^H$ and $t\in[0,T]$, from the integral  
\eqref{eqn-2}, we have
$$
\begin{aligned}
	h(t)=&\int_0^T K_H(t,s)\dot{h}(s)\mathrm{d}s
	=\mathbb{E}\Big[\int_0^T K_H(t,s)\mathrm{d}B_s \int_0^T \dot{h}(s)\mathrm{d}B_s\Big]
	=\mathbb{E}\Big[B^H_t\int_0^T\dot{h}(s)\mathrm{d}B_s \Big],
\end{aligned}
$$ 
where $B=\sum_{i=1}^{\infty}\sqrt{\lambda_i}e_i\beta^i_s$ is a $V$-valued BM. Consequently, for any $s,t\in[0,T]$, from H\"{o}lder's inequality, we have 
\begin{align}\label{eqn-3.9}
	\|h(t)\|_1\leq& \left(\mathbb{E}\left[\|B^H_t\|^2\right]\right)^{\frac12}\|\dot{h}\|_{L^2([0,T],V_1)},\\
	\label{eqn-3.10}
	\|h(t)-h(s)\|_1\leq&\left(\mathbb{E}\left[\|B^H_t-B^H_s\|^2\right]\right)^{\frac12}\|\dot{h}\|_{L^2([0,T],V_1)}.
\end{align}
Using the relevant conclusions of one-dimensional BM and  \eqref{eqn-2.3}, we obtain
\begin{align}
	\mathbb{E}[\|B^H_t\|^2]
	=&\mathbb{E}\bigg[\Big\|\sum_{i=1}^{\infty}\sqrt{\lambda_i}e_i \beta^{H,i}_t\Big\|^2\bigg]
	\leq C\text{tr}(Q_1)t^{2H},\label{eqn-3.32}\\
	\mathbb{E}\left[\|B^H_t-B^H_s\|^2\right]
	=&\mathbb{E}\bigg[\Big\|\sum_{i=1}^{\infty}\sqrt{\lambda_i}e_i \big(\beta^{H,i}_t-\beta^{H,i}_s\big)\Big\|^2\bigg]
	\leq C\text{tr}(Q_1)(t-s)^{2H}.\label{eqn-3.33}
\end{align}
By substituting   $\eqref{eqn-3.32}$-$\eqref{eqn-3.33}$ into  $\eqref{eqn-3.9}$-$\eqref{eqn-3.10}$, we have
$$\|h(t)\|_1\leq C(\text{tr}(Q_1))^\frac12 t^{H}\|\dot{h}\|_{L^2([0,T],V_1)} \leq C\|h\|_{\mathcal{H}^H},$$
$$\|h(t)-h(s)\|_1\leq C(\text{tr}(Q_1))^\frac12|t-s|^{H}\|\dot{h}\|_{L^2([0,T],V_1)}
\leq C|t-s|^{H} \|h\|_{\mathcal{H}^H}.$$
Then, we have
$$
\|Q_1^{-\frac12}h\|_{H-hld}=\sup_{0\leq t\leq T}\|h(t)\|_1+\sup_{0\leq s<t\leq T}\frac{\|h(t)-h(s)\|_1}{(t-s)^H}
\leq C\|h\|_{\mathcal{H}^H}
<\infty,
$$
which completes the proof.$\hfill\square$\\

\subsection{Proof of Lemma \ref{lem4.2}}\label{secA.3}

Set
$$\Lambda^{0,T}_{\alpha,u}:=\sum_{i=1}^{\infty}\sqrt{\lambda_i}\Lambda_{\alpha,Q_1^{-1/2}ue_i}^{0,T}.$$
Since for any $\kappa \in(0, \alpha)$, we have $C^{1-\alpha+\kappa}\subset W^{\alpha, \infty}_T\subset C^{1-\alpha-\kappa}$, combining Lemma \ref{lem4.1} with $\alpha>1-H$, it follows that $$\sup_{i\in\mathbb{N}}\Lambda_{\alpha, Q_1^{-1/2}ue_i}^{0, T}
\leq C\big\|Q_1^{-\frac12}u\big\|_{\alpha, 0, T}\leq C\big\|Q_1^{-\frac12}u\big\|_{H-hld}\leq C \|u\|_{\mathcal{H}^H}<\infty.$$ 
Here, the first inequality can be obtained by the definition of $\Lambda_{\alpha,Q_1^{-1/2}u e_i}^{0,T}$. Therefore, the series
$$\sum_{i=1}^{\infty}\sqrt{\lambda_i} \Lambda_{\alpha, Q_1^{-1/2}ue_i}^{0, T}
$$
is convergent, and it follows that
$\Lambda^{0, T}_{\alpha, u}
<\infty.$ 
Therefore, from  \eqref{eqn-4}, it follows that for any 
$n\in\mathbb{N}$ and $t\in[0, T]$, we have
\begin{align}
	\sum_{i=n}^{\infty}\sqrt{\lambda_i}\Big\|\int_0^tG( s)e_i\mathrm{d}(Q_1^{-\frac12}u_se_i)\Big\|\leq \sum_{i=n}^{\infty}\sqrt{\lambda_i}\Lambda_{\alpha, Q_1^{-1/2}ue_i}^{0, T} \sup_{i\in\mathbb{N}}\|Ge_i\|_{\alpha, 1}.\label{eqn-2.12}
\end{align}
As $n\to\infty$, the right-hand side of the above equation tends to 0. Therefore, the series in  (\ref{eqn-9}) converges in $V$. Furthermore, from  (\ref{eqn-9}) and (\ref{eqn-2.12}), it follows that  (\ref{eqn-2.11}) holds.

$\hfill\square$\\

\subsection{Proof of Lemma \ref{lem4.5}}\label{secA.4}
In this proof, $0<\delta\ll\varepsilon<1$ are fixed and $C$ is a positive constant depending on only $p$, $T$ and $N$, which may change from line to line. We define $(X,Y):=(\tilde{X}^{\varepsilon, \delta},\tilde{Y}^{\varepsilon, \delta})$,  and $Y^*_t:=\sup_{0\leq s\leq t}\|\tilde{Y}^{\varepsilon, \delta}_s\|$ to simplify the symbols.

For any $1\leq p<\infty$, applying H\"{o}lder's inequality and Burkholder-Davis-Gundy's inequality, from Assumption \textbf{(A2)}, we indicate that 
\begin{align}
	\mathbb{E}\left[\|Y^*_t\|^p\right]
	\leq&\|S_{\frac t \delta}Y_0\|^p
	+ C\mathbb{E}\bigg[\Big( \int_0^t  \|G(X_r,Y_r)Q_2^{\frac12}\|^2_{HS}\mathrm{d}r \Big)^{\frac p2} \Big(\int_0^t \Big\|Q_2^{-\frac12}\frac{\mathrm{d}v^\varepsilon_r}{\mathrm{d}r}\Big\|^2  \mathrm{d}r  \Big)^{\frac p2} \bigg]   \notag\\
	&+C\mathbb{E}\bigg[\Big( \int_0^t \left(1+\|X_r\|^2+\|Y_r\|^2 \right)  \mathrm{d}r  \Big)^{\frac p2}\bigg] 
	+ 
	C\mathbb{E}\bigg[\Big( \int_0^t(1+\|X_r\|^2  ) \mathrm{d}r \Big)^\frac p2\bigg] 
	\notag\\
	\leq &C\|Y_0\|^p
	+C\int_0^t\mathbb{E}\big[\|Y^*_s\|^p\big]\mathrm{d}s+C,
	\notag
\end{align}
where we used H\"{o}lder's inequality again and the conclusion of Lemma \ref{lem4.3} in the last step. Hence, according to Gronwall's inequality, one can conclude that
$$\mathbb{E}\big[\|Y^*_t\|^p\big]\leq Ce^{Ct}\leq C,~~~~t\in[0,T].$$
The proof is completed.$\hfill\square$\\

\subsection{Proof of Lemma \ref{lem4.6}}\label{secA.5}
The It\^{o}'s formula yields that
\begin{align}
	\mathbb{E}\big[\|\tilde{Y}^{\varepsilon,\delta}_t\|^2\big]=& \|Y_0\|^2+\frac2\delta\mathbb{E}\left[\int_0^t\langle\tilde{Y}^{\varepsilon,\delta}_s,A\tilde{Y}^{\varepsilon,\delta}_s\rangle \mathrm{d}s\right]
	+\frac2\delta \mathbb{E}\left[\int_0^t\langle\tilde{Y}^{\varepsilon,\delta}_s,F(\tilde{X}^{\varepsilon,\delta}_s,\tilde{Y}^{\varepsilon,\delta}_s)\rangle \mathrm{d}s\right]\notag\\
	&+\frac{2}{\sqrt{\delta\varepsilon}}\mathbb{E}\left[ \int_0^t\langle \tilde{Y}^{\varepsilon,\delta}_s,G(\tilde{X}^{\varepsilon,\delta}_s,\tilde{Y}^{\varepsilon,\delta}_s) \frac{\mathrm{d}v^\varepsilon_s}{\mathrm{d}s}  \rangle \mathrm{d}s \right]\notag\\
	&+\frac{2}{\sqrt{\delta}}\mathbb{E}\left[\int_0^t \langle \tilde{Y}^{\varepsilon,\delta}_s,G(\tilde{X}^{\varepsilon,\delta}_s,\tilde{Y}^{\varepsilon,\delta}_s)   \rangle \mathrm{d}W_s \right]
	+\frac1\delta\mathbb{E}\left[\int_0^t\|G(\tilde{X}^{\varepsilon,\delta}_s,\tilde{Y}^{\varepsilon,\delta}_s)\|^2_{HS}\mathrm{d}s\right].
	\notag
\end{align}

According to Lemmas \ref{lem4.3} and \ref{lem4.5}, we can obtain that the fifth term is a true martingale. In particular, we have $\mathbb{E}\big[\int_0^t \langle \tilde{Y}^{\varepsilon,\delta}_s,G(\tilde{X}^{\varepsilon,\delta}_s,\tilde{Y}^{\varepsilon,\delta}_s)   \rangle \mathrm{d}W_s\big]=0$. Owing to the Assumption \textbf{(A7)}, one can take $\frac\delta\varepsilon\leq\left(\bar\lambda_1+\beta_1\right)^2$. According to Assumptions \textbf{(A4)}-\textbf{(A6)} and Young's inequality, we have
\begin{align}
	\frac{\mathrm{d}\mathbb{E}\big[\|\tilde{Y}^{\varepsilon,\delta}_t\|^2\big]}{\mathrm{d}t}
	=&\frac2\delta\mathbb{E}\big[\langle \tilde{Y}^{\varepsilon,\delta}_t,A\tilde{Y}^{\varepsilon,\delta}_t     \rangle\big]
	+\frac2\delta \mathbb{E}\big[\langle\tilde{Y}^{\varepsilon,\delta}_t,F(\tilde{X}^{\varepsilon,\delta}_t,\tilde{Y}^{\varepsilon,\delta}_t)\rangle \big]\notag\\
	&+\frac1\delta\mathbb{E}\big[\|G(\tilde{X}^{\varepsilon,\delta}_t,\tilde{Y}^{\varepsilon,\delta}_t)\|^2_{HS}\big]
	+ \frac{2}{\sqrt{\delta\varepsilon}}\mathbb{E}\bigg[ \Big\langle \tilde{Y}^{\varepsilon,\delta}_t,G(\tilde{X}^{\varepsilon,\delta}_t,\tilde{Y}^{\varepsilon,\delta}_t) \frac{\mathrm{d}v^\varepsilon_t}{\mathrm{d}t}  \Big\rangle  \bigg] \notag\\
	\leq& \Big(\frac{-2\bar{\lambda}_1}{\delta}+\frac{1}{\sqrt{\delta\varepsilon}}\Big)\mathbb{E}\big[\|\tilde{Y}^{\varepsilon,\delta}_t\|^2\big]
	+\frac2\delta \mathbb{E}\big[-\beta_1\|\tilde{Y}^{\varepsilon,\delta}_t\|^2+\beta_2\big]
	\notag\\ 
	& +\frac{C_6}{\sqrt{\delta\varepsilon}} \mathbb{E}\bigg[\big(1+\|\tilde{X}^{\varepsilon,\delta}_t\|^2\big)\left\|\frac{\mathrm{d}v^\varepsilon_t}{\mathrm{d}t}\right\|^2\bigg]
	+\frac{C_6}{\delta}\mathbb{E}\big[1+\|\tilde{X}^{\varepsilon, \delta}_t\|^2\big]\notag\\
	\leq& \frac{-\bar\lambda_1-\beta_1}{\delta}\mathbb{E}\big[\|\tilde{Y}^{\varepsilon,\delta}_t\|^2\big]
	+\frac{C}{\sqrt{\delta\varepsilon}}\mathbb{E}\bigg[\big(1+\|\tilde{X}^{\varepsilon,\delta}_t\|^2\big)\left\|\frac{\mathrm{d}v^\varepsilon_t}{\mathrm{d}t}\right\|^2\bigg]
	+\frac{C}{\delta},
	\notag
\end{align}
where the last step is due to Lemma \ref{lem4.3}. 

Moreover, consider the ODE
$$
	\frac{\mathrm{d}A_t}{\mathrm{d}t}=\frac{-\bar\lambda_1-\beta_1}{\delta}A_t+\frac{C}{\sqrt{\delta\varepsilon}}\mathbb{E}\bigg[\|\tilde{X}^{\varepsilon,\delta}_t\|^2\left\|\frac{\mathrm{d}v^\varepsilon_t}{\mathrm{d}t}\right\|^2\bigg]+\frac{C}{\sqrt{\delta\varepsilon}}\mathbb{E}\bigg[\left\|\frac{\mathrm{d}v^\varepsilon_t}{\mathrm{d}t}\right\|^2\bigg]
	+\frac{C}{\delta},\quad
	A_0=\|Y_0\|^2.
$$
Hence, the solution of the above ODE has an explicit expression
$$
\begin{aligned}
	A_t=&\|Y_0\|^2e^{\frac{-\bar\lambda_1-\beta_1}{\delta}t}+\frac{C}{\sqrt{\delta\varepsilon}}\int_0^t e^{\frac{-\bar\lambda_1-\beta_1}{\delta}(t-s)}\mathbb{E}\bigg[\|\tilde{X}^{\varepsilon,\delta}_s\|^2\left\|\frac{\mathrm{d}v^\varepsilon_s}{\mathrm{d}s}\right\|^2\bigg]\mathrm{d}s\\
	&+\frac{C}{\sqrt{\delta\varepsilon}}\int_0^t e^{\frac{-\bar\lambda_1-\beta_1}{\delta}(t-s)}\mathbb{E}\bigg[\left\|\frac{\mathrm{d}v^\varepsilon_s}{\mathrm{d}s}\right\|^2\bigg]\mathrm{d}s+\frac{C}{\delta}\int_0^t e^{\frac{-\bar\lambda_1-\beta_1}{\delta}(t-s)}\mathrm{d}s.
\end{aligned}
$$
Besides, by some simple computations, we obtain
$$\frac{\mathrm{d}(\mathbb{E}[\|\tilde{Y}^{\varepsilon, \delta}_t\|^2]-A_t)}{\mathrm{d}t}\leq \frac{-\bar\lambda_1-\beta_1}{\delta} \left(\mathbb{E}\big[\|\tilde{Y}^{\varepsilon, \delta}_t\|^2\big]-A_t\right).$$
The comparison theorem implies that, for any $t\in[0,T]$,
$$
\begin{aligned}
	\mathbb{E}\big[\|\tilde{Y}^{\varepsilon, \delta}_t\|^2\big]\leq A_t.
\end{aligned}
$$
Then, according to Fubini's theorem and the fact that $(u^\varepsilon,v^\varepsilon)\in\mathcal{A}_b^N$, we have
\begin{align}
	\int_0^T \mathbb{E}\big[\|\tilde{Y}^{\varepsilon, \delta}_t\|^2\big] \mathrm{d}t
	\leq&C\|Y_0\|^2 
	+C\sqrt{\frac\delta\varepsilon}\mathbb{E}\Big[\sup_{t\in[0,T]}\|\tilde{X}^{\varepsilon,\delta}_t\|^2 \Big]
	+C\sqrt{\frac\delta\varepsilon}
	+C\notag\\
	\leq&C\mathbb{E}\Big[\sup_{t\in[0,T]}\|\tilde{X}^{\varepsilon,\delta}_t\|^2 \Big] +C,\label{eqn-4.39}
\end{align}
where
the last step owes to the fact that $\frac\delta\varepsilon\leq\left(\bar\lambda_1+\beta_1\right)^2$. 
Combining  \eqref{eqn-4.39} with Lemma \ref{lem4.3}, we finish the proof.
$\hfill\square$\\




\end{appendices}

\bibliography{sn-bibliography}

\end{document}